\documentclass[10pt]{amsart}
      \usepackage[mathscr]{eucal}
      \usepackage{amsmath,amsfonts}
\def\mapright#1{\smash{
\mathop{\rg}\limits^{#1}}}
\def\mapdown#1{\bigg\downarrow
\rlap{$\vcenter{\hbox{$\scriptstyle#1$}}$}}

\def\rg{\hbox to 30pt{\rightarrowfill}}
\def\lg{\hbox to 30pt{\leftarrowfill}}

      \parskip\smallskipamount
          \newtheorem{theorem}{Theorem}[section]

      \newtheorem{corollary}[theorem]{Corollary}
      \newtheorem{lemma}[theorem]{Lemma}
      \newtheorem{example}[theorem]{Example}

      \makeatletter
      \@addtoreset{equation}{section}
      \makeatother

      \newcommand{\BB}{{\mathbb B}}
      \newcommand{\CC}{{\mathbb C}}
      \newcommand{\NN}{{\mathbb N}}

      \newcommand{\DD}{{\mathbb D}}
      \newcommand{\RR}{{\mathbb R}}
      \newcommand{\FF}{{\mathbb F}}

      \newcommand{\cA}{{\mathcal A}}
      \newcommand{\cB}{{\mathcal B}}
      
      \newcommand{\cD}{{\mathcal D}}
      \newcommand{\cE}{{\mathcal E}}

      \newcommand{\cH}{{\mathcal H}}
      \newcommand{\cK}{{\mathcal K}}
      
      \newcommand{\cM}{{\mathcal M}}
      
      \newcommand{\cP}{{\mathcal P}}
      
      \newcommand{\cS}{{\mathcal S}}

      \newcommand{\cY}{{\mathcal Y}}

      \newcommand{\rank}{\hbox{\rm{rank}}\,}

      \newdimen\expt
      \def\boxit#1{\setbox0\hbox{$\displaystyle{#1}$}
            \hbox{\lower.4\expt
       \hbox{\lower3\expt\hbox{\lower\dp0
            \hbox{\vbox{\hrule height.4\expt
       \hbox{\vrule width.4\expt\hskip3\expt
            \vbox{\vskip3\expt\box0\vskip2\expt}%
       \hskip3\expt\vrule width.4\expt}\hrule height.4\expt}}}}}}
      
\begin{document}
       \pagestyle{myheadings}
      \markboth{ Gelu Popescu}{  Free biholomorphic classification of  noncommutative  domains }

      \title [      Free biholomorphic classification of   noncommutative  domains ]
      {            Free biholomorphic classification of  noncommutative  domains
      }
        \author{Gelu Popescu}
\date{January 10, 2008}
      \thanks{Research supported in part by an NSF grant}

      \address{Department of Mathematics, The University of Texas
      at San Antonio \\ San Antonio, TX 78249, USA}
      \email{\tt gelu.popescu@utsa.edu}

\begin{abstract} In this paper we develop a theory of free holomorphic functions on  noncommutative Reinhardt domains ${\bold
D}_f^m(\cH)\subset B(\cH)^n$, generated by positive regular free
holomorphic functions $f$ in $n$ noncommuting variables,  and  by positive integers $m\geq 1$, where
$B(\cH)$ is the algebra of all bounded linear operators on a Hilbert
space $\cH$. Noncommutative Berezin transforms are used  to study    Hardy algebras
 $H^\infty({\bf D}_{f,\text{\rm
rad}}^m)$  and  domain algebras $A({\bf D}_{f, \text{\rm rad}}^m)$ associated with ${\bold
D}_f^m(\cH)$, and   compositions of free holomorphic functions.

We obtain noncommutative   Cartan type results for
formal power series, in several noncommuting   indeterminates, which
leave invariant the nilpotent parts of the corresponding domains. As a consequence, we characterize  the set of all
   free biholomorphic functions $F:{\bf D}_f^m(\cH)\to{\bf D}_g^l(\cH)$
with $F(0)=0$.

  We show that the  free biholomorphic classification of the domains ${\bold
D}_f^m(\cH)$ is the same as the classification, up to unital
completely isometric isomorphisms having completely contractive hereditary extension,
 of the corresponding noncommutative domain algebras $A({\bf D}_{f, \text{\rm rad}}^m)$.
In particular, we prove that $\Psi:A({\bf D}_{f,\text{\rm
rad}}^1)\to A({\bf D}_{g,\text{\rm
rad}}^1)$ is a unital completely isometric  isomorphism  if and only if
 there is a free biholomorphic map  $\varphi \in Bih( {\bf D}_g^1,{\bf D}_f^1)$  such that
 $$
 \Psi(\chi)=\chi\circ \varphi,\qquad \chi\in A({\bf D}_{f,\text{\rm
rad}}^1).
 $$
 This implies that the noncommutative domains ${\bf D}_f^1(\cH)$ and  ${\bf
D}_g^1(\cH)$ are free biholomorphic equivalent if and only if   the
domain algebras $A({\bf D}_{f,\text{\rm rad}}^1)$  and $A({\bf
D}_{g,\text{\rm rad}}^1)$ are completely isometrically isomorphic.
 Using the interaction between the theory of functions in several complex variables and our noncommutative theory,
  we provide several results  concerning the free biholomorphic classification of  the noncommutative domains  ${\bold
D}_f^m(\cH)$ and  the classification, up to completely isometric  isomorphisms,
 of the  associated  noncommutative  domain (resp.~Hardy) algebras. In particular,
we  characterize the unit ball of $B(\cH)^n$ among the noncommutative
 domains ${\bf D}_f^m(\cH)$,  up to  free biholomorphisms.
 We  also obtain  characterizations for  the   unitarily implemented isomorphisms
  of noncommutative Hardy (resp.~domain) algebras in terms of free biholomorphic
  functions between the corresponding noncommutative domains.

\end{abstract}

      \maketitle

\bigskip

\section*{Contents}
{\it

\quad Introduction

\begin{enumerate}
   \item[1.]   Free holomorphic   functions on noncommutative Reinhardt  domains
   \item[2.] Noncommutative domain algebras and Hardy algebras
 \item[3.]  Compositions of free holomorphic functions
\item[4.]  Free biholomorphic functions and noncommutative Cartan type results
\item[5.]  Free biholomorphic classification of   noncommutative domains
\item[6.]  Isomorphisms of noncommutative Hardy algebras
   \end{enumerate}

\quad References

}

\bigskip

\section*{Introduction}

In recent years,  we have developed  a {\it noncommutative analytic function theory} on the open unit ball
$$
[B(\cH)^n]_1:=\{(X_1,\ldots, X_n)\in B(\cH)^n: \ \|X_1X_1^*+\cdots +X_nX_n^*\|<1\},
$$
where
$B(\cH)$ is the algebra of all bounded linear operators on a Hilbert
space $\cH$. We showed that
several classical
 results from complex analysis and  hyperbolic geometry  have
 free analogues in
 this  noncommutative multivariable setting (see
 \cite{Po-holomorphic}, \cite{Po-free-hol-interp},
  \cite{Po-pluri-maj},
 \cite{Po-pluriharmonic}, \cite{Po-hyperbolic}, \cite{Po-automorphism},  \cite{Po-holomorphic2}).  Related to our work, we
 mention  the papers \cite{DP2}, \cite{HKMS},   \cite{MuSo2},
 \cite{MuSo3},  \cite{Vo1}, and \cite{Vo2}, where several aspects of the theory
 of noncommutative analytic functions are considered in various
 settings.
 Some of the results  concerning  the noncommutative analytic function theory on the open unit ball $[B(\cH)^n]_1$ were extended   (see Chapter 2 from \cite{Po-domains})
  to free holomorphic functions on the interiors  of  noncommutative {\it ball-like  domains}
 $$
 {\bf D}^1_{p}(\cH):=\{X:=(X_1,\ldots, X_n)\in B(\cH)^n: \ \Phi_{p,X}(I)\leq I\},
 $$
 where $\Phi_{p,X}(Y):=\sum_{k=1}^\infty\sum_{|\alpha|=k} a_\alpha X_\alpha YX_\alpha^*$ for
  $Y\in B(\cH)$, and  $p=\sum_{k=1}^\infty\sum_{|\alpha|=k} a_\alpha X_\alpha$ is a
   positive regular noncommutative polynomial, i.e., its coefficients are positive
    scalars and $a_\alpha>0$ if  $\alpha\in \FF_n^+$ with $|\alpha|=1$.
Here,  $\FF_n^+$ is the unital free semigroup on $n$ generators
$g_1,\ldots, g_n$ and the identity $g_0$.  The length of $\alpha\in
\FF_n^+$ is defined by $|\alpha|:=0$ if $\alpha=g_0$  and
$|\alpha|:=k$ if
 $\alpha=g_{i_1}\cdots g_{i_k}$, where $i_1,\ldots, i_k\in \{1,\ldots, n\}$.
If $X:=(X_1,\ldots, X_n)\in B(\cH)^n$   we
denote $X_\alpha:= X_{i_1}\cdots X_{i_k}$  and $X_{g_0}:=I_\cH$.

In \cite{Po-Berezin}, we studied  more general noncommutative domains
$$
{\bold D}_p^m(\cH):=\left\{X:= (X_1,\ldots, X_n)\in B(\cH)^n: \
 (id-\Phi_{p,X})^k(I)\geq 0 \ \text{ for } \ 1\leq k\leq m\right\},
$$
where $m\geq 1$.
We showed that each  such a domain has a  universal model $(W_1,\ldots, W_n)$ of
{\it weighted creation operators} acting on the full Fock space $F^2(H_n)$  with $n$
generators. The  study of  ${\bf D}_p^m(\cH)$ is close related to
the study of the weighted shifts $W_1,\ldots,W_n$, their joint
invariant subspaces, and the representations of the algebras they
generate: the domain algebra $\cA_n({\bf D}_p^m)$, the Hardy algebra
$F_n^\infty({\bf D}_p^m)$, and the $C^*$-algebra $C^*(W_1,\ldots,
W_n)$.

The present paper is an attempt  to develop a noncommutative analytic function theory on the {\it radial parts}
 $$
 {\bf D}_{p,\text{\rm rad}}^m(\cH):=\bigcup_{0\leq r<1} r{\bf D}_{p}^m(\cH)
 $$
  of noncommutative {\it starlike   domains}  ${\bf D}_{p}^m(\cH)$, i.e., satisfying the condition $r{\bf D}_{p}^m(\cH)\subset {\bf D}_{p}^m(\cH)$ for any  $r\in [0,1)$. We mention that   ${\bold D}_p^1(\cH)$ is always a starlike domain  and ${\bf D}_{p,\text{\rm rad}}^1(\cH)$  is equal to the interior of ${\bold D}_p^1(\cH)$.
  We also remark that
if $q=X_1+\cdots+X_n$ and $m\geq 1$,  then ${\bold D}_q^m(\cH)$  is a starlike domain which coincides with  the set of all row contractions $(X_1,\ldots, X_n)\in [B(\cH)^n]_1^-$
satisfying the positivity condition
$$
\sum_{k=0}^m (-1)^k \left(\begin{matrix} m\\k\end{matrix}\right)\sum_{|\alpha|=k} X_\alpha X_\alpha^*\geq 0.
$$
The elements of  the domain ${\bold D}_q^m(\cH)$  can be seen as multivariable noncommutative analogues of Agler's $m$-hypercontractions \cite{Ag2}. In this case,  we also have that ${\bf D}_{q,\text{\rm rad}}^m(\cH)$  is the interior of ${\bold D}_q^m(\cH)$.

 We introduce   the class of free holomorphic functions
on the noncommutative {\it radial domain}  ${\bf D}_{p,\text{\rm rad}}^m$, $m\geq 1$.  A formal power series $F:=\sum_{\alpha\in \FF_n^+} c_\alpha Z_\alpha$,
 $c_{\alpha}\in \CC$,  in $n$ noncommuting indeterminates $Z_1,\ldots, Z_n$ is called
  free holomorphic function  (with
scalar coefficients) on
the noncommutative domain ${\bf D}_{p,\text{\rm rad}}^m $   if its representation on any Hilbert space
$\cH$, i.e.,  the  map
 $F:{\bf D}_{p,\text{\rm rad}}^m(\cH)\to B(\cH)$
 given  by
$$
F(X_1,\ldots, X_n):=\sum_{k=0}^\infty \sum_{|\alpha|=k}
c_\alpha X_\alpha, \qquad  (X_1,\ldots, X_n)\in {\bf
D}_{p,\text{\rm rad}}^m(\cH),
$$
is well-defined  in the operator norm topology. The  map $F$ is
called
  {\it free
holomorphic function} on ${\bf D}_{p,\text{\rm rad}}^m(\cH)$. We
remark that if $m=1$ and $p=X_1+\cdots +X_n$, then our definition coincides  with
the one for free holomorphic functions (see
\cite{Po-holomorphic}) on the open unit ball  of $B(\cH)^n$.

There is  an important connection between the theory of free
holomorphic functions on noncommutative domains ${\bold
D}_{p,\text{\rm rad}}^m(\cH)$ and the theory of holomorphic
functions on the corresponding domains  $Int({\bold
D}_p^m(\CC^s))\subset \CC^{ns^2}$,
 $s\in \NN$, where {\it Int} denotes the interior. This is due to the fact that the
representation of any free holomorphic function on a finite
dimensional  space ($\cH=\CC^s)$ is  a holomorphic function  in the
classical sense. This  connection will allow us to use  the theory
of functions in several complex variables  (\cite{Kr}, \cite{Su},
\cite{Th}) and our noncommutative theory in our attempt  to classify
the  domains ${\bold D}_p^m(\cH)$ and  the corresponding
noncommutative  domain (resp. Hardy) algebras (see   Theorem
\ref{rep-finite3}, Theorem \ref{Thullen}, and  Corollary
\ref{Thullen2}). As  in the case of biholomorphic maps between domains in
$\CC^n$, $n>1$, (see \cite{IK}, \cite{Su}), we expect a great deal of rigidity for free biholomorphic functions between noncommutative domains.

In Section 1,  we recall  basic
facts concerning the noncommutative Berezin transforms (\cite{Po-Berezin}) associated
with  noncommutative domains ${\bold D}_p^m(\cH)$, $m\geq 1$, and
characterize the algebra $Hol_\cE({\bf
D}_{p,\text{\rm rad}}^m)$
 of all free holomorphic functions on  ${\bf D}_{p,\text{\rm
rad}}^m$, with
 operator-valued coefficients in $B(\cE)$.

In Section 2,
we introduce the  Hardy algebra
$H^\infty({\bf D}_{p,\text{\rm rad}}^m)$ and  the domain algebra
$A({\bf D}_{p,\text{\rm rad}}^m)$ associated with the noncommutative domain ${\bf D}_{p,\text{\rm rad}}^m$, $m\geq1$.
Let $H^\infty({\bf D}_{p,\text{\rm
rad}}^m)$  denote the set of  all elements $\varphi$ in $Hol({\bf
D}_{p,\text{\rm rad}}^m)$     such that
$\|\varphi\|_\infty:= \sup \|\varphi(X)\|<\infty,
$
where the supremum is taken over all $n$-tuples $X\in {\bf D}_{p,\text{\rm rad}}^m(\cH)$ and $\cH$ is a separable
infinite dimensional Hilbert space. We  denote by  $A({\bf
D}_{p,\text{\rm rad}}^m)$   the set of all   $\varphi$
  in $Hol({\bf D}_{p,\text{\rm
rad}}^m)$   such that the mapping
${\bf
D}_{p,\text{\rm rad}}^m(\cH)\ni X\mapsto
\varphi(X)\in B(\cH)$
 has a continuous extension to  $[{\bf
D}_{p,\text{\rm rad}}^m(\cH)]^-={\bf
D}_{p}^m(\cH)$.
It turns out that  $H^\infty({\bf
D}_{p,\text{\rm rad}}^m)$ and  $A({\bf
D}_{p,\text{\rm rad}}^m)$ are   Banach algebras under pointwise
multiplication and the norm $\|\cdot \|_\infty$.
Using noncommutative Berezin transforms (\cite{Po-Berezin}), we identify the  noncommutative algebra
$F_n^\infty({\bf D}^m_p)$ and  the  algebra $\cA_n({\bf D}^m_p)$ with the  Hardy algebra
$H^\infty({\bf
D}_{p,\text{\rm rad}}^m)$ and  the domain algebra $A({\bf D}_{p,\text{\rm rad}}^m)$,
 respectively (see Theorem \ref{f-infty} and Theorem \ref{A-infty}). We recall
  (\cite{Po-Berezin}) that the Banach algebra $F_n^\infty({\bf D}^m_p)$ is
  the sequential WOT-(resp. w*-)closure of all polynomials in $W_1,\ldots, W_n,$
  and the identity, while
$\cA_n({\bf D}^m_p)$ is the norm closed algebra generated by the
same operators. Hardy algebras  of  bounded free holomorphic functions with
 operator-valued coefficients are also discussed.

In Section 3, we present several results concerning the composition
of free holomorphic functions on noncommutative domains ${\bf
D}_{p,\text{\rm rad}}^m$. These results are used, throughout this
paper,  to study free biholomorphic functions. Let $p$ and $g$ be
positive regular noncommutative polynomials with $n$ and $q$
indeterminates, respectively,  and let $m,l\geq 1$. A map $F:{\bf
D}_p^m(\cH)\to {\bf D}_g^l(\cH)$ is called free biholomorphic
function  if $F$ is a homeomorphism  in the operator norm topology
and $F|_{{\bf D}_{p,\text{\rm rad}}^m(\cH)}$ and $F^{-1}|_{{\bf
D}_{g,\text{\rm rad}}^l(\cH)}$ are   free homolorphic functions. In this case,  the domains
${\bf
D}_p^m(\cH)$ and  ${\bf D}_g^l(\cH)$ are called free biholomorphic equivalent.

 In  Section 4, we obtain  a noncommutative   Cartan (\cite{Ca}) type result
  (see Theorem \ref{Cartan3}) characterizing  the set $Bih_0({\bf D}_p^m, {\bf D}_g^l)$ of all
   free biholomorphic functions $F:{\bf D}_p^m(\cH)\to{\bf D}_g^l(\cH)$
with $F(0)=0$. Several consequences  concerning the biholomorphic
classification of  noncommutative domains are obtained. In
particular, we characterize
 the unit ball of $B(\cH)^n$ among the noncommutative domains ${\bf D}_p^m(\cH)$,
  up to  free biholomorphisms.

In Section 5, we present  the main result of this paper  (see Theorem \ref{auto-disk})
 which   shows that the  free biholomorphic classification of the noncommutative starlike
 domains ${\bold
D}_p^m(\cH)\subset B(\cH)^n$, generated by positive regular noncommutative polynomials $p$
 and  positive integers $m\geq 1$, is the same as the classification, up to unital
  completely isometric isomorphisms having completely contractive hereditary extension, of the corresponding noncommutative domain
  algebras $\cA_n({\bf D}_p^m)\simeq
 A({\bf D}_{p, \text{\rm rad}}^m)$.
We show that   $\Psi:A({\bf D}_{p,\text{\rm
rad}}^m)\to A({\bf D}_{g,\text{\rm
rad}}^l)$ is a unital completely isometric  isomorphism   having  completely
contractive hereditary extension if and only if
 there is a free biholomorphic map  $\varphi \in Bih( {\bf D}_g^l,{\bf D}_p^m)$
 such that
 $$
 \Psi(\chi)=\chi\circ \varphi,\qquad \chi\in A({\bf D}_{p,\text{\rm
rad}}^m).
 $$
 In particular, we prove that  the noncommutative domains ${\bf D}_p^1(\cH)$ and  ${\bf
D}_g^1(\cH)$ are free biholomorphic equivalent if and only if   the
domain algebras  $\cA_n({\bf D}_p^1)$  and $\cA_q({\bf D}_g^1)$ are
completely isometrically isomorphic.
 We mention that if $m=l=1$ and  $p=g=X_1+\cdots + X_n$,  we recover
  (with a different proof) the corresponding result from \cite{Po-automorphism}.
  In the particular case when $n=m=l=1$ and $p=g=X$, we obtain the  well-known  result
  (\cite{H}) that the automorphisms of the disc algebra
$A(\DD)$, of all bounded analytic functions on the open unit disc
$\DD:=\{z\in \CC:\ |z|<1\}$ with continuous extensions to the closed
disc, are the maps
$
\Theta_\tau (\varphi)=\varphi\circ\tau$, \  $\varphi\in A (\DD),
$
 where
$\tau$ is a conformal automorphism of the unit disc $\DD$.

 Combining the  main result of this paper with the Cartan type results   from Section 4, we
  obtain a characterization of    all  isomorphisms $\Psi:A({\bf D}_{p,\text{\rm
rad}}^m)\to A({\bf D}_{g,\text{\rm
rad}}^l)$  with $\varphi(0)=0$.
 We mention that, in the particular case when $m=l=1$, our result strengthens and provides
 a converse of  the main result of   Arias and Latr\' emoli\` ere  (see  Theorem 3.18 from
 \cite{ArL}), who  proved
   that if  there is a completely isometric isomorphism between two noncommutative domain
   algebras  $\cA_n({\bf D}_p^1)$ and  $\cA_n({\bf D}_g^1)$ whose dual map fixes the origin,
   then the algebras are related by  a linear  relation of  their generators.

   In Section 6,
 we  characterize the
   unitarily implemented   isomorphisms  between  noncommutative   Hardy algebras
 $H^\infty({\bf D}_{p,\text{\rm
rad}}^m)$ and  $H^\infty({\bf D}_{g,\text{\rm rad}}^l)$ in terms of
free biholomorphic
  functions between the corresponding noncommutative domains. Similar
results are deduced for noncommutative domain algebras. In
particular, when $m=1$,  we prove that the group $Aut_{u}(
\cA_n({\bf D}_p^1)))$ of all unitarily implemented automorphisms of
the noncommutative domain algebra  $\cA_n({\bf D}_p^1)$ is
isomorphic to the group $Aut_w({\bf D}_{p}^1)$ of all free
biholomorphic functions $\varphi\in Aut({\bf D}_{p}^1)$ with the
property that the model boundary function $\widetilde\varphi$ is a
pure $n$-tuple, i.e., $\text{\rm SOT-}\lim_{k\to
\infty}\Phi_{p,\widetilde \varphi}^k(I)=0$, with
  $\text{\rank}\,[I-\Phi_{p,\widetilde \varphi}(I)]=1$ and the
  characteristic function $\Theta_{\widetilde\varphi}=0$.

 We should mention that the results of this paper are
presented in a
 more general setting, namely, when the polynomials $p$ and $g$ are
  replaced by positive regular free holomorphic functions on open balls
  $[B(\cH)^n]_\gamma$, $\gamma>0$.

\bigskip

\section{Free holomorphic   functions on noncommutative Reinhardt  domains}

We begin by   setting up the notation and  by recalling some basic
facts concerning the noncommutative Berezin transforms associated
with  noncommutative domains ${\bold D}_f^m(\cH)$, $m\geq 1$. We
characterize the algebra
 of all free holomorphic functions on  ${\bf D}_{f,\text{\rm
rad}}^m$, with
 operator-valued coefficients, and discuss the connection between the theory of free holomorphic
functions on noncommutative domains ${\bf D}_{f,\text{\rm rad}}^m$
  and the theory of holomorphic functions on domains in
$\CC^d$.

Let $f:= \sum_{\alpha\in \FF_n^+} a_\alpha X_\alpha$, \ $a_\alpha\in
\CC$,  be a free holomorphic function on a ball $[B(\cH)^n]_\gamma$
for some $\gamma>0$.
  As shown
in \cite{Po-holomorphic}, this condition is equivalent  to
\begin{equation*}
\limsup_{k\to\infty} \left( \sum_{|\alpha|=k}
|a_\alpha|^2\right)^{1/2k}<\infty.
\end{equation*}
Throughout this paper, we
  assume that  $a_\alpha\geq 0$ for any $\alpha\in \FF_n^+$, \ $a_{g_0}=0$,
 \ and  $a_{g_i}>0$, $i=1,\ldots, n$.
 A function $f$ satisfying  all these conditions on the coefficients is
 called a {\it positive regular free holomorphic function}.
 We denote by $id$ the identity map acting on
 the algebra of all bounded linear operators an a Hilbert space.
Given $m,n\in \NN:=\{1,2,\ldots\}$ and  a positive regular free
holomorphic
  function $f:=\sum_{\alpha\in \FF_n^+, |\alpha|\geq 1} a_\alpha X_\alpha$,
  we define  the noncommutative domain
$$
{\bold D}_f^m(\cH):=\left\{X:= (X_1,\ldots, X_n)\in B(\cH)^n: \
 (id-\Phi_{f,X})^k(I)\geq 0  \ \text{ for } \ 1\leq k\leq m\right\},
$$
where $\Phi_{f,X}:B(\cH)\to B(\cH)$ is defined by $
\Phi_{f,X}(Y):=\sum_{k=1}^\infty\sum_{|\alpha|=k} a_\alpha X_\alpha
YX_\alpha^*$, \  $Y\in B(\cH), $ and the convergence is in the week
operator topology.  ${\bold D}_f^m(\cH)$ can be seen as a
noncommutative Reinhardt domain, i.e., $(e^{i\theta_1}X_1,\ldots,
e^{i\theta_n}X_n)\in {\bold D}_f^m(\cH)$ for any $(X_1,\ldots,
X_n)\in {\bold D}_f^m(\cH)$ and $\theta_1,\ldots, \theta_n\in \RR$.
We say that $\Phi_{f,X}$ is power bounded if there exists a constant
$M>0$ such that $\|\Phi_{f,X}^k\|\leq M$ for any $k\in \NN$. In
particular, this is the case if $\Phi_{f,X}(I)\leq I$.
  We recall (see
\cite{Po-Berezin}, Lemma 1.4) the following result. If $\Phi_{f,X}$
is  a power bounded positive map and
 $m\in \NN$,
then
$$(id-\Phi_{f,X})^m(I)\geq 0\quad \text{ if and only if }\quad  (id-\Phi_{f,X})^k(I)\geq
0, \quad k=1,2,\ldots,m.
$$
Note that if $(X_1,\ldots, X_n)\in {\bold D}_f^m(\cH)$ then
\begin{equation}\label{bound}
\|X_1X_1^*+\cdots +X_nX_n^*\|\leq \frac{1}{\min\{ {a_\alpha}: \
|\alpha|=1\}}.
\end{equation}
Indeed, since $\Phi_{f,X}(I)\leq I$ and $ a_\alpha\geq 0$, we deduce
that
$$
\min\{ {a_\alpha}: \ |\alpha|=1\}(X_1X_1^*+\cdots +X_nX_n^*)\leq
\sum_{|\alpha|\geq 1} a_\alpha X_\alpha X_\alpha^*\leq I.
$$

 Throughout this paper, unless otherwise specified,  we assume that $\cH$ is a separable
infinite dimensional Hilbert space and  ${\bf D}_f^m (\cH)$  is a
closed (in the operator norm topology)  starlike domain, i.e.,
$(rX_1,\ldots, rX_n)\in {\bf D}_f^m (\cH)$ for any
    $n$-tuple $ (X_1,\ldots,
X_n)\in {\bf D}_f^m(\cH)$ and $r\in [0,1)$. In particular, ${\bf
D}_f^m(\cH)$ is path connected to $0$. Note that if $m=1$, then
${\bf D}_f^1(\cH)$ is always a closed starlike domain. This case was
extensively studied in \cite{Po-domains}. When $m\geq 2$, we  point
out  the following class of noncommutative starlike domains.

\begin{example}$($\cite{Po-Berezin}$)$\label{ex-radial}
If $p(X_1,\ldots, X_n):=a_1X_1+\cdots +a_n X_n$, $a_i>0$, then ${\bf
D}_p^m(\cH)$, $m\in \NN$, is a closed starlike domain.
\end{example}

Now, we recall ( \cite{Po-Berezin}, \cite{Po-domains}) a few facts
concerning the noncommutative Berezin transforms associated with
noncommutative domains ${\bold D}_f^m(\cH)$, $m\geq 1$. Let $H_n$ be
an $n$-dimensional complex  Hilbert space with orthonormal
      basis
      $e_1$, $e_2$, $\dots,e_n$, where $n\in\NN$, or $n=\infty$.
       Consider the full Fock space  of $H_n$ defined by
      $$F^2(H_n):=\CC1\oplus \bigoplus_{k\geq 1} H_n^{\otimes k},$$
      where  $H_n^{\otimes k}$ is the (Hilbert)
      tensor product of $k$ copies of $H_n$. We denote
$e_\alpha:= e_{i_1}\otimes\cdots \otimes  e_{i_k}$ if
$\alpha=g_{i_1}\cdots g_{i_k}\in \FF_n^+$, where $i_1,\ldots, i_k\in
\{1,\ldots, n\}$, and $e_{g_0}:=1$. Note that
$\{e_\alpha\}_{\alpha\in \FF_n^+}$ is an orthonormal basis for
$F^2(H_n)$. Define the left   creation
      operators  $S_i$, $i=1,\ldots,n$, acting on $F^2(H_n)$  by
      setting
      $$
       S_i\varphi:=e_i\otimes\varphi, \qquad  \varphi\in F^2(H_n).
      $$
Let  $D_i:F^2(H_n)\to F^2(H_n)$, $i=1,\ldots, n$,  be   the diagonal
operators  given by
$$
D_ie_\alpha:=\sqrt{\frac{b_\alpha^{(m)}}{b_{g_i \alpha}^{(m)}}}
e_\alpha,\qquad
 \alpha\in \FF_n^+,
$$
where
\begin{equation}
\label{b-al} b_{g_0}^{(m)}:=1\quad \text{ and } \quad
 b_\alpha^{(m)}:= \sum_{j=1}^{|\alpha|}
\sum_{{\gamma_1\cdots \gamma_j=\alpha }\atop {|\gamma_1|\geq
1,\ldots, |\gamma_j|\geq 1}} a_{\gamma_1}\cdots a_{\gamma_j}
\left(\begin{matrix} j+m-1\\m-1
\end{matrix}\right)  \qquad
\text{ if } \ \alpha\in \FF_n^+, |\alpha|\geq 1.
\end{equation}
Since  $a_{g_i}>0$, Lemma 1.2 from \cite{Po-Berezin} implies
$$
b_{g_i\alpha}^{(m)}\geq
\sum\limits_{{\gamma\sigma=g_i\alpha}\atop{\sigma\in \FF_n^+,
 |\gamma|\geq 1}}
b_\sigma^{(m)} a_\gamma\geq a_{g_i} b_\alpha^{(m)}, \qquad
i=1,\ldots,n,$$
 and, consequently,
$$
\|D_i\|=\sup_{\alpha\in \FF_n^+} \sqrt{\frac{b_\alpha^{(m)}}{b_{g_i
\alpha}^{(m)}}}\leq \frac{1}{\sqrt{a_{g_i}}}, \qquad i=1,\ldots,n.
$$

We define the {\it weighted left creation  operators}
$W_i:F^2(H_n)\to F^2(H_n)$, $i=1,\ldots, n$,  associated with the
noncommutative domain  ${\bold D}_f^m $  by setting $W_i:=S_iD_i$,
where
 $S_1,\ldots, S_n$ are the left creation operators on the full
 Fock space $F^2(H_n)$.
 Note that
\begin{equation} \label{w-shift}
W_i e_\alpha=\frac {\sqrt{b_\alpha^{(m)}}}{\sqrt{b_{g_i
\alpha}^{(m)}}} e_{g_i \alpha}, \qquad \alpha\in \FF_n^+,
\end{equation}
where the coefficients  $b_\alpha^{(m)}$, $\alpha\in \FF_n^+$, are
given by relation \eqref{b-al}.
 When necessary, we  denote the weighted left
creation operators $(W_1,\ldots, W_n)$ associated with ${\bf
D}_f^{m}$  by $(W_1^{(f)},\ldots, W_n^{(f)})$. Using relation
\eqref{w-shift},  one can easily see that
\begin{equation}\label{WbWb}
W_\beta e_\gamma= \frac {\sqrt{b_\gamma^{(m)}}}{\sqrt{b_{\beta
\gamma}^{(m)}}} e_{\beta \gamma} \quad \text{ and }\quad W_\beta^*
e_\alpha =\begin{cases} \frac
{\sqrt{b_\gamma^{(m)}}}{\sqrt{b_{\alpha}^{(m)}}}e_\gamma& \text{ if
}
\alpha=\beta\gamma \\
0& \text{ otherwise }
\end{cases}
\end{equation}
 for any $\alpha, \beta \in \FF_n^+$.
According to Theorem 1.3 from \cite{Po-Berezin}, the weighted left
creation
  operators $W_1,\ldots, W_n$     associated with ${\bold D}_f^m $  have the following properties:
 \begin{enumerate}
 \item[(i)] $\sum_{k=1}^\infty\sum_{|\beta|=k} a_\beta W_\beta W_\beta^*\leq I$, where the
convergence is in the strong operator topology;
 \item[(ii)]
 $\left(id-\Phi_{f,W}\right)^{m}(I)=P_\CC$, where $P_\CC$ is the
 orthogonal projection from $F^2(H_n)$ onto $\CC 1\subset F^2(H_n)$, and
  $\lim\limits_{p\to\infty} \Phi^p_{f,W}(I)=0$ in the strong operator
 topology.
 \end{enumerate}
Consequently, we deduce that   $(W_1,\ldots, W_n)\in {\bf D}_f^m
(F^2(H_n))$.   In \cite{Po-Berezin}, we proved  that the $n$-tuple
$(W_1,\ldots, W_n)$ plays the role of universal model for the
noncommutative domain ${\bf D}_f^m $. We also introduced
(\cite{Po-domains}, \cite{Po-Berezin}) the domain algebra
$\cA_n({\bf D}^m_f)$ associated with the noncommutative domain ${\bf
D}^m_f$ to be the norm closure of all polynomials in $W_1,\ldots,
W_n$, and the identity.

Following \cite{Po-von}, \cite{Po-funct},  the noncommutative Hardy
algebra $F_n^\infty({\bf D}^m_f)$ was defined  (\cite{Po-domains},
\cite{Po-Berezin}) as follows. Let $\varphi(W_1,\ldots,
W_n)=\sum_{\beta\in \FF_n^+} c_\beta W_\beta $ be   a formal Fourier
representation  such   that $\sum_{\beta\in \FF_n^+} |c_\beta|^2
\frac{1}{b_\beta^{(m)}}<\infty$, where the coefficients
$b_\beta^{(m)}$, $\beta\in \FF_n^+$, are given by relation
 \eqref{b-al}. Then
$\sum_{\beta\in \FF_n^+} c_\beta W_\beta (p)\in F^2(H_n)$ for any
$p\in \cP$, where $\cP$ is the set of all polynomial in $F^2(H_n)$.
If, in addition,
$$
\sup_{p\in\cP, \|p\|\leq 1} \left\|\sum\limits_{\beta\in \FF_n^+}
c_\beta W_\beta (p)\right\|<\infty,
$$
  then there
is a unique bounded operator acting on $F^2(H_n)$, which we denote
also  by $\varphi(W_1,\ldots, W_n)$, such that
$$
\varphi(W_1,\ldots, W_n)p=\sum\limits_{\beta\in \FF_n^+} c_\beta
W_\beta (p),\qquad   p\in \cP.
$$
The set of all operators $\varphi(W_1,\ldots, W_n)\in B(F^2(H_n))$
satisfying the above-mentioned properties is denoted by
$F_n^\infty({\bf D}^m_f)$. We proved in \cite{Po-Berezin} that the
Banach algebra  $F_n^\infty({\bf D}^m_f)\subset B(F^2(H_n))$  is the
sequential SOT-(resp.~WOT-, $w^*$-) closure of all polynomials in
$W_1,\ldots, W_n$, and the identity.

In the particular case when $m=1$ and $q=X_1+\cdots + X_n$, the
algebras $\cA_n({\bf D}^1_q)$  and $F_n^\infty({\bf D}^1_q)$
coincide with the  noncommutative disc algebra $\cA_n$  and the
noncommutative analytic Toeplitz algebra $F_n^\infty$, respectively,
which  were introduced in \cite{Po-von}(see also \cite{Po-funct}, \cite{Po-poisson}) in connection with a
noncommutative von Neumann  type inequality and have been
extensively studied in the last two decades. For the classical von Neumann inequality we refer  the reader to \cite{von}, \cite{SzF-book},  \cite{Pi},  and \cite{Pa-book}.

 The joint spectral radius of an $n$-tuple  $T:=(T_1,\ldots, T_n)\in {\bf
D}_f^m(\cH)$ is defined by
$$
r_f(T_1,\ldots, T_n):=\lim_{k\to\infty}\|\Phi_{f,T}^k(I)\|^{1/2k},
$$
where $\Phi_{f,X}:B(\cH)\to B(\cH)$ is given by $
\Phi_{f,X}(Y):=\sum_{k=1}^\infty\sum_{|\alpha|=k} a_\alpha X_\alpha
YX_\alpha^*$, \  $Y\in B(\cH)$. For each $T:=(T_1,\ldots, T_n)\in
{\bf D}_f^m(\cH)$ with $r_f(T_1,\ldots, T_n)<1$, we introduced (see
\cite{Po-Berezin})  the {\it noncommutative Berezin transform}  at
$T$ as the map  ${\bf B}_T:B(F^2(H_n)) \to B(\cH)$ defined by

\begin{equation*}
\left<{\bf B}_T[g]x,y\right>:= \left<\left( I-\sum_{|\alpha|\geq 1}
\overline{a}_{\tilde\alpha} \Lambda_\alpha^* \otimes
T_{\tilde\alpha} \right)^{-m}
 (g\otimes \Delta_{f,m,T}^2)
  \left(
I-\sum_{|\alpha|\geq 1} a_{\tilde\alpha} \Lambda_\alpha \otimes
T_{\tilde\alpha}^* \right)^{-m}(1\otimes x), 1\otimes y\right>
\end{equation*}
  where $g\in B(F^2(H_n))$, $\Delta_{f,m,T}:= [(id-\Phi_{f,T})^m(I)]^{1/2}$ and $x,y\in \cH$.
  We should add that the
$n$-tuple $(\Lambda_1,\ldots, \Lambda_n)$  is  the universal model
associated with the noncommutative domain
 ${\bold D}_{\widetilde f}^m $, where  $\widetilde
f= \sum a_{\tilde \alpha} Z_\alpha$ and $\widetilde \alpha:=
g_{i_k}\cdots g_{i_1}$ denotes the reverse of $\alpha =g_{i_1}\cdots
g_{i_k}\in \FF_n^+$.
    We remark that in the particular case when $n=1$, $m=1$, $f=Z$,
     $\cH=\CC$,
 and $T=\lambda\in \DD$, we recover the Berezin transform  \cite{Be} of a bounded
linear operator on the Hardy space $H^2(\DD)$, i.e.,
$$
{\bf B}_\lambda [g]=(1-|\lambda|^2)\left<g k_\lambda,
k_\lambda\right>,\qquad g\in B(H^2(\DD)),
$$
where $k_\lambda(z):=(1-\overline{\lambda} z)^{-1}$ and  $z,
\lambda\in \DD$. The noncommutative Berezin  transform  will play an
important role in this paper.

The {\it extended noncommutative Berezin transform}
$\widetilde{\bf B}_T$
 at  $T\in {\bf D}_f^m(\cH)$ is defined  by
 \begin{equation}
 \label{def-Be2}
 \widetilde{\bf B}_T[g]:= {K_{f,T}^{(m)}}^* (g\otimes I_\cH)K_{f,T}^{(m)},
 \qquad g\in B(F^2(H_n)),
 \end{equation}
where the {\it noncommutative Berezin kernel} \  $K_{f,T}^{(m)}:\cH
\to F^2(H_n)\otimes \overline{\Delta_{f,m,T}(\cH)}$  is given
  by
 \begin{equation*}
 K_{f,T}^{(m)}h=\sum_{\alpha\in \FF_n^+} \sqrt{b^{(m)}_\alpha}
e_\alpha\otimes \Delta_{f,m,T} T_\alpha^* h,\qquad h\in \cH.
\end{equation*}
We recall that
 the noncommutative Berezin transforms $\widetilde{\bf B}_T$ and
  ${\bf B}_T$ coincide for any $n$-tuple of operators
   $T:=(T_1, \ldots, T_n)\in{\bf D}_f^m(\cH)$ with joint spectral radius
   $r_f(T_1,\ldots, T_n)<1$. In the particular case when $m=1$ and $f=X_1+\cdots +X_n$, the extended Berezin transform coincides with the noncommutative Poisson transform
   (see \cite{Po-poisson}, \cite{Po-curvature}).

Let $T:=(T_1,\ldots, T_n)$ be an $n$-tuple of operators  in the
 noncommutative starlike domain $ {\bf D}_f^m(\cH)$   and let
  $\cS:=\overline{\text{\rm  span}} \{ W_\alpha W_\beta^*;\
\alpha,\beta\in \FF_n^+\}$.
   Then there is
    a unital completely contractive linear map
$ \cB_T:\cS \to B(\cH)$, called the {\it noncommutative Berezin
transform} at $T$  on the operator system $\cS$,  such that
\begin{equation}\label{BT}
\cB_T(g)=\lim_{r\to 1} {\bf B}_{rT}[g],\qquad g\in \cS,
\end{equation}
where the limit exists in the operator norm topology of $B(\cH)$,
and
$$
\cB_T(W_\alpha W_\beta^*)=T_\alpha T_\beta^*, \qquad \alpha,\beta\in
\FF_n^+. $$ In particular,  we have the following von Neumann type
inequality:
\begin{equation}
\label{vn1} \left\|\sum_{\alpha,\beta\in \Lambda} T_\alpha
T_\beta^*\otimes C_{\alpha,\beta}\right\| \leq
\left\|\sum_{\alpha,\beta\in \Lambda}
 W_\alpha W_\beta^*\otimes C_{\alpha,\beta}\right\|
 \end{equation}
 for any
  finite  set $\Lambda\subset \FF_n^+$ and $C_{\alpha,\beta}\in
  B(\cE)$, where $\cE$ is a Hilbert space. Moreover, the restriction of
$\cB_T$ to the domain algebra $\cA_n({\bf D}_f^m)$ is a
completely contractive homomorphism.

After these preliminaries on noncommutative Berezin transforms, we
are ready to
 introduce   the class of free holomorphic functions
on the noncommutative radial domain  ${\bf D}_{f,\text{\rm rad}}^m$,
with operator-valued coefficients.

For $0<r<1$, we define  the  domain ${\bf D}_{f,r}^m(\cH)\subset
B(\cH)^n$ by setting
$$
{\bf D}_{f,r}^m(\cH):=\left\{ (Y_1,\ldots, Y_n)\in B(\cH)^n:\
\left(\frac{1}{r} Y_1,\ldots,\frac{1}{r}Y_n\right)\in {\bf
D}_{f}^m(\cH)\right\}=r{\bf D}_{f}^m(\cH).
$$
  Since ${\bf
D}_{f}^m(\cH)$ is a starlike domain, we have ${\bf
D}_{f,r}^m(\cH)\subset {\bf D}_{f}^m(\cH)$ for any $r\in [0,1)$.
Consider  the radial domain
$$
{\bf D}_{f,\text{\rm rad}}^m(\cH):=\bigcup_{0\leq r<1}{\bf
D}_{f,r}^m(\cH)=\bigcup_{0\leq r<1} r{\bf D}_{f}^m(\cH).
$$
Since, throughout the paper,  ${\bf D}_f^m (\cH)$  is  assumed to be
closed in the operator norm topology,  we have $ {\bf
D}_{f,\text{\rm rad}}^m(\cH)^- ={\bf D}_f^m (\cH).$ Note that the
interior  of ${\bf D}_f^m (\cH)$, which we denote by $Int({\bf
D}_f^m (\cH))$, is a subset of ${\bf D}_{f,\text{\rm rad}}^m(\cH)$.
In particular, if $q$ is any positive regular noncommutative
polynomial, then
$$
Int({\bf D}_q^1 (\cH))={\bf D}_{q,\text{\rm rad}}^1(\cH)\quad \text{
and } \quad \overline{Int({\bf D}_q^1 (\cH))}={\bf D}_q^1 (\cH).
$$
A similar result holds, for example,  for the domains ${\bf D}_f^m
(\cH)$ if $m\geq 1$ and $f=\sum_{i=1}^n a_iX_i$ with $a_i>0$.

 Let $\cE $ be  a separable
Hilbert space. A formal power series $G :=\sum_{\alpha\in \FF_n^+}
Z_\alpha\otimes C_{(\alpha)}$, where $C_{(\alpha)}\in B(\cE)$,
  is called
  free holomorphic function on
the noncommutative domain ${\bf D}_{f,\text{\rm rad}}^m $,   with
coefficients in $B(\cE )$, if its representation on any Hilbert
space $\cH$, i.e.,  the  mapping
 $G:{\bf D}_{f,\text{\rm rad}}^m(\cH)\to B(\cH)\bar\otimes_{min} B(\cE)$
 given  by
$$
G(X_1,\ldots, X_n):=\sum_{k=0}^\infty \sum_{|\alpha|=k}
X_\alpha\otimes C_{(\alpha)}, \qquad  (X_1,\ldots, X_n)\in {\bf
D}_{f,\text{\rm rad}}^m(\cH),
$$
is well-defined  in the operator norm topology. The  map $G$ is
called
  {\it free
holomorphic function} on ${\bf D}_{f,\text{\rm rad}}^m(\cH)$. We
remark that if $m=1$ and $f=X_1+\cdots +X_n$, then our definition
coincides with the one for free holomorphic functions (see
\cite{Po-holomorphic}) on the open unit ball  of $B(\cH)^n$.

Our first result is the following characterization of free
holomorphic functions on the noncommutative domain ${\bf
D}_{f,\text{\rm rad}}^m$, with operator-valued coefficients.

\begin{theorem}\label{free-ho}
Let $G(Z_1,\ldots, Z_n):=\sum_{\alpha\in \FF_n^+}
   Z_\alpha\otimes C_{(\alpha)}$ be a  formal power series with operator-valued
coefficients  \  $C_{(\alpha)}\in B(\cE)$,
   and let $\cH$ be a separable infinite dimensional Hilbert space.
Then the following statements are equivalent:
\begin{enumerate}
\item [(i)] $G$ is a  free holomorphic function on
the noncommutative domain ${\bf D}_{f,\text{\rm rad}}^m$.
\item[(ii)] For  any  $ r\in
[0,1)$, the series
$$
 G(rW_1,\ldots, rW_n):=\sum_{k=0}^\infty  \sum_{|\alpha|=k}
r^{|\alpha|}  W_\alpha\otimes C_{(\alpha)}
$$
is convergent in the operator norm topology,  where $(W_1,\ldots,
W_n)$ is the universal model associated with the noncommutative
domain ${\bf D}_f^m $.
\item[(iii)] The inequality
$$
 \limsup\limits_{k\to\infty}
\left\|\sum_{|\alpha|=k}
 \frac{1}{b_\alpha^{(m)}}C_{(\alpha)}^* C_{(\alpha)}\right\|^{\frac{1} {2k}}\leq 1,
 $$
 holds,
where the coefficients $b_\alpha^{(m)}$   are given by relation
\eqref{b-al}.
\item[(iv)] For  any  $ r\in
[0,1)$, the series
$
 \sum_{k=0}^\infty \left\|\sum_{|\alpha|=k}
r^{|\alpha|}   W_\alpha\otimes C_{(\alpha)}\right\| $ is convergent.
\item[(v)] For  any $(X_1,\ldots, X_n)\in {\bf D}_{f,\text{\rm
rad}}^m(\cH)$, the series
$
 \sum_{k=0}^\infty \left\|\sum_{|\alpha|=k}
    X_\alpha\otimes C_{(\alpha)}\right\|
$ is convergent.
\end{enumerate}
\end{theorem}

\begin{proof} Since $(W_1,\ldots, W_n)\in {\bf D}_f^m (F^2(H_n))$, the implication
  $(i)\implies (ii)$ is obvious. To prove that $(ii)\implies (iii)$, we assume that
$$
 \limsup\limits_{k\to\infty}
\left\|\sum_{|\alpha|=k}
 \frac{1}{b_\alpha^{(m)}}C_{(\alpha)}^* C_{(\alpha)}\right\|^{\frac{1} {2k}}>\gamma>1.
$$
Then there is $k$ as large as  necessary  such that
$\left\|\sum_{|\alpha|=k}
 \frac{1}{b_\alpha^{(m)}}C_{(\alpha)}^* C_{(\alpha)}\right\|>\gamma^{2k}$. Let $r$ be  such that
$\frac{1}{\gamma}<r<1$ and let $h\in \cE$.  Note that, using
relation \eqref{WbWb}, we have
\begin{equation*}
\begin{split}
\left\|\sum_{|\beta|=k}  r^{|\beta|}W_\beta(1) \otimes
C_{(\beta)}h\right\|&=\left\|\sum_{|\beta|=k}
r^{|\beta|}\frac{1}{\sqrt{b_\beta^{(m)}}} e_\beta \otimes
C_{(\beta)}h
 \right\|\\
&=\left(r^{2k}\left<\sum_{|\beta|=k}
 \frac{1}{b_\beta^{(m)}}C_{(\beta)}^*C_{(\beta)}h,h\right>\right)^{1/2}.
\end{split}
\end{equation*}
Hence, we deduce that
$$
\left\|\sum_{|\beta|=k}  r^{|\beta|}W_\beta \otimes
C_{(\beta)}\right\|\geq r^k\left\|\sum_{|\alpha|=k}
 \frac{1}{b_\alpha^{(m)}}C_{(\alpha)}^* C_{(\alpha)}\right\|^{1/2}
>(r\gamma)^k.
$$
 Since $r\gamma>1$, the series
$\sum\limits_{k=0}^\infty
 \sum\limits_{|\beta|=k}   r^{|\beta|}W_\beta\otimes C_{(\beta)}$ is divergent in the operator norm topology.

Now we prove that $(iii)\implies (iv)$. Assume that  item $(iii)$
holds and let $0<r<\rho<1$. Then, we can find $m_0\in \NN$ such that
\begin{equation}
\label{box}
 \left\|\sum_{|\alpha|=k}
 \frac{1}{b_\alpha^{(m)}}C_{(\alpha)}^* C_{(\alpha)}\right\|<\left(\frac{1}{\rho}\right)^{2k}\quad
\text{ for any } \ k\geq m_0.
\end{equation}
According to  Lemma 1.1 from \cite{Po-Berezin} and  relation
\eqref{WbWb}, the operators $\{W_\beta\}_{|\beta|=k}$ have
orthogonal ranges and

$$\|W_\beta x\|\leq \frac{1}{\sqrt{b^{(m)}_\beta}} M_{|\beta|,m}\|x\|,
\qquad x\in F^2(H_n),
$$
where $M_{|\beta|, m}:=\left(\begin{matrix} |\beta|+m-1\\m-1
\end{matrix}\right)$.
Consequently,  we have
\begin{equation*}
  \left\|\sum\limits_{|\beta|=k} b_\beta^{(m)} W_\beta
W_\beta^*\right\|\leq \left(\begin{matrix} k+m-1\\m-1
\end{matrix}\right)\quad \text{  for any } \quad k=0,1,\ldots.
\end{equation*}
  Hence, and using  relation \eqref{box},
 we deduce that
\begin{equation*}
 \begin{split}
 \sum_{k=0}^\infty r^k \left\|\sum_{|\beta|=k}   W_\beta\otimes C_{(\beta)}\right\|
 &\leq
 \sum_{k=0}^\infty r^k \left\|\sum_{|\beta|=k}  \frac{1}
 {b^{(m)}_\beta}C_{(\beta)}^* C_{(\beta)}\right\|^{1/2}\left\|\sum_{|\beta|=k} b^{(m)}_\beta
  W_\beta W_\beta^*\right\|^{1/2}  \\
  &\leq
  \sum_{k=0}^\infty \left\|\sum_{|\beta|=k}  \frac{1}
 {b^{(m)}_\beta}C_{(\beta)}^* C_{(\beta)}\right\|^{1/2} r^k \left(\begin{matrix}
k+m-1\\m-1 \end{matrix}\right)^{1/2}  \\
 &\leq
     \sum_{k=0}^\infty \left(\frac{r}{\rho}\right)^{k}\left(\begin{matrix}
k+m-1\\m-1 \end{matrix}\right)^{1/2} < \infty,
 \end{split}
 \end{equation*}
which proves  (iv).
  Since the implication $(v)\implies (i)$ is obvious, it
remains to prove that $(iv)\implies (v)$. To this end,  assume that
item $(iv)$ holds    and let $(X_1\ldots, X_n)\in {\bf
D}_{f,\text{\rm rad}}^m(\cH)$. Then there exists $r\in
(0,1)$ such that
 $(\frac{1}{r} X_1,\ldots,\frac{1}{r}X_n)\in {\bf D}_f^m(\cH)$
and,  using the  noncommutative von Neumann inequality \eqref{vn1},
we deduce that
$$
\left\|\sum_{|\alpha|=k}  X_\alpha\otimes C_{(\alpha)}\right\|\leq
\left\|\sum_{|\alpha|=k}  r^{|\alpha|}W_\alpha\otimes
C_{(\alpha)}\right\|.
$$
Hence, the series
$\sum\limits_{k=0}^\infty\left\|\sum\limits_{|\alpha|=k}
X_\alpha\otimes C_{(\alpha)}\right\|$ is convergent and, therefore,
item $(v)$ holds.
 The proof is complete.
\end{proof}

We remark  that  the coefficients  of a  free holomorphic function
are uniquely determined by its representation on
   an infinite dimensional   Hilbert space. Indeed,
   let $F:{\bf
D}_{f,\text{\rm rad}}^m(\cH)\to
  B( \cH)\bar\otimes_{min} B(\cE)$ be a
   free
holomorphic function  on  ${\bf D}_{f,\text{\rm rad}}^m(\cH)$  with
coefficients
 $A_{(\alpha)}\in B(\cE)$,
$\alpha\in \FF_n^+$. That is
$$
F(X_1,\ldots, X_n)=\sum\limits_{k=0}^\infty \sum\limits_{|\alpha|=k}
X_\alpha\otimes  A_{(\alpha)},
$$
where the series converges in the operator  norm topology  for any
$(X_1,\ldots, X_n)\in {\bf D}_{f,\text{\rm rad}}^m(\cH)$.
   Let $0<r<1$ and assume that
   $F(rW_1,\ldots, rW_n)=0$.  Taking into account  relation  \eqref{WbWb},
    we have
\begin{equation*}
\left< F(rW_1,\ldots, rW_n)(1\otimes x), ( W_\alpha\otimes
I_\cE)(1\otimes
y)\right>=\frac{r^{|\alpha|}}{b_\alpha^{(m)}}\left<A_{(\alpha)}x,y\right>=0
\end{equation*}
for any $x, y\in \cE$   and $\alpha\in \FF_n^+$. Therefore
$A_{(\alpha)}=0$ for any $\alpha\in \FF_n^+$, which proves our
assertion. Due to this reason, throughout  this paper, we assume
that $\cH$ is a separable infinite dimensional Hilbert space.

Let $G(X_1,\ldots, X_n):=\sum_{k=0}^\infty\sum_{|\alpha|=k }
  X_\alpha \otimes C_{(\alpha)}$ be  a  free holomorphic function on
  ${\bf D}^m_{f, \text{\rm rad}}(\cH)$ with coefficients in $B(\cE)$.
 Note that if $0<r_1<r_2<1$, then ${\bf D}_{f,r_1}^m(\cH)\subseteq
{\bf D}_{f,r_2}^m(\cH)$ and
\begin{equation}\label{Gr}
\|G_{r_1}(W_1,\ldots,  W_n)\|\leq \|G_{r_2}(W_1,\ldots,  W_n)\|,
\end{equation}
where $G_{r}(W_1,\ldots,  W_n):=\sum_{k=0}^\infty \sum_{|\alpha|=k}
r^{|\alpha|}W_\alpha\otimes C_{(\alpha)}$, $0\leq r<1$, is
convergent in the operator norm topology. Indeed, since
$\varphi(W_1,\ldots, W_n):=\sum_{k=0}^\infty \sum_{|\alpha|=k}
r_2^{|\alpha|} W_\alpha\otimes C_{(\alpha)}$ is in the tensor
algebra $\cA_n({\bf D}_f^m)\bar\otimes_{min}B(\cE)$ and
$(rW_1,\ldots, rW_n)\in {\bf D}_f^m(F^2(H_n))$, the noncommutative
von Neumann inequality \eqref{vn1} implies
$$\|\varphi(rW_1,\ldots, rW_n)\|\leq
\|\varphi(W_1,\ldots, W_n)\|, \qquad   r\in [0,1).
$$
Taking $r:=\frac{r_1}{r_2}$ in the latter inequality,  our assertion
follows.

 We  also remark that, for each $r\in [0,1)$,  the map \ $G:{\bf D}_{f,r}^m(\cH)\to
 B(\cH)\otimes_{min} B(\cE)$
defined by $G(X_1,\ldots, X_n):=\sum_{k=0}^\infty \sum_{|\alpha|=k}
X_\alpha \otimes C_{(\alpha)}$ is continuous in the operator norm
topology and \begin{equation} \label{vnr}\|G(X_1,\ldots, X_n)\|\leq
\|G(rW_1,\ldots, rW_n)\|, \qquad (X_1,\ldots, X_n) \in {\bf
D}_{f,r}^m(\cH).
\end{equation}
Moreover, the series defining $G$  converges uniformly on ${\bf
D}_{f,r}^m(\cH)$ in the same topology. Indeed, since
$G_r(W_1,\ldots, W_n)\in \cA_n({\bf D}_f^m)\bar \otimes_{min}
B(\cE)$ and $\left(\frac{1}{r} X_1,\ldots,\frac{1}{r}X_n\right)\in
{\bf D}_f^{\text{\rm rad}}(\cH)$, and  using  again    the
noncommutative von Neumann inequality \eqref{vn1}, we obtain
\begin{equation}
\label{2ine} \|G(X_1,\ldots, X_n)\|=\left\|G_r\left(\frac{1}{r}
X_1,\ldots,\frac{1}{r}X_n\right)\right\|\leq \|G_r(W_1,\ldots,
W_n)\|
\end{equation}
and
$$
\sum_{k=0}^\infty \left\|\sum_{|\alpha|=k} X_\alpha\otimes
C_{(\alpha)}\right\|\leq \sum_{k=0}^\infty \left\|\sum_{|\alpha|=k}
r^{|\alpha|}W_\alpha\otimes C_{(\alpha)}\right\|
$$
for any $(X_1,\ldots, X_n) \in {\bf D}_{f,r}^m(\cH)$.  This clearly
implies our assertion.

\bigskip

As in the particular case $m=1$ (see \cite{Po-domains}), there is an
important connection between the theory of free holomorphic
functions on noncommutative domains ${\bf D}_{f,\text{\rm rad}}^m$,
$m\geq 1$,  and the theory of holomorphic functions on domains in
$\CC^d$(\cite{Kr}).

Indeed, consider the case when $\cH=\CC^p$ and $p=1,2\ldots.$ Then
     ${\bf D}_{f}^m(\CC^p)$ can be seen as a subset of $\CC^{np^2}$ with
an arbitrary norm. We denote by $Int({\bf D}_{f}^m(\CC^p))$ the
interior of the closed set $ {\bf D}_{f}^m(\CC^p)$.  Note that
$Int({\bf D}_{f}^m(\CC^p))$ are circular domains, i.e.,
$(e^{i\theta}\Lambda_1,\ldots, e^{i\theta}\Lambda_n)\in Int({\bf
D}_{f}^m(\CC^p))$ for any $(\Lambda_1,\ldots \Lambda_n)\in Int({\bf
D}_{f}^m(\CC^p))$ and $\theta\in \RR$. In the particular case when
$p=1$, the interior $Int({\bf D}_{f}^m(\CC))$ is a Reinhardt domain,
i.e., $(e^{i\theta_1}\lambda_1,\ldots, e^{i\theta_n}\lambda_n)\in
Int({\bf D}_{f}^m(\CC))$ for any $(\lambda_1,\ldots \lambda_n)\in
Int({\bf D}_{f}^m(\CC))$ and $\theta_1,\ldots \theta_n\in \RR$.

If $K$ is a compact subset in the interior of
${\bf D}_{f}^m(\CC^p)$, then there exists  $r\in (0,1)$ such that
$K\subset {\bf D}_{f, r}^m(\CC^p)$. Given a free holomorphic
function on
  ${\bf D}^m_{f, \text{\rm rad}}(\cH)$ with scalar coefficients,
   $F(X_1,\ldots, X_n):=\sum\limits_{k=0}^\infty\sum\limits_{|\alpha|=k
} c_\alpha X_\alpha $, the results above show that
$$
\left\|F(\Lambda_1,\ldots, \Lambda_n)-\sum_{k=0}^q
\sum_{|\alpha|=k}c_\alpha \Lambda_\alpha\right\|\leq
\sum_{k=q+1}^\infty \left\|\sum_{|\alpha|=k} r^{|\alpha|}c_\alpha
W_\alpha  \right\|
$$
for any $(\Lambda_1,\ldots, \Lambda_n)\in K$. Consequently,
$\sum_{k=0}^q \sum_{|\alpha|=k}c_\alpha \Lambda_\alpha$ converges
 to $F(\Lambda_1,\ldots, \Lambda_n)$ uniformly on $K$, as
$q\to\infty$. Therefore, the map $(\Lambda_1,\ldots,
\Lambda_n)\mapsto F(\Lambda_1,\ldots, \Lambda_n)$ is holomorphic on
the interior of  ${\bf D}_{f}^m(\CC^p)$.   To record this result,
let $M_p$  denote the set of all $p\times p$ matrices with entries
in $\CC$.

\begin{corollary} \label{rep-finite1} If  $p\in \NN$ and
$F(X_1,\ldots, X_n):=\sum\limits_{k=0}^\infty\sum\limits_{|\alpha|=k
}
  c_\alpha X_\alpha $ is  a  free holomorphic function on
  ${\bf D}^m_{f, \text{\rm rad}}(\cH)$ with scalar coefficients,
  then its representation on $\CC^p$,  i.e., the map
  $F_p$ defined by
  $$
  \CC^{np^2}\supset {\bf
D}_{f,\text{\rm rad}}^m(\CC^p)\ni (\Lambda_1,\ldots,
\Lambda_n)\mapsto F(\Lambda_1,\ldots, \Lambda_n)\in M_{ p}\subset
\CC^{p^2}
$$
is a  holomorphic function on the interior of ${\bf
D}_{f}^m(\CC^p)$.
\end{corollary}

\bigskip

\section{Noncommutative   domain algebras and Hardy algebras}

We introduce the  Hardy algebra
$H^\infty({\bf D}_{f,\text{\rm rad}}^m)$ and  the domain algebra
$A({\bf D}_{f,\text{\rm rad}}^m)$ associated with the noncommutative domain ${\bf D}_{f,\text{\rm rad}}^m$, $m\geq1$. Using noncommutative Berezin
transforms, we identify  them with the  noncommutative algebras $F_n^\infty({\bf
D}^m_f)$ and  $\cA_n({\bf D}^m_f)$,
 respectively. Hardy algebras of bounded free holomorphic functions with
 operator-valued coefficients are  also discussed.

Let us recall some definitions concerning completely bounded maps
 on operator spaces. We identify $M_p(B(\cH))$, the set of
$p\times p$ matrices with entries from $B(\cH)$, with
$B( \cH^{(p)})$, where $\cH^{(p)}$ is the direct sum of $p$ copies
of $\cH$.
Thus we have a natural $C^*$-norm on
$M_p(B(\cH))$. If $X$ is an operator space, i.e., a closed subspace
of $B(\cH)$, we consider $M_p(X)$ as a subspace of $M_p(B(\cH))$
with the induced norm.
Let $X, Y$ be operator spaces and $u:X\to Y$ be a linear map. Define
the map
$u_p:M_p(X)\to M_p(Y)$ by
$$
u_p ([x_{ij}]):=[u(x_{ij})].
$$
We say that $u$ is completely bounded   if
$
\|u\|_{cb}:=\sup_{p\ge1}\|u_p\|<\infty.
$
If $\|u\|_{cb}\leq1$
(resp. $u_p$ is an isometry for any $p\geq1$) then $u$ is completely
contractive (resp. isometric),
 and if $u_p$ is positive for all $p$, then $u$ is called
 completely positive.
   For more information  on completely bounded maps,  we refer
 to \cite{Pa-book}, \cite{Pi}, and \cite{ER}.

 According to \cite{Po-Berezin}, if  $T:=(T_1,\ldots, T_n)\in {\bf
 D}^m_f(\cH)$  is {\it completely
non-coisometric }(c.n.c) with respect to ${\bf D}_f^m(\cH)$, i.e.,
there is no vector $h\in \cH$, $h\neq 0$, such that
$$
\left< \Phi_{f,T}^k(I)h,h\right>=\|h\|^2 \quad \text{ for any } \
k=1,2,\ldots,
$$
 then
 \begin{equation}
 \label{Be-transf}
 \cB_T(G):=\text{\rm SOT-}\lim_{r\to 1} G_r(T_1,\ldots, T_n), \qquad
  G=G(W_1,\ldots,
 W_n)\in F_n^\infty({\bf D}_f^m),
 \end{equation}
 exists in the strong operator topology   and defines a map
 $\cB_T:F_n^\infty({\bf D}_f^m)\to B(\cH)$, called the
 {\it noncommutative Berezin transform} at $T$ on  the Hardy algebra $F_n^\infty({\bf D}_f^m)$,
  with the following
 properties:
\begin{enumerate}
\item[(i)]
$\cB_T(G)=\text{\rm SOT-}\lim\limits_{r\to 1}{\bf
B}_{rT}[G]$, where ${\bf B}_{rT}$ is the noncommutative  Berezin
transform  at $rT\in {\bf D}_f^m(\cH)$;
\item[(ii)] $\cB_T$ is    WOT-continuous (resp.
SOT-continuous)  on bounded sets;
\item[(iii)]
$\cB_T$ is a unital completely contractive homomorphism.
\end{enumerate}
We remark that if $T$ is a {\it pure} $n$-tuple of operators, i.e.,
$\text{\rm SOT-}\lim_{p\to\infty}\Phi_{f,T}^p(I)=0$,  then $T$ is
c.n.c. and
$$
\cB_T(G)=\widetilde{\bf B}_T[G]={K_{f,T}^{(m)}}^* (G\otimes
I_{\cD_{f,m,T}})K_{f,T}^{(m)},
 \qquad G\in F_n^\infty({\bf D}_f^m),
$$
where  $\cD_{f,m,T}:=\overline{\Delta_{f,m,T}(\cH)}$  and
$\widetilde{\bf B}_T$ is the extended noncommutative Berezin
transform  defined by relation \eqref{def-Be2}. In particular, if
$G(W_1,\ldots, W_n):=\sum_{\alpha\in \FF_n^+} c_\alpha W_\alpha $ in
the Hardy algebra $F_n^\infty({\bf D}^m_f)$, then
\begin{equation}
\label{sot}
G(W_1,\ldots, W_n)=\text{\rm SOT-}\lim\limits_{r\to
1}G_r(W_1,\ldots, W_n),
\end{equation}
 where
$G_r(W_1,\ldots, W_n):= {\bf B}_{rW}[G]=\sum_{k=0}^\infty
\sum_{|\alpha|=k} c_\alpha r^{|\alpha|} W_\alpha $, and
\begin{equation}
\label{GrG}
\|G_r(W_1,\ldots, W_n)\|\leq \|G(W_1,\ldots, W_n)\|,
\qquad r\in [0,1).
\end{equation}
 Moreover,   we can prove that
\begin{equation}\label{many eq}
\|G(W_1,\ldots, W_n)\|=\sup_{0\leq r<1}\|G(rW_1,\ldots, rW_n)\|=
\lim_{r\to 1}\|G(rW_1,\ldots, rW_n)\|.
\end{equation}
Indeed, if   $G(W_1,\ldots, W_n)\in F_n^\infty({\bf D}^m_f))$ and
$\epsilon>0$, then there exists a polynomial $q\in F^2(H_n)$ with
$\|q\|=1$ such that
$$
\|G(W_1,\ldots, W_n)q\|>\|G(W_1,\ldots, W_n)\|-\epsilon.
$$
By   relation \eqref{sot}, there is $r_0\in (0,1)$ such that
$$\|G_{r_0}(W_1,\ldots, W_n)q\|>\|G(W_1,\ldots, W_n)\|-\epsilon.
$$
Combining this with the inequality \eqref{GrG}, we obtain
$$
\sup_{0\leq r<1}\|G(rW_1,\ldots, rW_n)\|=\|G(W_1,\ldots, W_n)\|.
$$
  Since    the   function
 $[0,1)\ni r\to \|G(rW_1,\ldots, rW_n)\|\in \RR^+$
is increasing (see relation \eqref{Gr}), we can  complete the proof
of our assertion.

Now, we are ready to introduce the noncommutative Hardy algebra
$H^\infty({\bf D}_{f,\text{\rm rad}}^m)$. We denote by $Hol({\bf
D}_{f,\text{\rm rad}}^m)$ the set of all free holomorphic functions
on  ${\bf D}_{f,\text{\rm rad}}^m$ with scalar coefficients. Due to
Theorem \ref{free-ho}, if $G\in Hol({\bf D}_{f,\text{\rm rad}}^m)$,
then $G(rW_1,\ldots, rW_n)$ is in the domain algebra $\cA_n({\bf
D}_f^m)$ for any $r\in [0,1)$. Consequently,  one can easily show
that $Hol({\bf D}_{f,\text{\rm rad}}^m)$ is an algebra. Let
$H^\infty({\bf D}_{f,\text{\rm rad}}^m)$  denote the set of  all
elements $\varphi$ in $Hol({\bf D}_{f,\text{\rm rad}}^m)$     such
that
$$\|\varphi\|_\infty:= \sup \|\varphi(X_1,\ldots, X_n)\|<\infty,
$$
where the supremum is taken over all $n$-tuples $(X_1,\ldots,
X_n)\in {\bf D}_{f,\text{\rm rad}}^m(\cH)$ and any Hilbert space
$\cH$. One can  show that $H^\infty({\bf D}_{f,\text{\rm
rad}}^m)$ is a Banach algebra under pointwise multiplication and the
norm $\|\cdot \|_\infty$.

For each $p=1,2,\ldots$, we define the norms $\|\cdot
\|_p:M_p\left(H^\infty({\bf D}_{f,\text{\rm rad}}^m)\right)\to
[0,\infty)$ by setting
$$
\|[F_{ij}]_p\|_p:= \sup \|[F_{ij}(X_1,\ldots, X_n)]_p\|,
$$
where the supremum is taken over all $n$-tuples $(X_1,\ldots,
X_n)\in {\bf D}_{f,\text{\rm rad}}^m$ and any Hilbert space $\cH$.
It is easy to see that the norms  $\|\cdot\|_p$, $p=1,2,\ldots$,
determine  an operator space structure  on $H^\infty({\bf
D}_{f,\text{\rm rad}}^m)$,
 in the sense of Ruan (\cite{ER}).

Given  $g\in F_n^\infty({\bf D}^m_f)$, the noncommutative Berezin
transform  associated with the noncommutative domain ${\bf D}^m_f$
generates a function
$$
{\bf B}[\varphi]:{\bf D}_{f,\text{\rm rad}}^m(\cH)\to B(\cH)
$$
by setting
$$
{\bf B}[\varphi](X_1,\ldots, X_n):={\bf B}_{X}[\varphi],\qquad
X:=(X_1,\ldots, X_n)\in {\bf D}_{f,\text{\rm rad}}^m(\cH).
$$
We call ${\bf B}[\varphi]$  the Berezin transform of $\varphi$. In
what follows,  we identify the  noncommutative algebra
$F_n^\infty({\bf D}^m_f)$ with the  Hardy subalgebra $H^\infty({\bf
D}_{f,\text{\rm rad}}^m)$ of   bounded free holomorphic functions on
${\bf D}_{f,\text{\rm rad}}^m$.

\begin{theorem}\label{f-infty} The map
$ \Phi:H^\infty({\bf D}_{f,\text{\rm rad}}^m)\to F_n^\infty({\bf
D}^m_f) $ defined by
$$
\Phi\left(\sum\limits_{\alpha\in \FF_n^+} c_\alpha
Z_\alpha\right):=\sum\limits_{\alpha\in \FF_n^+} c_\alpha W_\alpha
$$
is a completely isometric isomorphism of operator algebras.
Moreover, if  $G:=\sum\limits_{\alpha\in \FF_n^+} c_\alpha Z_\alpha$
is  a free holomorphic function on the noncommutative domain ${\bf
D}_{f,\text{\rm rad}}^m$, then the following statements are equivalent:
 \begin{enumerate}
 \item[(i)]$G\in H^\infty({\bf D}_{f,\text{\rm
rad}}^m)$;
\item[(ii)] $\sup\limits_{0\leq r<1}\|G(rW_1,\ldots,
rW_n)\|<\infty$, where $G(rW_1,\ldots,
rW_n):=\sum_{k=0}^\infty\sum_{|\alpha|=k}c_\alpha r^{|\alpha|} W_\alpha$;
\item[(iii)]
there exists $\varphi\in F_n^\infty({\bf D}^m_f)$ with $G={\bf
B}[\varphi]$.
\end{enumerate}

In this case,
$$
\Phi(G)=\text{\rm SOT-}\lim_{r\to 1}G(rW_1,\ldots, rW_n)  \quad
\text{ and } \quad  \Phi^{-1}(\varphi)={\bf B}[\varphi],\quad
\varphi\in F_n^\infty({\bf D}^m_f),
$$
where ${\bf B}$ is the  noncommutative Berezin
transform  associated with the noncommutative domain ${\bf D}^m_f$.
\end{theorem}

\begin{proof} First we need to show that the map $\Phi$ is well-defined. Let $G\in H^\infty({\bf D}_{f,\text{\rm rad}}^m)$ and let
$$G(X_1,\ldots, X_n):=\sum\limits_{k=0}^\infty \sum\limits_{|\alpha|=k}c_\alpha
X_\alpha, \qquad (X_1,\ldots, X_n)\in {\bf D}_{f,\text{\rm rad}}^m
(\cH),
$$
be its representation on a separable infinite dimensional  Hilbert
space $\cH$. Since $(rW_1,\ldots, rW_n)\in {\bf D}_{f,\text{\rm
rad}}^m(F^2(H_n))$,   we have
$$
\sup_{0\leq r<1}\|G(rW_1,\ldots, rW_n)\|\leq \|G\|_\infty<\infty.
$$
 Hence,  taking into account
relation \eqref{WbWb}, we deduce that
\begin{equation*}
\begin{split}
\sum_{\alpha\in \FF_n^+} r^{2|\alpha|} |c_\alpha|^2
\frac{1}{b_\alpha^{(m)}}&=
\left\|\sum_{\alpha\in \FF_n^+} r^{|\alpha|} c_\alpha W_\alpha(1)\right\|\\
&\leq \sup\limits_{0\leq r<1}\|G(rW_1,\ldots, rW_n)\|<\infty
\end{split}
\end{equation*}
for any $0\leq r<1$. Consequently, $\sum\limits_{\alpha\in \FF_n^+}
|c_\alpha|^2\frac{1}{b_\alpha^{(m)}}<\infty$, which shows that $
G(W_1,\ldots, W_n)p$ is in the full Fock space $F^2(H_n)$ for any
polynomial $p\in F^2(H_n)$. Now assume that
  $G(W_1,\ldots, W_n)\notin F_n^\infty({\bf
D}^m_f)$.
 According to the definition of $F_n^\infty({\bf
D}^m_f)$,  for any fixed  positive
  number $M$, there exists a polynomial $q\in F^2(H_n)$ with $\|q\|=1$ such that
$$
\|G(W_1,\ldots, W_n)q\|>M.
$$
Since $\|G_r(W_1,\ldots, W_n)(1)-G(W_1,\ldots, W_n)(1)\|\to 0$ as
$r\to 1$, we have
$$\|G(W_1,\ldots, W_n)q-G_r(W_1,\ldots, W_n)q\| \to 0,
\quad \text{ as }\ r\to 1.
$$
Consequently, there is $r_0\in (0,1)$ such that $
\|G_{r_0}(W_1,\ldots, W_n) q\|> M$,  which implies $
\|G_{r_0}(W_1,\ldots, W_n)\|
>M. $ Since $M>0$ is arbitrary, we deduce that
$$
\sup_{0\leq r<1}\|G(rW_1,\ldots, rW_n)\|=\infty,
$$
which is a contradiction. Therefore, $G(W_1,\ldots, W_n)\in
F_n^\infty({\bf D}^m_f)$, which proves that the map $\Phi$ is
well-defined. Moreover,  due to relation \eqref{vnr}, we have
$\|G(X_1,\ldots, X_n)\|\leq \|G(rW_1,\ldots, rW_n)\|$  for any
$(X_1,\ldots, X_n)\in {\bf D}_{f,r}^m(\cH)$.  Using  now   relation
\eqref{many eq}, we deduce that
$$
\|G(W_1,\ldots, W_n)\|=\sup_{0\leq r<1}\|G(rW_1,\ldots,
rW_n)\|=\|G\|_\infty.
$$
Therefore, $\Phi$ is  a well-defined isometric linear map. Moreover,
 relation \eqref{sot},  implies
$$
\Phi(G)=G(W_1,\ldots, G_m)=\text{\rm SOT-}\lim\limits_{r\to
1}G_r(W_1,\ldots, W_n).
$$

 We show now
that $\Phi$ is a surjective map. To this end, let $
\varphi:=\sum_{\alpha\in \FF_n^+} c_\alpha W_\alpha$ be in
$F_n^\infty({\bf D}^m_f)$. Then $\sum_{\alpha\in \FF_n^+}
|c_\alpha|^2 \frac{1}{b_\alpha^{(m)}}<\infty$, which implies
\begin{equation*}
 \limsup\limits_{k\to\infty}
\left(\sum_{|\alpha|=k}
|c_\alpha|^2\frac{1}{b_\alpha^{(m)}}\right)^{\frac{1} {2k}}\leq 1.
\end{equation*}
Using now Theorem \ref{free-ho}, we deduce that $G(Z_1,\ldots,
Z_n):=\sum\limits_{\alpha\in \FF_n^+}
 c_\alpha Z_\alpha$
  is a free holomorphic function on the noncommutative domain ${\bf D}_{f,\text{\rm
  rad}}^m$.
   Due to inequalities  \eqref{2ine}, and \eqref{GrG}, we have
  $$
  \|G(X_1,\ldots, X_n)\|\leq \|G(rW_1,\ldots, rW_n)\|\leq \|G(W_1,\ldots, W_n)\|
  $$
  for any $(X_1,\ldots, X_n)\in {\bf D}_{f,r}^m(\cH)$ and $r\in [0,1)$. Hence, we
  deduce that
 $$ \sup_{(X_1,\ldots,
X_n)\in {\bf D}_{f,\text{\rm rad}}^m(\cH)} \|G(X_1,\ldots, X_n)\|\leq \|G(W_1,\ldots, W_n)\|<\infty,
$$
  which proves that $G\in H^\infty({\bf D}_{f,\text{\rm rad}}^m)$.
   This  shows that the map $\Phi$ is surjective.
   Therefore, we have proved
  that $\Phi$ is an isometric isomorphism of operator algebras.
  Using the same techniques and passing to matrices, one can prove that $\Phi$ is a
  completely isometric isomorphism.

Moreover,  note
   that if $X:=(X_1,\ldots, X_n)\in {\bf D}_{f,\text{\rm rad}}^m$,
   then  $G(X_1,\ldots, X_n)=\sum_{k=0}^\infty \sum_{|\alpha|=k}
   c_\alpha X_\alpha$ is convergent in the operator norm topology.
   Due to the properties of the noncommutative Berezin transform,
   for any $r\in [0,1)$, we have
   \begin{equation*}
   \begin{split}
\sum_{k=0}^\infty \sum_{|\alpha|=k}
   c_\alpha r^{|\alpha|}X_\alpha &=\sum_{k=0}^\infty {\bf B}_X\left[ \sum_{|\alpha|=k}
   c_\alpha r^{|\alpha|}W_\alpha\right]\\
   &={\bf B}_X\left[\sum_{k=0}^\infty \sum_{|\alpha|=k}
   c_\alpha r^{|\alpha|}W_\alpha\right]\\
   &={\bf B}_X[\varphi_r]={K_{X,f}^{(m)}}^* [\varphi_r\otimes
   I_\cH]K_{X,f}^{(m)}.
   \end{split}
   \end{equation*}
Recall that $\|\varphi_r\|\leq \|\varphi\|$ for any $r\in [0,1)$ and
$\varphi=\text{\rm SOT-}\lim\limits_{r\to 1}\varphi_r$.  Passing to
the limit as $r\to 1$ in the equalities above, and using the
continuity of the free holomorphic function $G$ on ${\bf
D}_{f,\text{\rm rad}}^m$, we obtain
$$
G(X_1,\ldots, X_n)={\bf B}_X[\varphi], \qquad X:=(X_1,\ldots,
X_n)\in {\bf D}_{f,\text{\rm rad}}^m.
$$
The equivalences mentioned in the theorem  can be easily deduced
from the considerations above and the properties of the
noncommutative Berezin transform.
  The proof is complete.
\end{proof}

 We  denote by  $A({\bf
D}_{f,\text{\rm rad}}^m)$   the set of all  elements $G$
  in $Hol({\bf
D}_{f,\text{\rm rad}}^m)$   such that the mapping
$${\bf
D}_{f,\text{\rm rad}}^m(\cH)\ni (X_1,\ldots, X_n)\mapsto
G(X_1,\ldots, X_n)\in B(\cH)$$
 has a continuous extension to  $[{\bf
D}_{f,\text{\rm rad}}^m(\cH)]^-={\bf
D}_{f}^m(\cH)$ for any Hilbert space $\cH$. One can show that  $A({\bf
D}_{f,\text{\rm rad}}^m)$ is a  Banach algebra under pointwise
multiplication and the norm $\|\cdot \|_\infty$.  Moreover, we can
identify the domain algebra $\cA_n({\bf D}^m_f)$ with the subalgebra
 $A({\bf D}_{f,\text{\rm rad}}^m)$.

\begin{theorem}\label{A-infty} The map
$ \Phi:A({\bf D}_{f,\text{\rm rad}}^m)\to \cA_n({\bf D}^m_f) $
defined by
$$
\Phi\left(\sum\limits_{\alpha\in \FF_n^+} c_\alpha
Z_\alpha\right):=\sum\limits_{\alpha\in \FF_n^+} c_\alpha W_\alpha
$$
is a completely isometric isomorphism of operator algebras.
Moreover, if  $G:=\sum\limits_{\alpha\in \FF_n^+} c_\alpha Z_\alpha$
is  a free holomorphic function on the  domain ${\bf
D}_{f,\text{\rm rad}}^m$,
then the following statements are equivalent:
 \begin{enumerate}
 \item[(i)]$G\in A({\bf D}_{f,\text{\rm
rad}}^m)$;
\item[(ii)] $G(rW_1,\ldots,
rW_n):=\sum\limits_{k=0}^\infty\sum\limits_{|\alpha|=k}c_\alpha r^{|\alpha|} W_\alpha$ is convergent in the operator norm topology  as $r\to 1$;
\item[(iii)]
there exists $\varphi\in \cA_n({\bf D}^m_f)$ with $G={\bf
B}[\varphi]$.
\end{enumerate}

In this case,
$$
\Phi(G)=\lim_{r\to 1}G(rW_1,\ldots,
rW_n)  \quad \text{ and } \quad  \Phi^{-1}(\varphi)={\bf B}[\varphi],\quad \varphi\in \cA_n({\bf D}^m_f),
$$
where ${\bf B}$ is the  noncommutative Berezin
transform  associated with the noncommutative domain ${\bf D}^m_f$.
\end{theorem}

\begin{proof}  Assume that $G\in A({\bf D}_{f,\text{\rm
rad}}^m)$.  Since $(rW_1,\ldots, rW_n)\in {\bf D}_{f,\text{\rm
rad}}^m (F^2(H_n))$ for $r\in [0,1)$, we deduce that
$\lim\limits_{r\to \infty} G(rW_1,\ldots, rW_n)$ exists in the
operator norm topology. Using  the fact that $G(rW_1,\ldots, rW_n)$
is in the domain algebra $ \cA_n({\bf D}^m_f)$, which is closed in
the operator norm topology, we have
$$
\varphi :=\lim\limits_{r\to \infty} G(rW_1,\ldots, rW_n)\in
\cA_n({\bf D}^m_f).
$$
  On the other
hand, due to Theorem \ref{f-infty}, we deduce that $\Phi(G):=\sum\limits_{\alpha\in \FF_n} c_\alpha W_\alpha$ is in $
F_n^\infty(\cD_f)$  and
$$\Phi(G)=\text{\rm
SOT-}\lim_{r\to 1} G(rW_1,\ldots, rW_n). $$
Therefore, we have
$$
\Phi(G)=\varphi  =\lim\limits_{r\to 1} G(rW_1,\ldots, rW_n)\in
\cA_n({\bf D}^m_f)
$$
 and $$G(X_1,\ldots, X_n)= \lim\limits_{r\to 1}{\bf B}_X[  G(rW_1,\ldots,
rW_n)]={\bf B}_X[\varphi],\qquad X:=(X_1,\ldots, X_n)\in {\bf D}_{f,\text{\rm
rad}}^m(\cH).
$$
Here we used the continuity of the Berezin transform  in the
operator norm topology.

Now, let $\varphi\in \cA_n({\bf D}^m_f)\subset F_n^\infty({\bf
D}^m_f)$ have the Fourier representation $\varphi=\sum_{\alpha\in
\FF_n^+} c_\alpha W_\alpha$.  Then, for any  $n$-tuple $(Y_1,\ldots,
Y_n)\in {\bf D}_{f}^m(\cH)$,
\begin{equation}
\label{tildaG} \tilde G(Y_1,\ldots, Y_n):=\lim_{r\to 1}{\bf
B}_{rY}[\varphi] =\lim_{r\to 1} \varphi(rY_1,\ldots, rY_n)
\end{equation}
exists in the operator norm  and
$$
\|\tilde G(Y_1,\ldots, Y_n)\|\leq \|\varphi\|\quad \text{ for any } \ (Y_1,\ldots, Y_n)\in {\bf
D}_{f}^m(\cH).
$$
 Note also that $\tilde G$ is an extension of the free holomorphic function
 $$G(X_1,\ldots, X_n)={\bf B}_X[\varphi]=\sum_{k=0}^\infty\sum_{|\alpha|=k}
  c_\alpha X_\alpha,\qquad (X_1,\ldots, X_n)\in  {\bf D}_{f,\text{\rm
rad}}^m(\cH).
$$
Indeed, if $ (X_1,\ldots, X_n)\in {\bf D}_{f,\text{\rm
rad}}^m(\cH)$, then
\begin{equation*}
\begin{split}
\tilde G(X_1,\ldots, X_n)&=\lim_{r\to 1}\varphi(rX_1,\ldots, rX_n)\\
&=\lim_{r\to 1}G(rX_1,\ldots, rX_n)=G(X_1,\ldots, X_n).
\end{split}
\end{equation*}
The last equality is due to the fact that $G$ is continuous on ${\bf
D}_{f,\text{\rm rad}}^m(\cH)$.

Let us prove that $\tilde G:{\bf D}_{f}^m(\cH)\to B(\cH)$ is a
continuous map. Since $\varphi\in \cA_n({\bf D}^m_f)$,  for any
$\epsilon>0$ there exists $r_0\in [0,1)$ such that
$\|\varphi-\varphi(r_0W_1,\ldots, r_0 W_n)\|<\epsilon$. Applying the
noncommutative von Neumann inequality to $ \varphi-\
\varphi_{r_0}\in \cA_n({\bf D}^m_f)$ and using relation
\eqref{tildaG}, we deduce that
\begin{equation}
\label{tild-f}
\|\tilde G(T_1,\ldots, T_n)-\varphi_{r_0}
(T_1,\ldots, T_n)\|\leq \|\varphi-\varphi_{r_0} \|< \frac{\epsilon}{3}
\end{equation}
for any $(T_1,\ldots, T_n)\in {\bf D}_{f}^m(\cH)$. We recall that
$G$ is a continuous function on ${\bf D}_{f,\text{\rm rad}}^m(\cH)$.
Therefore, there exists $\delta>0$ such that
$$
\|G_{r_0}(T_1,\ldots, T_n)-G_{r_0}(Y_1,\ldots, Y_n)\|<\frac{\epsilon}{3}
$$
 for any $n$-tuple  $(Y_1,\ldots, Y_n)$   in ${\bf
D}_{f}^m(\cH)$     such that  $\|(T_1-Y_1,\ldots,
T_n-Y_n)\|<\delta$. Hence, and using relation \eqref{tild-f}, we
have
\begin{equation*}
\begin{split}
\|\tilde G(T_1,\ldots, T_n)-\tilde G(Y_1,\ldots, Y_n)\|
&\leq \|\tilde G(T_1,\ldots, T_n)-\varphi_{r_0}(T_1,\ldots, T_n)\|\\
&\qquad + \| G_{r_0}(T_1,\ldots, T_n)- G_{r_0}(Y_1,\ldots, Y_n)\|\\
&\qquad + \|\varphi_{r_0}(Y_1,\ldots, Y_n)-\tilde G(Y_1,\ldots, Y_n)\|
<\epsilon,
\end{split}
\end{equation*}
whenever $\|(T_1-Y_1,\ldots, T_n-Y_n)\|<\delta$. This proves the
continuity of $\tilde G$ on ${\bf D}_{f}^m(\cH)$. Therefore, $G\in
A({\bf D}_{f}^m)$ and, moreover, we have
$$
\Phi(G)= \varphi  \quad \text{ and } \quad \Phi^{-1}(\varphi)={\bf
B}[\varphi],\quad \varphi\in \cA_n({\bf D}^m_f).
$$ Using now Theorem \ref{f-infty}, we deduce that
he map $ \Phi:A({\bf D}_{f,\text{\rm rad}}^m)\to \cA_n({\bf D}^m_f)
$ is  a completely isometric isomorphism of operator algebras. The
equivalences mentioned in the theorem  can be easily deduced from
the considerations above and the properties of the noncommutative
Berezin transform. This completes the proof.
  \end{proof}

Due to Theorem \ref{A-infty} and Corollary \ref{rep-finite1}, we
deduce the following result.

\begin{corollary} \label{rep-finite2} If  $p\in \NN$ and
$F(X_1,\ldots, X_n):=\sum\limits_{k=0}^\infty\sum\limits_{|\alpha|=k
}
  c_\alpha X_\alpha $ is in $A({\bf D}_{f,\text{\rm rad}}^m)$,
  then its representation on $\CC^p$,  i.e., the map
  $F_p$ defined by
  $$
  \CC^{np^2}\supset {\bf
D}_{f}^m(\CC^p)\ni (\Lambda_1,\ldots, \Lambda_n)\mapsto
F(\Lambda_1,\ldots, \Lambda_n)\in M_{ p}\subset \CC^{p^2}
$$
is a continuous map on  ${\bf D}_{f}^m(\CC^p)$ and  holomorphic
 on the interior of ${\bf D}_{f}^m(\CC^p)$.
\end{corollary}

We remark that one can obtain operator-valued versions of Theorem
\ref{f-infty}  and Theorem \ref{A-infty}, which we summarize  in
what follows.
Let $\cE$ be a separable Hilbert space. We denote by $Hol_\cE({\bf
D}_{f,\text{\rm rad}}^m)$ the set of all free holomorphic functions
on the noncommutative ball ${\bf D}_{f,\text{\rm rad}}^m$ and
coefficients in $B(\cE)$. Let ${ H}_\cE^\infty({\bf D}_{f,\text{\rm
rad}}^m)$ denote the set of all elements $F$ in $Hol_\cE({\bf
D}_{f,\text{\rm rad}}^m)$  such that
$$
\|F\|_\infty:=\sup  \|F(X_1,\ldots, X_n)\|<\infty,
$$
where the supremum is taken over all $n$-tuples  of operators
$(X_1,\ldots, X_n)\in {\bf D}_{f,\text{\rm rad}}^m(\cH)$  and any
Hilbert space $\cH$. The noncommutative Hardy space
  ${ H}_\cE^\infty({\bf D}_{f,\text{\rm rad}}^m)$  can be identified
  to
$  F_n^\infty({\bf D}_{f }^m)\bar\otimes B(\cE)$, the weakly closed
operator algebra generated by the spatial tensor product. More
precisely, a bounded free holomorphic function $F$ on ${\bf
D}_{f,\text{\rm rad}}^m$ with coefficients in $B(\cE)$  is uniquely
determined by its {\it model boundary function} $\widetilde
F(W_1,\ldots, W_n)\in F_n^\infty({\bf D}_{f }^m)\bar \otimes B(\cE)$
defined by
$$\widetilde F=\widetilde F(W_1,\ldots, W_n):=\text{\rm SOT-}\lim_{r\to 1}
F(rW_1,\ldots, rW_n). $$ Moreover, $F$ is  the noncommutative
Berezin  transform of $\widetilde F(W_1,\ldots, W_n)$ at
$X:=(X_1,\ldots, X_n)\in{\bf D}_{f,\text{\rm rad}}^m(\cH)$, i.e.,
$$
F(X_1,\ldots, X_n)= ({\bf B}_X\otimes I_\cE)[\widetilde
F(W_1,\ldots, W_n)].
$$

We  denote by  $A_\cE({\bf D}_{f,\text{\rm rad}}^m)$   the set of
all elements $\varphi$
  in $Hol_\cE({\bf D}_{f,\text{\rm rad}}^m)$   such that the mapping
$${\bf
D}_{f,\text{\rm rad}}^m(\cH)\ni (X_1,\ldots, X_n)\mapsto
\varphi(X_1,\ldots, X_n)\in B(\cH)$$
 has a continuous extension to  $[{\bf
D}_{f,\text{\rm rad}}^m(\cH)]^-={\bf D}_{f}^m(\cH)$ for any Hilbert
space $\cH$. One can show that  $A_\cE({\bf D}_{f,\text{\rm
rad}}^m)$ is a  Banach algebra under pointwise multiplication and
the norm $\|\cdot \|_\infty$.  Moreover, we can identify the tensor
algebra $\cA_n({\bf D}^m_f)\bar\otimes_{min}  B(\cE)$ (the closed
operator algebra generated by the spatial tensor product) with the
subalgebra
 $A_\cE({\bf D}_{f,\text{\rm rad}}^m)$.
Since the proofs of these results are very similar to those of
Theorem \ref{f-infty}  and Theorem \ref{A-infty}, we shall omit
them.

 \bigskip

\section{ Compositions of   free holomorphic functions}

In this section we present several results concerning the
composition of free holomorphic functions on noncommutative domains
${\bf D}_{f,\text{\rm rad}}^m$. These results are used, throughout
this paper,  to study free biholomorphic functions.

Let $f$ be a positive regular  free holomorphic function with $n$
indeterminates and
 consider a free holomorphic function $\Phi:{\bf D}_{f,\text{\rm rad}}^m(\cH)\to [B(\cH)
\bar\otimes_{min} B(\cY)]^p$.  Then we have $\Phi=(\Phi_1,\ldots,
\Phi_p)$, where each mapping $\Phi_j:{\bf D}_{f,\text{\rm
rad}}^m(\cH)\to B(\cH)\bar \otimes_{min} B(\cY)$  is a  free
holomorphic function with  representation
$$
\Phi_j(X)=\sum_{k=0}^\infty \sum_{\alpha\in \FF_n^+,|\alpha|=k}
X_\alpha\otimes B_{(\alpha)}^{(j)}, \qquad X:=(X_1,\ldots, X_n)\in
{\bf D}_{f,\text{\rm rad}}^m(\cH),
$$
for some $B_{(\alpha)}^{(j)}\in B(\cY)$,   $\alpha\in \FF_n^+$.
Assume  now that
$$ \Phi(X)\in{\bf D}_{g,\text{\rm rad}}^l(\cH\otimes \cY) \quad \text{ for any } \  X\in
 {\bf D}_{f,\text{\rm rad}}^m(\cH),
$$
where $g$ is  a positive regular free holomorphic function with $p$
indeterminates, and  $l\geq 1$. Consider a free holomorphic function
$F:{\bf D}_{g,\text{\rm rad}}^l(\cK)\to B(\cK)\bar \otimes_{min}
B(\cE)$ with standard representation
$$
F(Y_1,\ldots, Y_p):= \sum_{k=0}^\infty \sum_{\alpha\in
\FF_p^+,|\alpha|=k} Y_\alpha\otimes A_{(\alpha)}, \qquad
(Y_1,\ldots, Y_p)\in{\bf D}_{g,\text{\rm rad}}^l(\cK),
$$
for some  bounded operators $A_{(\alpha)}\in B(\cE)$, $\alpha\in
\FF_p^+$. Note that  it makes sense to define the mapping  $F\circ
\Phi:{\bf D}_{f,\text{\rm rad}}^m(\cH)\to B(\cH)\bar \otimes_{min}
B(\cY)\bar\otimes_{min} B(\cE)$ by setting
$$
(F\circ \Phi)(X_1,\ldots, X_n):=\sum_{k=0}^\infty \sum_{\alpha\in
\FF_p^+,|\alpha|=k} \Phi_\alpha(X_1,\ldots, X_n)\otimes
A_{(\alpha)}, \qquad (X_1,\ldots, X_n)\in {\bf D}_{f,\text{\rm
rad}}^m(\cH),
$$
where  the convergence is in the operator norm topology. In what
follows, we   prove
 that  $F\circ \Phi$ is a free holomorphic
function on ${\bf D}_{f,\text{\rm rad}}^m(\cH)$ with standard
representation
$$
(F\circ \Phi)(X_1,\ldots, X_n)=\sum_{k=0}^\infty \sum_{\sigma\in
\FF_n^+,|\sigma|=k} X_\sigma\otimes C_{(\sigma)}, \qquad
(X_1,\ldots, X_n)\in {\bf D}_{f,\text{\rm rad}}^m(\cH),
$$
where $C_{(\sigma)}\in B(\cY\otimes \cE)$ is defined by
$$
\left< C_{(\sigma)}x,y\right>=
\frac{b_\sigma^{(m)}}{r^{|\sigma|}}\left<( W_\sigma^*\otimes
I_{\cY\otimes \cE})  (F\circ \Phi)(rW_1,\ldots, rW_n)(1\otimes x),
1\otimes y\right>
$$
for any $\sigma\in \FF_n^+$ and  $x,y\in \cY\otimes \cE$.

First, we need the following result.

\begin{lemma} \label{range} Let $f$ and $g$ be positive regular free holomorphic functions
with $n$ and $p$ indeterminates, respectively, and let $m,l\geq 1$.
If $\Phi:{\bf D}_{f,\text{\rm rad}}^m(\cH)\to
[B(\cH)\bar\otimes_{min} B(\cY)]^p$ is a free holomorphic function,
then $\text{\rm range}\, \Phi\subseteq {\bf D}_{g,\text{\rm
rad}}^l(\cH\otimes \cY)$ if and only if
$$
\Phi(rW_1,\ldots, rW_n)\in {\bf D}_{g,\text{\rm
rad}}^l(F^2(H_n)\otimes \cY) \quad \text{ for any } \ r\in[0,1),
$$
where $(W_1,\ldots,
W_n)$ is the universal model associated with the noncommutative
domain ${\bf D}_f^m$.
\end{lemma}

\begin{proof}
Since $\cH$ is a separable infinite dimensional Hilbert space and
$(rW_1,\ldots, rW_n)\in {\bf D}_{f,\text{\rm rad}}^m (F^2(H_n))$ the
direct implication is obvious. To prove the converse, assume that
$\Phi=(\Phi_1, \ldots, \Phi_p)$ is a free holomorphic function on
${\bf D}_{f,\text{\rm rad}}^m(\cH)$ and $\Phi(rW_1,\ldots,
rW_n)\in{\bf D}_{g,\text{\rm rad}}^l(F^2(H_n)\otimes \cY)$ for any
$r\in(0,1)$. Let $X:=(X_1,\ldots, X_n)$ be in ${\bf D}_{f,\text{\rm
rad}}^m(\cH)$. Then there exists $\gamma \in (0,1)$ such that $ X
\in {\bf D}_{f,\gamma}^m(\cH)$.  Since $\Phi$ is a free holomorphic
function on ${\bf D}_{f,\text{\rm rad}}^m(\cH)$,  for each
$j=1,\ldots,p$, the operator $\Phi_j(\gamma W_1,\ldots, \gamma W_n)$
is in $ \cA_n({\bf D}_f^m)\bar\otimes_{min} B(\cY)$  and there is
$s\in (0,1)$ such that
\begin{equation}
\label{fract} \frac{1}{s} \Phi(\gamma W):=\left( \frac{1}{s}
\Phi_1(\gamma W_1,\ldots, \gamma W_n),\ldots,
 \frac{1}{s} \Phi_p(\gamma W_1,\ldots, \gamma W_n)\right)\in {\bf D}_g^l (F^2(H_n)\otimes
 \cY).
 \end{equation}
 This implies $(id-\Phi_{g,\frac{1}{s} \Phi(\gamma W)})^k(I)\geq 0$   for  $ 1\leq k\leq
 l$.
Since $\Phi=(\Phi_1, \ldots, \Phi_p)$ is a free holomorphic function
on ${\bf D}_{f,\text{\rm rad}}^m(\cH)$, each $\Phi_j(X_1,\ldots,
X_n)$ is given by a series which is convergent in the operator
topology and so is $\Phi_j(\gamma W)\in \cA_n({\bf D}_f^m)$.
Therefore, using the properties of the noncommutative Berezin
transform, we  have
$$
\Phi_\alpha (X )\Phi_\alpha (X )^*={\bf
B}_{\frac{1}{\gamma}X}[\Phi_\alpha (\gamma W )\Phi_\alpha (\gamma W
)^*]
$$
for any  $ X \in {\bf D}_{f,\gamma}^m(\cH)$ and $\alpha\in \FF_n^+$.
Since, due to \eqref{fract},  $(id-\Phi_{g,\frac{1}{s} \Phi(\gamma
W)})^k(I)\leq I$ for $1\leq k\leq l$,  and using the fact that the
noncommutative Berezin transform is a completely positive map which
is SOT-continuous on bounded sets, we deduce that
$$
(id-\Phi_{g,\frac{1}{s} \Phi(X)})^k(I)={\bf
B}_{\frac{1}{\gamma}X}\left[(id-\Phi_{g,\frac{1}{s} \Phi(\gamma
W)})^k(I)\right]\geq 0,\qquad 1\leq k\leq l.
$$
 Hence,  we infer that  $\left( \frac{1}{s}
\varphi_1(X_1,\ldots, X_n),\ldots,
 \frac{1}{s} \varphi_p(X_1,\ldots, X_n)\right)\in {\bf D}_g^l(\cH\otimes \cY)$, which shows that the $n$-tuple
 $\left(  \varphi_1(X_1,\ldots, X_n),\ldots,
  \varphi_p(X_1,\ldots, X_n)\right)$ is in $ {\bf D}_{g,
  \text{\rm rad}}^l(\cH\otimes \cY)$.
This completes the proof.
\end{proof}

The next result shows that the composition of free holomorphic
functions  with operator-valued coefficients is a free holomorphic function.

\begin{theorem}
\label{compo} Let $f$ and $g$ be positive regular free holomorphic
functions with $n$ and $p$ indeterminates, respectively, and let
$m,l\geq 1$. Let $F:{\bf D}_{g,\text{\rm rad}}^l(\cK)\to B(\cK)\bar
\otimes B(\cE)$ and $\Phi: {\bf D}_{f,\text{\rm rad}}^m(\cH)\to
[B(\cH)\bar\otimes_{min}B(\cY)]^p$ be free holomorphic functions
such that $\text{\rm range}\, \Phi \subseteq {\bf D}_{g,\text{\rm
rad}}^l(\cH\otimes \cY)$.  Then $F\circ \Phi$ is a free holomorphic
function on ${\bf D}_{f,\text{\rm rad}}^m(\cH)$. If, in addition,
$F$ is bounded, then $F\circ \Phi$ is bounded and $\|F\circ
\Phi\|_\infty\leq \|F\|_\infty$.
\end{theorem}
\begin{proof}
Suppose that  $F$ has the  standard representation
\begin{equation*}
 F(Y_1,\ldots, Y_p)=\sum_{k=0}^\infty \sum_{\beta\in
\FF_p^+, |\beta|=k}
 Y_\alpha\otimes A_{(\beta)}, \qquad (Y_1,\ldots, Y_p)\in
{\bf
D}_{g,\text{\rm rad}}^l(\cK),
\end{equation*}
where the  convergence is  in the operator norm topology. Let
$\Phi=(\Phi_1,\ldots, \Phi_p)$, where $\Phi_1,\ldots, \Phi_p$ are
free holomorphic functions on ${\bf D}_{f,\text{\rm rad}}^m(\cH)$
with coefficients in $B(\cY)$, and  the property  that $\text{\rm
range}\, \Phi \subseteq {\bf D}_{g,\text{\rm rad}}^l(\cH\otimes
\cY)$.
 Hence, we deduce that
$$
(F\circ \Phi)(X_1,\ldots, X_n)=\sum_{k=0}^\infty \sum_{\beta\in
\FF_p^+, |\beta|=k}  \Phi_\beta(X_1,\ldots, X_n)\otimes
A_{(\beta)},
$$
where   $\Phi_\beta:=\Phi_{i_1}\cdots \Phi_{i_k}$ if
$\beta=g_{i_1}\cdots g_{i_k}\in \FF_p^+$, and
 the  convergence is  in the operator norm topology for any
$(X_1,\ldots, X_n)\in {\bf D}_{f,\text{\rm rad}}^m(\cH)$. According
to Lemma \ref{range}, we have $\Phi(rW_1,\ldots, rW_n)\in {\bf
D}_{f,\text{\rm rad}}^m(F^2(H_n)\otimes \cY)$ for  $r\in [0,1)$.
Taking into account that $F$ is free holomorphic  function on ${\bf
D}_{f,\text{\rm rad}}^m(\cK)$,  we deduce that, for each $r\in
[0,1)$,
\begin{equation}
\label{Mr}
 Q_r:=\sum_{k=0}^\infty \sum_{\beta\in
\FF_p^+, |\beta|=k}
 \Phi_\beta(rW_1,\ldots, rW_n)\otimes A_{(\beta)}
\end{equation}
is convergent in the operator norm topology. Since
$\Phi_i(rW_1,\ldots, rW_n)$ is in the tensor algebra $\cA_n({\bf
D}_f^m)\bar\otimes_{min} B(\cY)$ for each $i=1,\ldots, p$, the
operator $Q_r$ is in the operator algebra $$\cA_n({\bf
D}_f^m)\bar\otimes_{min} B(\cY\otimes\cE)\subset F_n^\infty({\bf
D}_f^m)\bar \otimes B(\cY\otimes\cE).$$ Therefore,
  $M_r$ has a   Fourier
representation  $\sum_{k=0}^\infty \sum_{\alpha\in \FF_n^+,
|\alpha|=k} r^{|\alpha|} W_\alpha\otimes B_{(\alpha)}(r)$, where $B_{(\alpha)}(r)\in
B(\cY\otimes \cE)$, $\alpha\in \FF_n^+$,  and
\begin{equation}
\label{rep2} Q_r=\text{\rm SOT-}\lim_{\gamma\to 1}\sum_{k=0}^\infty
\sum_{\alpha\in \FF_n^+, |\alpha|=k} r^{|\alpha|}\gamma^{|\alpha|}
W_\alpha\otimes B_{(\alpha)}(r),
\end{equation}
where the series $\sum_{k=0}^\infty \sum_{\alpha\in \FF_n^+,
|\alpha|=k}
 r^{|\alpha|}\gamma^{|\alpha|} W_\alpha\otimes B_{(\alpha)}(r)$
converges in the operator norm topology.

 Now, we show that the coefficients $B_{(\alpha)}(r)$, $\alpha\in \FF_n^+$,  don't depend on $r\in
[0,1)$. Taking into account  relations \eqref{Mr} and \eqref{rep2},
we  have
\begin{equation*}
\begin{split}
\left< B_{\alpha}(r)(y\otimes x),y'\otimes z\right>&= \left<( W_\alpha^*\otimes
I_{\cY\otimes \cE})\frac{b_\alpha^{(m)}}{r^{|\alpha|}} Q_r(1\otimes y\otimes x), 1\otimes y'\otimes  z\right>\\
&=\lim_{q\to\infty}\sum_{k=0}^q\sum_{\beta\in \FF_p^+,
|\beta|=k}\left< A_{(\beta)}
x,z\right>\left<\frac{b_\alpha^{(m)}}{r^{|\alpha|}}
(W_\alpha^*\otimes I_\cY)\Phi_\beta(rW_1,\ldots, rW_n)(1\otimes y),1\otimes y'\right>
\end{split}
\end{equation*}
for any $x, z\in \cE$,  $y,y'\in \cY$,  and $\alpha\in \FF_n^+$. On
the other hand, for each $\beta\in \FF_p^+$, $\Phi_\beta$  is a free
holomorphic function on ${\bf D}_{f,\text{\rm rad}}^m(\cH)$  with
coefficients  in $B(\cY)$ and has a representation
$\Phi_\beta(X_1,\ldots, X_n)=\sum_{k=0}^\infty \sum_{\alpha\in
\FF_n^+, |\alpha|=k} X_\alpha\otimes D_{(\alpha)}$ for $(X_1,\ldots,
X_n)\in {\bf D}_{f,\text{\rm rad}}^m(\cH)$,  where $D_{(\alpha)}\in
B(\cY)$.  This implies
$$
\left<\frac{1}{r^{|\alpha|}} (W_\alpha^*\otimes I_\cY)\Phi_\beta(rW_1,\ldots,
rW_n)(1\otimes y), 1\otimes y'\right>=\frac{1}{b_\alpha^{(m)}} \left< D_{(\alpha)}y,y'\right>
$$
for any $r\in [0,1)$, and $\alpha,\beta\in \FF_n^+$. Now, it is
clear that $B_{(\alpha)}:=B_{(\alpha)}(r)$ does not depend on $r\in
[0,1)$.  Now   relation \eqref{rep2} becomes
\begin{equation}\label{Mr2}
 Q_r=\text{\rm SOT-}\lim_{\gamma\to 1}\sum_{k=0}^\infty
\sum_{\alpha\in \FF_n^+, |\alpha|=k}  r^{|\alpha|}\gamma^{|\alpha|} W_\alpha \otimes
B_{(\alpha)},
\end{equation}
where the series $\sum_{k=0}^\infty \sum_{\alpha\in \FF_n^+,
|\alpha|=k} r^{|\alpha|}\gamma^{|\alpha|} W_\alpha\otimes
B_{(\alpha)}$ converges in the operator norm topology, for any $r,
\gamma\in[0,1)$. According to Theorem \ref{free-ho},
$$
G(X_1,\ldots, X_n):=\sum_{k=0}^\infty \sum_{\alpha\in \FF_n,
|\alpha|=k} X_\alpha\otimes B_{(\alpha)},\qquad (X_1,\ldots, X_n)\in
{\bf
D}_{f,\text{\rm rad}}^m(\cH),
$$
is a free holomorphic function on ${\bf D}_{f,\text{\rm
rad}}^m(\cH)$. Using the continuity of $G$ in the norm operator
topology and relations \eqref{Mr} and \eqref{Mr2}, we deduce that
\begin{equation}
\label{AB} Q_r:=\sum_{k=0}^\infty \sum_{\beta\in \FF_p^+, |\beta|=k}
 \Phi_\beta(rW_1,\ldots, rW_n)\otimes A_{(\beta)}=
\sum_{k=0}^\infty \sum_{\alpha\in \FF_n^+, |\alpha|=k} r^{|\alpha|}
W_\alpha\otimes B_{(\alpha)}
\end{equation}
for any $r\in [0,1)$.

Let  $X:=(X_1,\ldots, X_n)\in {\bf D}_{f,\text{\rm rad}}^m(\cH)$.
 Then there exists $\gamma \in (0,1)$
such that $ \frac{1}{\gamma}X \in {\bf D}_{f}^m(\cH)$.
 Applying  the noncommutative Berezin
transform  at $(\frac{1}{\gamma} X_1,\ldots, \frac{1}{\gamma}
X_n)\in {\bf D}_f^m(\cH)$ to relation \eqref{AB}, when $r:=\gamma$,
we deduce that
\begin{equation*}
 (F\circ \Phi)(X_1,\ldots, X_n)=\sum_{k=0}^\infty \sum_{\beta\in
\FF_p^+, |\beta|=k} \Phi_\beta(X_1,\ldots, X_n)\otimes
A_{(\beta)}=\sum_{k=0}^\infty \sum_{\alpha\in \FF_n^+, |\alpha|=k}
X_\alpha\otimes B_{(\alpha)}
\end{equation*}
for any $(X_1,\ldots, X_n)\in {\bf
D}_{f,\text{\rm rad}}^m(\cH)$. This completes
the proof.
\end{proof}

Now, we present  results concerning the composition of {\it bounded
free holomorphic functions}, with operator-valued coefficients, on
noncommutative domains ${\bf D}_{f\text{\rm rad}}^m$, $m\geq 1$.

\begin{theorem}
\label{more-prop} Let $f$ and $g$ be positive regular free
holomorphic functions with $n$ and $p$ indeterminates, respectively,  and let $m,l\geq 1$.
 Let  $F:{\bf
D}_{g,\text{\rm rad}}^l(\cK)\to B(\cK)\bar
\otimes_{min} B(\cE)$  and $\Phi: {\bf
D}_{f,\text{\rm rad}}^m(\cH)\to
[B(\cH)\bar \otimes_{min} B(\cY)]^p$  be   bounded free holomorphic
functions.
If $F$ and $\Phi$ have continuous extensions to the noncommutative domains ${\bf D}_g^l(\cK)$ and ${\bf D}_f^m(\cH)$, respectively, and $\text{\rm range}\, \Phi \subseteq
{\bf D}_g^l(\cH\otimes \cY)$, then $F\circ \Phi$ is a  bounded free holomorphic function  on
${\bf
D}_{f,\text{\rm rad}}^m(\cH)$ which has continuous extension to ${\bf D}_f^m(\cH)$.

Moreover,  we have
\begin{enumerate}
\item[(a)]
 $\|F\circ \Phi\|_\infty\leq
\|F\|_\infty$;
\item[(b)]
$ (F\circ \Phi)(X)=({\cB}_X\otimes I_{\cY\otimes \cE})\left\{
({\cB}_{\widetilde \Phi}\otimes I_\cE)[\widetilde F]\right\}$, \ $
X\in {\bf D}_{f}^m(\cH), $ where ${\cB}_X$, ${\cB}_{\widetilde \Phi}$ are the noncommutative Berezin transforms at
$X$ and  $\widetilde \Phi$, respectively;
\item[(c)] the model boundary function  of  the composition $F\circ\Phi$  satisfies the equation
$$
\widetilde {F\circ\Phi}= \lim_{r\to 1} F(r\widetilde
\Phi_1,\ldots, r\widetilde \Phi_p)=({\cB}_{\widetilde \Phi}\otimes
I_\cE)[\widetilde F],
$$
where  the convergence is in the operator norm topology and ${\cB}_{\widetilde \Phi}$ is the noncommutative Berezin transform
at $\widetilde \Phi$.
\end{enumerate}
\end{theorem}

\begin{proof}
Let $F$ have the representation
$$
F(Y_1,\ldots, Y_p):= \sum_{k=0}^\infty \sum_{\beta\in
\FF_p^+,|\beta|=k} Y_\beta\otimes A_{(\beta)}, \qquad (Y_1,\ldots,
Y_p)\in {\bf
D}_{g,\text{\rm rad}}^l(\cK),
$$
where the  series converges in the operator norm  topology. Since
$F$ and $\Phi$ have continuous extensions to the noncommutative domains ${\bf D}_g^l(\cK)$ and ${\bf D}_f^m(\cH)$, respectively, the operator-valued version of Theorem \ref{A-infty} implies
 that  $\widetilde F\in \cA_p({\bf D}_g^l)\bar
\otimes_{min} B(\cE)$ and $\widetilde \Phi:=(\widetilde
\Phi_1,\ldots, \widetilde \Phi_p)$ is
  such that   $\widetilde \Phi_j \in \cA_n({\bf D}_f^m)\bar \otimes_{min} B(\cY)$,
$j=1,\ldots, p$.
  Due to the properties  of the noncommutative Berezin  transform  $\cB_{\widetilde \Phi}$
 and  the fact that $(\widetilde \Phi_1,\ldots, \widetilde
\Phi_p)$ is in ${\bf D}_g^l(F^2(H_n)\otimes \cY)$,
$$\lim_{r\to 1} F(r\widetilde \Phi_1,\ldots,
r\widetilde \Phi_p))=\lim_{r\to 1}\sum_{k=0}^\infty \sum_{\beta\in
\FF_p^+,|\beta|=k} r^{|\beta|}\widetilde \Phi_\beta\otimes
A_{(\beta)}=({\cB}_{\widetilde \Phi}\otimes I_\cE)[\widetilde F]$$
 where the convergence is  in the  operator
norm topology. Hence and  due to the fact that ${\cA}_n({\bf
D}_f^m)\bar \otimes_{min} B(\cY\otimes \cE)$ is closed in the
operator norm, we deduce that $\lim_{r\to 1} F(r\widetilde
\Phi_1,\ldots, r\widetilde \Phi_p)$ is in $\cA_n({\bf D}_f^m)\bar
\otimes_{min} B(\cY\otimes \cE)$. If $X\in {\bf D}_{f}^m(\cH)$, then
using the continuity of $F, \Phi$, and  the continuity of the
noncommutative Berezin  transform ${\cB}_X$ in the operator norm
topology, we obtain
\begin{equation*}
\begin{split}
(F\circ \Phi)(X)&=\lim_{r\to 1} F(r\Phi_1(X),\ldots, r\Phi_p(X))\\
&=\lim_{r\to 1} \sum_{k=0}^\infty \sum_{\beta\in
\FF_p^+,|\beta|=k} r^{|\beta|}\Phi_\beta(X)\otimes A_{(\beta)}\\
&=\lim_{r\to 1}\sum_{k=0}^\infty \sum_{\beta\in
\FF_p^+,|\beta|=k} r^{|\beta|} ({\cB}_X\otimes I_\cY)[\widetilde\Phi_\beta]\otimes A_{(\beta)}\\
&= \lim_{r\to 1}({\cB}_X\otimes I_{\cY\otimes
\cE})\left(\sum_{k=0}^\infty \sum_{\beta\in
\FF_p^+,|\beta|=k} r^{|\beta|} \widetilde\Phi_\beta\otimes A_{(\beta)}\right)\\
&= ({\cB}_X\otimes I_{\cY\otimes \cE})\lim_{r\to
1}\left(\sum_{k=0}^\infty \sum_{\beta\in
\FF_p^+,|\beta|=k} r^{|\beta|} \widetilde\Phi_\beta\otimes A_{(\beta)}\right)\\
&=({\cB}_X \otimes I_{\cY\otimes \cE})\left\{( {\cB}_{\widetilde
\Phi}\otimes I_\cE)[\widetilde F]\right\},
\end{split}
\end{equation*}
which proves part (b).
Since $( {\cB}_{\widetilde \Phi}\otimes I_\cE)[\widetilde F]\in
\cA_n({\bf D}_f^m)\bar\otimes_{min} B(\cY\otimes \cE)$, the operator-valued version of Theorem \ref{A-infty} shows that
$F\circ \Phi$ is a free holomorphic  function on ${\bf
D}_{f,\text{\rm rad}}^m(\cH)$ and continuous on ${\bf D}_f^m(\cH)$ and
its model boundary function   satisfies the equation
$$
\widetilde {F\circ\Phi}= \lim_{r\to 1} F(r\widetilde
\Phi_1,\ldots, r\widetilde \Phi_p)=({\cB}_{\widetilde \Phi}\otimes
I_\cE)[\widetilde F],
$$
which proves part (c).  Part (a) is now obvious.
The proof is complete.
\end{proof}

We recall that an $n$-tuple  $T:=(T_1,\ldots, T_n)\in {\bf D}_{f
}^m(\cH)$  is  {\it pure} if ~$\text{\rm SOT-}\lim_{k\to\infty}
\Phi_{f,T}^k(I)=0$, where
 $\Phi_{f,T}:B(\cH)\to
B(\cH)$ is the  positive linear map defined  by $
\Phi_{f,T}(X)=\sum_{k=1}^\infty\sum_{|\alpha|=k} a_\alpha T_\alpha
XT_\alpha^*, $ where the convergence is in the weak operator
topology. We define
$${\bf
D}_{f,\text{\rm pure}}^m(\cH):=\{(X_1,\ldots, X_n)\in {\bf  D}_f^{m}(\cH):\
 (X_1,\ldots, X_n) \text{ is pure}\}.
$$
Note that ${\bf D}_{f,\text{\rm rad}}^m(\cH)\subset{\bf
D}_{f,\text{\rm pure}}^m(\cH)\subset {\bf D}_{f}^m(\cH)$.

\begin{lemma}\label{pure}  Let $f$ and $g$ be positive regular free
holomorphic functions with $n$ and $p$ indeterminates, respectively,
and let $m,l\geq 1$. Let $\psi_i\in H^\infty( {\bf D}_{f}^m)$,
$i=1,\ldots,p$, be such that
 $\widetilde \psi:= (\widetilde
\psi_1,\ldots, \widetilde \psi_p)$ is in  ${\bf D}_{g,\text{\rm
pure}}^l(F^2(H_n))$, where $\widetilde \psi_i\in F_n^\infty({\bf
D}_f^m)$  is the model boundary function. Then
$$
\psi(X):=(\psi_1(X),\ldots, \psi_p(X))\in {\bf D}_{g,\text{\rm
pure}}^l(\cH)
$$
for any $X:=(X_1,\ldots, X_n)\in {\bf D}_{f,\text{\rm
pure}}^m(\cH)$.
\end{lemma}
\begin{proof} Taking into account that
 $\widetilde \psi:=(\widetilde \psi_1,\ldots, \widetilde \psi_p)\in
{\bf
D}_{g,\text{\rm pure}}^l(F^2(H_n))$ and
 $X:=(X_1,\ldots, X_n)\in {\bf
D}_{f,\text{\rm pure}}^m(\cH)$,  and  using the properties of the
extended noncommutative Berezin transform, we deduce that
$$\psi_\alpha(X)\psi_\alpha(X)^*={K_{f, X}^{(m)}}^*
(\widetilde \psi_\alpha \widetilde\psi_\beta^*\otimes I_\cH)K_{f,
X}^{(m)},\qquad \alpha\in \FF_n^+.
$$
Hence, we obtain
$$
\Phi_{g, \psi(X)}^k(I)={K_{f, X}^{(m)}}^* ( \Phi_{g,\widetilde \psi}^k(I)\otimes
 I_\cH)K_{f, X}^{(m)}, \qquad  k\in \NN.
$$
 Since $\Phi_{g,\widetilde \psi}^k(I)\leq I$ for any $k\in \NN$ and
  $\text{\rm SOT-}\lim_{k\to\infty}
\Phi_{g,\widetilde \psi}^k(I)=0$, we deduce that
$$\text{\rm
SOT-}\lim_{k\to\infty} \Phi_{g, \psi(X)}^k(I)=0,
$$
which shows that
$\psi(X)$ is pure.
\end{proof}

\begin{theorem}
\label{more-prop2} Let $f$ and $g$ be positive regular free
holomorphic functions with $n$ and $p$ indeterminates, respectively,
and let $m,l\geq 1$.
 Let  $F:{\bf
D}_{g,\text{\rm rad}}^l(\cK)\to B(\cK)\bar \otimes_{min} B(\cE)$ and
$\Phi: {\bf D}_{f,\text{\rm rad}}^m(\cH)\to B(\cH)^p$ be bounded
free holomorphic functions. If the model boundary function
$\widetilde \Phi$ is a pure $n$-tuple in ${\bf D}_g^l(F^2(H_n))$,
then the map
$$
(F\circ \Phi)(X):=\text{\rm SOT-}\lim_{r\to 1} F(r\Phi_1(X),\ldots,
r\Phi_n(X)),\qquad X\in {\bf D}_{f,\text{\rm pure}}^m(\cH),
$$
is a bounded free holomorphic function on ${\bf
D}_{f,\text{\rm rad}}^m(\cH)$.
Moreover,  we have
\begin{enumerate}
\item[(a)]
 $\|F\circ \Phi\|_\infty\leq
\|F\|_\infty$;
\item[(b)]
$ (F\circ \Phi)(X)=({\cB}_X\otimes I_{\cE})\left\{
({\cB}_{\widetilde \Phi}\otimes I_\cE)[\widetilde F]\right\}$, \ $
X\in {\bf D}_{f,\text{\rm pure}}^m(\cH), $ where ${\cB}_X$,
${\cB}_{\widetilde \Phi}$ are the noncommutative Berezin transforms
at $X$ and  $\widetilde \Phi$, respectively;
\item[(c)] the model boundary function  of  the composition $F\circ\Phi$  satisfies the equation
$$
\widetilde {F\circ\Phi}=\text{\rm SOT-}\lim_{r\to 1} F(r\widetilde
\Phi_1,\ldots, r\widetilde \Phi_p)=({\cB}_{\widetilde \Phi}\otimes
I_\cE)[\widetilde F],
$$
where ${\cB}_{\widetilde \Phi}$ is the noncommutative Berezin
transform at $\widetilde \Phi$.
\end{enumerate}
\end{theorem}

\begin{proof}
Let $F$ have the representation
$$
F(Y_1,\ldots, Y_p):= \sum_{k=0}^\infty \sum_{\beta\in
\FF_p^+,|\beta|=k} Y_\beta\otimes A_{(\beta)}, \qquad (Y_1,\ldots,
Y_p)\in {\bf
D}_{g,\text{\rm rad}}^l(\cK),
$$
where the  series converges in the operator norm  topology. Since
$F$ and $\Phi$ are bounded free holomorphic functions on  the
 noncommutative domains ${\bf D}_{g, \text{\rm rad}}^l(\cK)$ and
  ${\bf D}_{f,\text{\rm rad}}^m(\cH)$, respectively, the operator-valued
   version of Theorem \ref{f-infty} implies
 that  $\widetilde F\in F_n^\infty({\bf D}_g^l)\bar
\otimes B(\cE)$ and $\widetilde \Phi:=(\widetilde \Phi_1,\ldots,
\widetilde \Phi_p)$ is
  such that   $\widetilde \Phi_j \in F_n^\infty({\bf D}_f^m)$,
$j=1,\ldots, p$.
  Due to  the properties  of the noncommutative Berezin  transform  $\cB_{\widetilde \Phi}$
 and  the fact that $(\widetilde \Phi_1,\ldots, \widetilde
\Phi_p)$ is a pure $n$-tuple in ${\bf D}_g^l(F^2(H_n))$, we have
\begin{equation}
\label{sot-tilde}\text{\rm SOT-}\lim_{r\to 1} F(r\widetilde
\Phi_1,\ldots, r\widetilde \Phi_p))=\text{\rm SOT-}\lim_{r\to
1}\sum_{k=0}^\infty \sum_{\beta\in \FF_p^+,|\beta|=k}
r^{|\beta|}\widetilde \Phi_\beta\otimes
A_{(\beta)}=({\cB}_{\widetilde \Phi}\otimes I_\cE)[\widetilde F].
\end{equation}
    Hence and  due to the fact that $F_n^\infty({\bf D}_f^m)\bar
    \otimes
B( \cE)$ is closed in the  strong operator  topology, we deduce that
$\text{\rm SOT-}\lim_{r\to 1} F(r\widetilde \Phi_1,\ldots,
r\widetilde \Phi_p)$ is in $F_n^\infty({\bf D}_f^m)\bar \otimes
B(\cE)$. If $X\in {\bf D}_{f,\text{\rm pure}}^m(\cH)$, then Lemma
\ref{pure} implies that $(\Phi_1(X),\ldots, \Phi_p(X))$ is a pure
$n$-tuple of operators in ${\bf D}_g^l(\cH)$. Consequently,  as in
the proof of Theorem \ref{more-prop}, using   the continuity of the
noncommutative Berezin transform ${\cB}_X$ in the  strong operator
topology (see \eqref{Be-transf}), and relation \eqref{sot-tilde},
we obtain
\begin{equation*}
\begin{split}
(F\circ \Phi)(X)&=\text{\rm SOT-}\lim_{r\to 1} F(r\Phi_1(X),\ldots, r\Phi_p(X))\\
&=\text{\rm SOT-}\lim_{r\to 1} \sum_{k=0}^\infty \sum_{\beta\in
\FF_p^+,|\beta|=k} r^{|\beta|}\Phi_\beta(X)\otimes A_{(\beta)}\\
&=({\cB}_X \otimes I_{ \cE})\left\{( {\cB}_{\widetilde \Phi}\otimes
I_\cE)[\widetilde F]\right\},
\end{split}
\end{equation*}
which proves part (b).
Since $( {\cB}_{\widetilde \Phi}\otimes I_\cE)[\widetilde F]\in
F^\infty_n({\bf D}_f^m)\bar\otimes_{min} B(\cE)$, the
operator-valued version of Theorem \ref{A-infty} shows that $F\circ
\Phi$ is a free holomorphic  function on ${\bf D}_{f,\text{\rm
rad}}^m(\cH)$  and its model boundary function   satisfies the
equation
$$
\widetilde {F\circ\Phi}=\text{\rm SOT-}\lim_{r\to 1} F(r\widetilde
\Phi_1,\ldots, r\widetilde \Phi_p)=({\bf B}_{\widetilde \Phi}\otimes
I_\cE)[\widetilde F],
$$
which proves part (c).  Part (a) is now obvious.
The proof is complete.
\end{proof}

\bigskip

\section{Free biholomorphic functions and noncommutative Cartan type results}

In this section we obtain  noncommutative   Cartan type results for
formal power series in several noncommuting   indeterminates, which
leave invariant the nilpotent parts of the corresponding domains.
These results are used to
    characterize  the set   of all
   free biholomorphic functions $F:{\bf D}_f^m(\cH)\to{\bf D}_g^l(\cH)$
with $F(0)=0$.
  As a consequence, we  determine  the free holomorphic
automorphisms of the noncommutative domain  ${\bf D}_{f,\text{\rm
rad}}^m$ which  fix the origin. Several other  consequences
concerning the biholomorphic classification of  noncommutative
domains are obtained.

An $n$-tuple of operators $(N_1,\ldots, N_n)\in B(\cH)^n$ is called
nilpotent if there is  $p\in \NN:=\{1, 2,\ldots\}$  such that
$N_\alpha=0$ for any $\alpha\in \FF_n^+$ with $|\alpha|=p$. Define
the nilpotent part of the noncommutative domain  ${\bf D}_f^m(\cH)$
by setting
 $$
  {\bf D}_{f,\text{\rm nil}}^m(\cH):=\{(N_1,\ldots, N_n)\in {\bf D}_f^m(\cH):
  \ (N_1,\ldots, N_n) \text{ is nilpotent} \}.
  $$

For simplicity, throughout this paper, $[X_1,\ldots, X_n]$ denotes
either the $n$-tuple $(X_1,\ldots, X_n)\in B(\cH)^n$ or the operator
row matrix $[ X_1\, \cdots \, X_n]$ acting from $\cH^{(n)}$, the
direct sum  of $n$ copies of  a Hilbert space $\cH$, to $\cH$.

 \begin{theorem} \label{cartan1} Let $f$   be  a positive regular free
holomorphic function with $n$  indeterminates  and let $m \geq 1$.
  Let $H_1,\ldots, H_n$ be  formal  power series in $n$ noncommuting indeterminates
   $Z_1,\ldots, Z_n$
   of the form
   $$
 H_i(Z_1,\ldots, Z_n):=\sum_{k=2}\sum_{|\alpha|=k} a_\alpha^{(i)}
 Z_\alpha,\qquad a_\alpha^{(i)}\in \CC,\  i=1,\ldots,n.
 $$
 If
 $$
 F(Z_1,\ldots, Z_n):=\left(Z_1+H_1(Z_1,\ldots, Z_n),\ldots, Z_n+ H_n(Z_1,\ldots, Z_n)\right)
 $$
 has the property  that
  $$F({\bf D}_{f,\text{\rm nil}}^m(\cH))\subseteq {\bf D}_{f,\text{\rm nil}}^m(\cH)
  $$
  for any Hilbert space $\cH$,
   then
 $$F(Z_1,\ldots, Z_n)=(Z_1,\ldots, Z_n).
 $$
\end{theorem}

\begin{proof}
Assume that there exists $\alpha\in \FF_n^+$, $|\alpha|\geq2$, and
$i\in \{1,\ldots, n\}$ such that $a_\alpha^{(i)}\neq 0$ . Let $p\geq
2$ be the smallest natural number such that there exists
$\alpha_0\in \FF_n^+$, $|\alpha_0|=p$, and $i_0\in \{1,\ldots, n\}$
such that $a_{\alpha_0}^{(i_0)}\neq 0$. Then $F$ has the form
$F=(F_1,\ldots, F_n)$, where, for  each  $i=1,\ldots, n$,
$$
F_i(Z_1,\ldots, Z_n)=Z_i+ H_i(Z_1,\ldots, Z_n) \quad \text{ and }
\quad H_i(Z_1,\ldots, Z_n)= \sum_{k=p}^\infty G_k^{(i)}(Z_1,\ldots,
Z_n)
$$
 where
$G_k^{(i)}(Z_1,\ldots, Z_n):= \sum_{|\alpha|=k} a_\alpha^{(i)}
Z_\alpha$ for each $k\geq p$ and $i\in \{1,\ldots, n\}$.
 Note that, for each $i\in \{1,\ldots, n\}$,
$G_p^{(i)}\circ F$ is a power series and
$$
(G_p^{(i)}\circ F)(Z_1,\ldots, Z_n)=\sum_{|\alpha|=p} a_\alpha^{(i)}
F_\alpha (Z_1,\ldots, Z_n)=G_p^{(i)}(Z_1,\ldots, Z_n)+
K_{p+1}^{(i)}(Z_1,\ldots, Z_n),
$$
where $K_{p+1}^{(i)}$ is a power series containing only monomials of
degree $\geq p+1$  in its representation.  Therefore,  $F\circ F$ is
a power series and
\begin{equation*}
\begin{split}
(F\circ F)(Z_1,\ldots, Z_n) &=(Z_1,\ldots,
Z_n)+\left(2G_p^{(1)}(Z_1,\ldots, Z_n),\ldots,2G_p^{(n)}(Z_1,\ldots, Z_n)\right)\\
&\qquad+ \left(K_{p+1}^{(1)}(Z_1,\ldots, Z_n),\ldots,
K_{p+1}^{(n)}(Z_1,\ldots, Z_n)\right).
\end{split}
\end{equation*}
Iterating this process  and setting $F^N:=\underbrace{(F\circ\cdots
\circ F)}_{\text{$N$}\ times}$,  $N\in \NN$,  we deduce  that
\begin{equation*}
\begin{split}
 F^N(Z_1,\ldots, Z_n) &=(Z_1,\ldots,
Z_n)+\left(NG_p^{(1)}(Z_1,\ldots, Z_n),\ldots,NG_p^{(n)}(Z_1,\ldots, Z_n)\right)\\
&\qquad+ \left(E_{p+1}^{(1)}(Z_1,\ldots, Z_n),\ldots,
E_{p+1}^{(n)}(Z_1,\ldots, Z_n)\right),
\end{split}
\end{equation*}
where, for each $i=1,\ldots, n$, $E_{p+1}^{(i)}$  is a power series
containing only monomials of degree $\geq p+1$ in its
representation. Let $\cM\subset F^2(\cH_n)$ be the linear span
generated by $e_\alpha$, where $\alpha\in \FF_n^+$ with
$|\alpha|\leq p$. Note that $\cM$ is an invariant subspace under
$W_1^*,\ldots, W_n^*$. Using the definition of the weighted shifts
$W_1,\ldots, W_n$ associated with the noncommutative domain ${\bf
D}_f^m$, we deduce that
\begin{equation}
\label{Fk} [F^N(P_\cM W_1|_\cM,\ldots, P_\cM W_n|_\cM)]^* e_\alpha =
\left[\begin{matrix} W_1^*\\\vdots
\\W_n^*\end{matrix}\right]e_\alpha+
N\left[\begin{matrix} [G_p^{(1)}(P_\cM W_1|_\cM,\ldots, P_\cM W_n|_\cM)]^*\\\vdots\\
[G_p^{(n)}(P_\cM W_1|_\cM,\ldots, P_\cM
W_n|_\cM)]^*\end{matrix}\right]e_\alpha
\end{equation}
for any $\alpha\in \FF_n^+$ with $|\alpha|\geq 2$, and $N=1,2,
\ldots$, where $P_\cM$ is the orthogonal projection from $F^2(H_n)$
onto $\cM$. We recall that $\{e_\alpha\}_{\alpha\in \FF_n^+}$ is the
standard orthonormal basis for $F^2(H_n)$.
 Since, for any $\alpha\in \FF_n^+$, $|\alpha|=p$,
 $$W_\alpha^* e_{\alpha_0} =\begin{cases}
\frac {{1}}{\sqrt{b_{\alpha_0}^{(m)}}} 1& \text{ if }
\alpha=\alpha_0 \\
0& \text{ otherwise }
\end{cases}
 $$
 and
$G_p^{(i)}$, $i=1,\ldots, n$, are homogeneous  noncommutative
polynomials of degree $p$ and $|\alpha_0|=p$, we have
\begin{equation}\label{gam}
\gamma:=\left\|\left[\begin{matrix} [G_p^{(1)}(P_\cM W_1|_\cM,\ldots, P_\cM W_n|_\cM)]^*\\\vdots\\
[G_p^{(n)}(P_\cM W_1|_\cM,\ldots, P_\cM
W_n|_\cM)]^*\end{matrix}\right]e_{\alpha_0}\right\|\geq
\frac{|a_{\alpha_0}^{(i_0)}|}{\sqrt{b_{\alpha_0}^{(m)}}}>0.
\end{equation}
Assume that $\alpha_0=g_j \gamma_0$, where $j=1,\ldots, n$ and
$\gamma_0\in \FF_n^+$, $|\gamma_0|=p-1$. Then, due to relation
\eqref{WbWb}, we have
$W_j^*e_{\alpha_0}=\frac{\sqrt{b_{\gamma_0}^{(m)}}}{\sqrt{b_{\alpha_0}^{(m)}}}
e_{\gamma_0}$ and $W_i^*e_{\alpha_0}=0$ for $i\neq j$. Hence and
using  relation  \eqref{Fk},  we deduce that
\begin{equation}
\label{NC}
  N \gamma< \frac{\sqrt{b_{\gamma_0}^{(m)}}}{\sqrt{b_{\alpha_0}^{(m)}}}+
  \|F^N(P_\cM W_1|_\cM,\ldots, P_\cM W_n|_\cM)^*e_{\alpha_0}\|\quad
\text{ for any } \ N\in \NN.
\end{equation}
On the other hand, since $(P_\cM W_1|_\cM,\ldots, P_\cM W_n|_\cM)$
is a nilpotent $n$-tuple in ${\bf D}_f^m(\cM)$ and due to the
hypothesis, we have $ F^N(P_\cM W_1|_\cM,\ldots, P_\cM W_n|_\cM)
\subseteq {\bf D}_{f,\text{\rm nil}}^m(\cM)$. Applying  inequality
\eqref{bound}  we obtain
 $$\|F^N(P_\cM W_1|_\cM,\ldots, P_\cM W_n|_\cM)\|<\frac{1}{\min\{\sqrt{a_\alpha}: \ |\alpha|=1\}}$$
  for any $N\in \NN$, where $a_\alpha$ are the coefficients of
   $f=\sum_{|\alpha|\geq 1} a_\alpha X_\alpha$.
   We recall that $a_\alpha>0$ if $|\alpha|=1$.  Hence and using  relation \eqref{NC}, we
deduce that
$$ N \gamma<\frac{\sqrt{b_{\gamma_0}^{(m)}}}{\sqrt{b_{\alpha_0}^{(m)}}}+
\frac{1}{\min\{\sqrt{a_\alpha}: \ |\alpha|=1\}}$$ for any $N\in
\NN$, which, due to \eqref{gam}, is a contradiction. This completes
the proof.
\end{proof}

In what follows, if $L:=[a_{ij}]_{n\times n}$ is a  bounded linear
operator on $\CC^n$,  we use the matrix notation
$$
 [X_1,\ldots, X_n]{
L}:=\left(\sum_{i=1}^n a_{i1} X_i,\cdots,\sum_{i=1}^n a_{in}
X_i\right).
$$

\begin{theorem}\label{Cartan2} Let $f$ and $g$ be positive regular free
holomorphic functions with $n$  indeterminates  and let $m,l\geq 1$.
 Let $F=(F_1,\ldots, F_n)$ and  $G=(G_1,\ldots, G_n)$ be  $n$-tuples of  formal power series in $n$ noncommuting indeterminates such that
$$
F(0)=G(0)=0\quad \text{ and } \quad F\circ G=G\circ F=\text{\rm id}.
$$
If  $ F({\bf D}_{f,\text{\rm nil}}^m(\cH))={\bf D}_{g,\text{\rm
nil}}^l(\cH)$ for any Hilbert space $\cH$, then  $F$  has the form
$$ F(Z_1,\ldots, Z_n)=[Z_1,\ldots, Z_n]U,
$$
where $U$ is  an invertible bounded linear operator  on $\CC^n$.
\end{theorem}

\begin{proof}
Since $F(0)=0$, $F$ has the representation $F=(F_1,\ldots, F_n)$,
where  each $F_j$ is a power series with scalar coefficients, having
the form
\begin{equation}
\label{Fj} F_j(Z_1,\ldots, Z_n)=\sum_{k=1}^n a_{kj}Z_k+
\Psi_2^{(j)}(Z_1,\ldots, Z_n),
\end{equation}
and $\Psi_2^{(j)}$ is a power series of the form
$\Psi_2^{(j)}(Z_1,\ldots, Z_n)=\sum_{p=2}^\infty \sum_{|\alpha|=p}
a_\alpha^{(j)} Z_\alpha$, $a_\alpha^{(j)}\in \CC$.
 Similarly, $G=(G_1,\ldots, G_n)$, where  each $G_j$ is a power series,
having the form
\begin{equation}
\label{Gj} G_j(Z_1,\ldots, Z_n)=\sum_{k=1}^n b_{kj}Z_k+
\Gamma_2^{(j)}(Z_1,\ldots, Z_n),
\end{equation}
and $\Gamma_2^{(j)}$ is a power series of the form
$\Gamma_2^{(j)}(Z_1,\ldots, Z_n)=\sum_{p=2}^\infty \sum_{|\alpha|=p}
b_\alpha^{(j)} Z_\alpha$, $b_\alpha^{(j)}\in \CC$. Consider the
matrices $U:=[a_{ij}]_{n\times n}$ and $B:=[b_{ij}]_{n\times n}$.
 Using the representations \eqref{Fj} and
\eqref{Gj}, we deduce that
\begin{equation*}
\begin{split}
&(G\circ F)(Z_1,\ldots, Z_n)\\
 & =\left(\sum_{j=1}^n b_{j1}
F_j+\Gamma_2^{(1)}(F_1,\ldots, F_n),\ldots, \sum_{j=1}^n b_{jn}
F_j+\Gamma_2^{(n)}(F_1,\ldots, F_n)\right)\\
&=\left(\sum_{j=1}^n b_{j1} \left(\sum_{k=1}^n a_{kj}
Z_k\right),\ldots, \sum_{j=1}^n b_{jn} \left(\sum_{k=1}^n a_{kj}
Z_k\right)\right) \\
&\quad
 +\left(\sum_{j=1}^n b_{j1} \Psi_2^{(j)}(Z_1,\ldots, Z_n),\ldots, \sum_{j=1}^n b_{jn}
  \Psi_2^{(j)}(Z_1,\ldots, Z_n)\right)
+ \left(\Gamma_2^{(1)}(F_1,\ldots,
F_n),\ldots,\Gamma_2^{(n)}(F_1,\ldots, F_n)\right)\\
&=[Z_1,\ldots, Z_n]UB + \left(\Lambda_2^{(1)}(Z_1,\ldots,
Z_n),\ldots,\Lambda_2^{(n)}(Z_1,\ldots, Z_n)\right),
\end{split}
\end{equation*}
where $\Lambda_2^{(j)}$, $j=1,\ldots, n$, are power series
containing only monomials of degree $\geq 2$ in their
representations. Since $(G\circ F)(Z_1,\ldots,
Z_n)=(Z_1,\ldots,Z_n)$, we  deduce that $UB=I_n$ and
$\Lambda_2^{(j)}(Z_1,\ldots, Z_n)=0$ for $j=1,\ldots, n$. Similarly,
since $(F\circ G)(Z_1,\ldots, Z_n)=(Z_1,\ldots,Z_n)$, we can prove
that $BU=I_n$. Therefore $U$ is an invertible operator on $\CC^n$.
Note that
$$
F_\theta(Z_1,\ldots, Z_n):= e^{-i\theta}F(e^{i\theta} Z_1, \ldots,
e^{i\theta} Z_n)
$$
is a power series, for all $\theta\in\RR$.  Due to relation
\eqref{Fj}, we  deduce that
\begin{equation*}
\begin{split}
e^{-i\theta}F(e^{i\theta} Z_1,& \ldots, e^{i\theta}
Z_n)\\
&=\left(\sum_{k=1}^n a_{k1}Z_k+
e^{-i\theta}\Psi_2^{(1)}(e^{i\theta}Z_1,\ldots,e^{i\theta}
Z_n),\ldots, \sum_{k=1}^n a_{kn}Z_k+
e^{-i\theta}\Psi_2^{(n)}(e^{i\theta}Z_1,\ldots,
e^{i\theta}Z_n)\right)
\end{split}
\end{equation*}
 Note that
\begin{equation}
\label{HGF} H(Z_1,\ldots, Z_n):=G\left(e^{-i\theta}F(e^{i\theta}
Z_1, \ldots, e^{i\theta} Z_n)\right)
\end{equation}
is a power series  with $H(0)=0$. Taking into account the
representations of the  power series involved in the definition of
$H$, calculations as above  lead to
$$
H(Z_1,\ldots, Z_n)=[Z_1,\ldots, Z_n]UB +
\left(\Phi_2^{(1)}(Z_1,\ldots, Z_n),\ldots,\Phi_2^{(n)}(Z_1,\ldots,
Z_n)\right),
$$
where $\Phi_2^{(j)}$, $j=1,\ldots, n$, are power series  containing
only monomials of degree $\geq 2$ in their representations. Note
that if $(N_1,\ldots, N_n)\in {\bf D}_{f,\text{\rm nil}}^m(\cH)$,
then $(e^{i\theta}N_1,\ldots, e^{i\theta}N_n)\in {\bf
D}_{f,\text{\rm nil}}^m(\cH)$ for  $\theta\in \RR$, and due to
relation $ F({\bf D}_{f,\text{\rm nil}}^m(\cH))={\bf D}_{g,\text{\rm
nil}}^l(\cH)$, we deduce that $F(e^{i\theta}N_1,\ldots,
e^{i\theta}N_n)\in {\bf D}_{g,\text{\rm nil}}^l(\cH)$. Consequently,
$e^{i\theta}F(e^{i\theta}N_1,\ldots, e^{i\theta}N_n)\in {\bf
D}_{g,\text{\rm nil}}^l(\cH)$ which together with relation $G({\bf
D}_{g,\text{\rm nil}}^l(\cH))={\bf D}_{f,\text{\rm nil}}^m(\cH)$
show that the  formal power series $H(Z_1,\ldots,
Z_n):=G\left(e^{-i\theta}F(e^{i\theta} Z_1, \ldots, e^{i\theta}
Z_n)\right)$ has the properties  that   $H(0)=0$ and $H({\bf
D}_{f,\text{\rm nil}}^m(\cH))\subseteq {\bf D}_{f,\text{\rm
nil}}^m(\cH)$.

 Now, since $UB=I_n$, we can
apply Theorem \ref{cartan1} to the power series $H$ to conclude that
$H(Z_1,\ldots, Z_n)=(Z_1,\ldots, Z_n)$. Hence, taking into account
that $G\circ F=id$ and due to relation \eqref{HGF}, we obtain
$$
e^{i\theta} F(Z_1,\ldots, Z_n)=F(e^{i\theta} Z_1, \ldots,
e^{i\theta} Z_n)
$$
for any $\theta\in
\RR$.  Using the representations given by \eqref{Fj},  the
latter equality implies
$$
a_{\alpha}^{(j)} e^{i\theta|\alpha|}=e^{i\theta} a_\alpha^{(j)}\quad
\text{ for any }\ \theta\in \RR,
$$
where $\alpha\in \FF_n^+$ with $|\alpha|\geq 2$, and $j=1,\ldots,n$.
Hence,  $a_\alpha^{(j)}=0$ and, consequently,
$$
F(Z_1,\ldots, Z_n)=\left(\sum_{k=1}^n a_{k1} Z_k,\ldots,
\sum_{k=1}^n a_{kn} Z_k\right)=[Z_1,\ldots, Z_n]U.
$$
The proof is complete.
\end{proof}

Let $f$ and $g$ be positive regular free holomorphic functions with
$n$ and $q$ indeterminates, respectively,  and let $m,l\geq 1$. A
map $F:{\bf D}_f^m(\cH)\to {\bf D}_g^l(\cH)$ is called free
biholomorphic function  if $F$ is a homeomorphism  in the operator
norm topology and $F|_{{\bf D}_{f,\text{\rm rad}}^m(\cH)}$ and
$F^{-1}|_{{\bf D}_{g,\text{\rm rad}}^l(\cH)}$ are   free holomorphic
functions on ${\bf D}_{f,\text{\rm rad}}^m(\cH)$ and ${\bf
D}_{g,\text{\rm rad}}^l(\cH)$, respectively. In this case,  the
domains ${\bf D}_f^m(\cH)$ and  ${\bf D}_g^l(\cH)$ are called free
biholomorphic equivalent. We denote by  $Bih({\bf D}_f^m, {\bf
D}_g^l)$
 the set of all the free biholomorphic functions $F:{\bf D}_f^m(\cH)\to{\bf D}_g^l(\cH)$.

As in the particular case $m=l=1$ (see \cite{Po-domains},
\cite{ArL}), there is  an important connection between the theory of
free biholomorphic functions on noncommutative domains  and the
theory of biholomorphic functions on domains in $\CC^d$(\cite{Kr}).
Let $\Omega_1$, $\Omega_2$ be domains (open and connected sets) in
$\CC^d$. If there exist free holomorphic maps $\varphi:\Omega_1\to
\Omega_2$ and $\psi:\Omega_2\to \Omega_1$ such that
$\varphi\circ\psi=id_{\Omega_2}$ and $\psi\circ
\varphi=id_{\Omega_1}$, then $\Omega_1$ and $\Omega_2$ are called
biholomorphic equivalent  and $\varphi$ and $\psi$ are called
biholomorphic maps.

\begin{theorem} \label{rep-finite3} Let $f$ and $g$ be positive regular free
holomorphic functions with $n$ and $q$ indeterminates, respectively,  and let $m,l,p\geq 1$. If
$F:{\bf D}_{f}^m(\cH)\to {\bf D}_{g}^l(\cH)$  is a free
biholomorphic function, then $n=q$ and
  its representation on $\CC^p$,  i.e., the map
  $F_p$ defined by
  $$
  \CC^{np^2}\supset {\bf
D}_{f}^m(\CC^p)\ni (\Lambda_1,\ldots, \Lambda_n)\mapsto
F(\Lambda_1,\ldots, \Lambda_n)\in {\bf D}_{g}^l(\CC^p) \subset
\CC^{qp^2}
$$
is a homeomorphism from   ${\bf D}_{f}^m(\CC^p)$ onto  ${\bf
D}_{g}^l(\CC^p)$ and   a biholomorphic function from $Int({\bf
D}_{f}^m(\CC^p))$ onto  $Int({\bf D}_{g}^l(\CC^p))$.
\end{theorem}
\begin{proof} The fact that $F_p$ is a homeomorphism from   ${\bf D}_{f}^m(\CC^p)$ onto  ${\bf
D}_{g}^l(\CC^p)$ is due to Theorem \ref{A-infty} and  Corollary
\ref{rep-finite2}. Moreover, from Corollary \ref{rep-finite2}, we
also have that $F_p$ and $(F^{-1})_p$  are  holomorphic functions on
$Int({\bf D}_{f}^m(\CC^p))\subset \CC^{np^2}$ and $Int({\bf D}_{g}^l(\CC^p))\subset \CC^{qp^2}$,
respectively. Now, since $F_p$ is a homeomorphism from   ${\bf
D}_{f}^m(\CC^p)$ onto  ${\bf D}_{g}^l(\CC^p)$, a standard argument
using Brouwer's invariance of domain theorem \cite{Bo}  shows that $F_p$ is a  biholomorphic function from $Int({\bf
D}_{f}^m(\CC^p))$ onto  $Int({\bf D}_{g}^l(\CC^p))$ and $n=q$.
This  completes the
proof.
\end{proof}

We remark that under the conditions of Theorem \ref{rep-finite3},
when $p=1$, the sets  $Int({\bf D}_{f}^m(\CC))\subset \CC^n$ and
$Int({\bf D}_{g}^l(\CC))\subset \CC^q$ are Reinhardt domains which
contain $0$. According to Sunada's result \cite{Su}, we deduce that
there exists a permutation $\sigma$ of the set $\{1,\ldots, n\}$ and
scalars $\mu_1,\ldots, \mu_n >0$ such that the map
$$
Int({\bf
D}_{f}^m(\CC))\ni(z_1,\ldots, z_n)\mapsto (\mu_1z_{\sigma(1)},\ldots, \mu_n z_{\sigma(n)})\in Int({\bf D}_{g}^l(\CC))
$$
is a biholomorphic map. It would be interesting to see  if there is an  analogue of Sunada's result for our noncommutative domains.

\begin{corollary}Let $f$ and $g$ be positive regular free
holomorphic functions with $n$ and $q$ indeterminates, respectively,  and let $m,l,p\geq 1$. If $n\neq q$ or there is $p\in \{1,2,\ldots\}$ such that  $Int({\bf
D}_{f}^m(\CC^p))$ is not biholomorphic equivalent  to  $Int({\bf D}_{g}^l(\CC^p))$, then the noncommutative domains ${\bf D}_{f}^m(\cH)$ and ${\bf D}_{g}^l(\cH)$  are not free
biholomorphic equivalent.
\end{corollary}

Let $f=\sum_{|\alpha|\geq 1}a_\alpha X_\alpha$ be a positive regular
free holomorphic function and fix $c_1,\ldots, c_n>0$. Define
$g:=\sum_{|\alpha|\geq 1} d_\alpha X_\alpha $ by setting
$d_\alpha:=\frac{1}{c^{2\alpha}}a_\alpha$ if $|\alpha|\geq 1$, where
$c:=(c_1,\ldots, c_n)$ and $c^\alpha:=c_{i_1}\cdots c_{i_k}$ when
$\alpha =g_{i_1}\cdots g_{i_k}\in \FF_n^+$. It is easy to see that
$g$ is also a positive regular free holomorphic function. Note that
the noncommutative domains ${\bf D}_f^m(\cH)$ and ${\bf D}_g^m(\cH)$
are free biholomorphic equivalent. Indeed, the map $\Lambda:{\bf
D}_f^m(\cH)\to {\bf D}_g^m(\cH)$ defined by $\Lambda(X_1,\ldots,
X_n):=(c_1X_1,\ldots, c_nX_n)$ is a free biholomorphic function.
Therefore, by rescaling  variables, we obtain free biholomorphic
equivalent domains.

In what follows, we characterize  the set $Bih_0({\bf D}_f^m, {\bf
D}_g^l)$ of all  free biholomorphic functions $F:{\bf
D}_f^m(\cH)\to{\bf D}_g^l(\cH)$ with $F(0)=0$. We use the notation
$(W_1^{(f)},\ldots, W_n^{(f)})$ for  the universal model associated
with the noncommutative domain ${\bf D}_f^m$. Here is the main
result of this section.

\begin{theorem}\label{Cartan3} Let $f$ and $g$ be positive regular free
holomorphic functions with $n$ and $q$  indeterminates, respectively,  and let $m,l\geq 1$.
A map $F:{\bf D}_f^m(\cH)\to{\bf D}_g^l(\cH)$ is a free
biholomorphic function  with $F(0)=0$ if and only if  $n=q$ and $F$  has the
form
$$ F(X_1,\ldots, X_n)=[X_1,\ldots, X_n]U,\qquad (X_1,\ldots, X_n)\in {\bf D}_f^m(\cH),
$$
where $U$ is  an invertible bounded linear operator  on
$\CC^n$ such that
$$
[W_1^{(f)},\ldots, W_n^{(f)}]U\in {\bf D}_g^l(F^2(H_n)) \quad \text{
and } \quad [W_1^{(g)},\ldots, W_n^{(g)}]U^{-1}\in {\bf
D}_f^m(F^2(H_n)).
$$

\end{theorem}

\begin{proof} Assume that $F:{\bf D}_f^m(\cH)\to{\bf D}_g^l(\cH)$ is a free
biholomorphic function  with $F(0)=0$. According to Theorem
\ref{rep-finite3}, we must have $n=q$. Then    $F$ has a
representation $F=(F_1,\ldots, F_n)$, where $F_j$ is a power series
with scalar coefficients, having the  representation
\begin{equation*}
 F_j(Z_1,\ldots, Z_n)=\sum_{k=1}^n a_{kj}Z_k+
\Psi_2^{(j)}(Z_1,\ldots, Z_n),
\end{equation*}
and $\Psi_2^{(j)}$ is a power series of the form
$\Psi_2^{(j)}(Z_1,\ldots, Z_n)=\sum_{p=2}^\infty \sum_{|\alpha|=p}
a_\alpha^{(j)} Z_\alpha$. It is easy to see that if $(N_1,\ldots,
N_n)\in {\bf D}_{f,\text{\rm nil}}^m(\cH)$, then $F(N_1,\ldots,
N_n)\in {\bf D}_{g,\text{\rm nil}}^l(\cH)$.  Similarly, one can show
that   $ F^{-1}({\bf D}_{g,\text{\rm nil}}^l(\cH))\subseteq {\bf
D}_{f,\text{\rm nil}}^m(\cH)$. Therefore, we have  $F({\bf
D}_{f,\text{\rm nil}}^m(\cH))={\bf D}_{g,\text{\rm nil}}^l(\cH)$.
Applying Theorem \ref{Cartan2}, we deduce that $ F(Z_1,\ldots,
Z_n)=[Z_1,\ldots, Z_n]U$,  where $U$ is  an invertible bounded
linear operator  on $\CC^n$. Since $F:{\bf D}_f^m(\cH)\to{\bf
D}_g^l(\cH)$ is a free biholomorphic function,  $(rW_1^{(f)},\ldots,
rW_n^{(f)})\in{\bf D}_{f, \text{\rm rad}}^m(\cH)$, and
$(rW_1^{(g)},\ldots, rW_n^{(g)})\in {\bf D}_{g, \text{\rm
rad}}^l(\cH)$, we deduce that
$$
[rW_1^{(f)},\ldots, rW_n^{(f)}]U\in {\bf D}_g^l(F^2(H_n)) \quad
\text{ and } \quad [rW_1^{(g)},\ldots, rW_n^{(g)}]U^{-1}\in {\bf
D}_f^m(F^2(H_n))
$$
for any $r\in [0,1)$. Since the domains ${\bf D}_g^l(F^2(H_n))$ and
${\bf D}_f^m(F^2(H_n))$ are closed in the operator norm topology,
and taking $r\to 1$,  we obtain the assertion of  the theorem.

Conversely, assume that $F$  has the form
$$ F(X_1,\ldots, X_n)=[X_1,\ldots, X_n]U,\qquad (X_1,\ldots, X_n)\in {\bf D}_f^m(\cH),
$$
where $U$ is  an invertible bounded linear operator  on $\CC^n$ such
that
\begin{equation}
\label{cond-incl}
 [W_1^{(f)},\ldots, W_n^{(f)}]U\in {\bf
D}_g^l(F^2(H_n)) \quad \text{ and } \quad [W_1^{(g)},\ldots,
W_n^{(g)}]U^{-1}\in {\bf D}_f^m(F^2(H_n)).
\end{equation}
Define the free holomorphic function $G$ by setting
$$ G(Y_1,\ldots, Y_n)=[Y_1,\ldots, Y_n]U^{-1},\qquad (Y_1,\ldots, Y_n)\in {\bf
D}_g^l(\cH).
$$
It is clear that $F$ and $G$ are continuous functions  in the
operator norm topology.  Moreover, note that $F|_{{\bf
D}_{f,\text{\rm rad}}^m(\cH)}$  and  $F^{-1}|_{{\bf D}_{g,\text{\rm
rad}}^l(\cH)}$ are free holomorphic functions on ${\bf
D}_{f,\text{\rm rad}}^m(\cH)$ and ${\bf D}_{g,\text{\rm
rad}}^l(\cH)$, respectively. Since $F\circ G=G\circ F=\text{\rm
id}$, it remains to show that
$$
F({\bf D}_f^m(\cH))\subseteq {\bf D}_g^l(\cH)\quad \text{ and }\quad
G({\bf D}_g^l(\cH))\subseteq {\bf D}_f^m(\cH).
$$
To this end, let $(X_1,\ldots, X_n)\in {\bf D}_f^m(\cH)$ and fix
$r\in [0, 1)$. Then $(rX_1,\ldots, rX_n)\in {\bf D}_{f,\text{\rm
rad}}^m(\cH)\subset {\bf D}_{f,\text{\rm pure}}^m(\cH)$  and, due to
the dilation theorem from \cite{Po-Berezin} there is a separable
infinite dimensional Hilbert space  $\cM$ such that
\begin{equation}
\label{dilation} rX_i^*=({W_i^{(f)}}^* \otimes I_\cM)|_{\cH},\qquad
i=1,\ldots, n,
\end{equation}
where $\cH$ is identified with a co-invariant subspace  of
$F^2(H_n)\otimes \cM$ under the operators ${W_1^{(f)}}^* \otimes
I_\cM$, $\ldots, $ ${W_n^{(f)}}^* \otimes I_\cM$. Let $U:=[a_{ji}]$
be the matrix representation of the invertible operator from
relation \eqref{cond-incl}. Setting $L_i:=\sum_{k=1}^n a_{ki}
W_k^{(f)}$, $i=1,\ldots, n$,  the first part of relation
\eqref{cond-incl} implies
\begin{equation} \label{figi}
(id-\Phi_{g,{\bf L}})^l(I_{F^2(H_n)\otimes \cM}) =\sum_{p=0}^l
(-1)^p \left(\begin{matrix} l\\p\end{matrix}\right)\Phi_{g,{\bf
L}}^p(I_{F^2(H_n)\otimes \cM})\geq 0,
\end{equation}
where ${\bf L}:=(L_1\otimes I_\cM,\ldots, L_n\otimes I_\cM)$. Note
that each $\Phi_{g,{\bf L}}^p(I_{F^2(H_n)\otimes \cM})$ is given by a
series which is convergent in the  weak operator topology and, due
to condition \eqref{dilation}, we deduce that
$$
\Phi_{g,F(rX_1,\ldots, rX_n)}^p (I_\cH)=P_\cH  \Phi_{g,{\bf
L}}^p(I_{F^2(H_n)\otimes \cM})|_{\cH},\qquad p=0,1,\ldots, l.
$$
Combining this relation  with  inequality \eqref{figi}, we obtain
$(id-\Phi_{g,F(rX_1,\ldots, rX_n)})^l(I_\cH)\geq 0$, which shows
that $F(rX_1,\ldots, rX_n)\in {\bf D}_g^l(\cH)$ for any $r\in
[0,1)$. Since ${\bf D}_g^l(\cH)$ is closed in the operator norm
topology and $F(rX_1,\ldots, rX_n)=r[X_1,\ldots, X_n]U$, we deduce
that $[X_1,\ldots, X_n]U\in {\bf D}_g^l(\cH)$ for any $n$-tuple
$(X_1,\ldots, X_n)\in {\bf D}_f^m(\cH)$. This proves that $F({\bf
D}_f^m(\cH))\subseteq {\bf D}_g^l(\cH)$. Similarly, one can prove
the inclusion $ G({\bf D}_g^l(\cH))\subseteq {\bf D}_f^m(\cH)$. This
completes the proof.
\end{proof}

We denote by $Aut({\bf D}_f^m):=Bih({\bf D}_f^m, {\bf D}_f^m)$ the set of all
 free biholomorphic functions  of ${\bf D}_f^m(\cH)$.
 Due to Theorem  \ref{more-prop},  $Aut({\bf D}_f^m)$ is  a group with respect to the
composition of free holomorphic functions.
As a consequence  of Theorem \ref{Cartan3},
  we obtain the following characterization of $Aut_0({\bf D}_f^m)$, the subgroup of all  free holomorphic automorphisms of
 ${\bf D}_f^m(\cH)$  that fix the origin.
\begin{corollary}
\label{aut1} Let $f$   be  a positive regular free holomorphic
function with $n$  indeterminates  and let $m \geq 1$. A map $\Psi:
{\bf D}_f^m(\cH)\to {\bf D}_f^m(\cH)$ is a free holomorphic
automorphism with $\Psi(0)=0$  if and only if  it has the form
\begin{equation*}
 \Psi(X_1,\ldots, X_n)=[X_1,\ldots, X_n]U , \qquad
(X_1,\ldots, X_n) \in {\bf D}_f^m(\cH),
\end{equation*}
where   $U$  is an invertible operator on $\CC^n$  such that
$$
[W_1,\ldots, W_n]U\in {\bf D}_f^m(F^2(H_n)) \quad \text{ and } \quad
[W_1,\ldots, W_n]U^{-1}\in {\bf D}_f^m(F^2(H_n)),
$$
 and $(W_1,\ldots, W_n)$ is the universal model associated
with the noncommutative domain ${\bf D}_f^m$.
\end{corollary}

Now, we  can  characterize the unit ball of $B(\cH)^n$ among the
noncommutative
 domains ${\bf D}_f^m(\cH)$,  up to  free biholomorphisms.

\begin{corollary}\label{biho-equi1} Let $g$ be  a positive regular free
holomorphic function with $q$  indeterminates and let $l\geq 1$.
Then the noncommutative domain ${\bf D}_g^l(\cH)$ is biholomorphic
equivalent to the unit ball $[B(\cH)^n]_1$ if and only if $q=n$ and
there is  an invertible bounded linear operator $U\in B(\CC^n)$ such
that
$$
[S_1,\ldots, S_n]U\in {\bf D}_g^l(F^2(H_n)) \quad \text{ and } \quad
[W_1^{(g)},\ldots, W_n^{(g)}]U^{-1}\in [B(\cH)^n]_1^-,
$$
where $S_1,\ldots, S_n$ are the left creation operators on the full
Fock space $F^2(H_n)$.
\end{corollary}
\begin{proof} Let $p:=X_1+\cdots + X_n$ and recall  that  ${\bf D}_p^1(\cH)=[B(\cH)^n]_1^-$.
Assume that $F:{\bf D}_{g}^l(\cH)\to {\bf D}_{p}^1(\cH)$  is a free
biholomorphic function. Due to Theorem \ref{rep-finite3},  $n=q$ and
  the  representation  of $F$ on $\CC$,  i.e., the map
  $F_1$ defined by
  $$
  \CC^{n}\supset {\bf
D}_{p}^1(\CC)\ni (\lambda_1,\ldots, \lambda_n)\mapsto
F(\lambda_1,\ldots, \lambda_n)\in {\bf D}_{g}^l(\CC) \subset \CC^{n}
$$
is a homeomorphism from   ${\bf D}_{p}^1(\CC)$ onto  ${\bf
D}_{g}^l(\CC)$ and   a biholomorphic function from $Int({\bf
D}_{p}^1(\CC))=[B(\cH)^n]_1$ onto  $Int({\bf D}_{g}^l(\CC))$.
Consequently, $z_0:=F(0)\in \BB_n$, the open unit ball of $\CC^n$.
According to Theorem 2.3 from \cite{Po-automorphism}, there is a
free holomorphic  automorphism $\Psi_{z_0}$ of the open unit ball
$[B(\cH)^n]_1$ such that $\Psi_{z_0}(z_0)=0$. Moreover, $\Psi_{z_0}$
has a continuous extension to the closed ball $[B(\cH)^n]_1^-$ which
is a homeomorphism of $[B(\cH)^n]_1^-$. Then, according to Theorem
\ref{more-prop},  the composition $\Psi_{z_0}\circ F:{\bf
D}_{g}^l(\cH)\to {\bf D}_{p}^1(\cH)$ is a free biholomorphic
function with the property that $(\Psi_{z_0}\circ F)(0)=0$. Applying
now Theorem \ref{Cartan3}, we find an invertible bounded linear
operator $U\in B(\CC^n)$ such that
$$
[S_1,\ldots, S_n]U\in {\bf D}_g^l(F^2(H_n)) \quad \text{ and } \quad
[W_1^{(g)},\ldots, W_n^{(g)}]U^{-1}\in [B(\cH)^n]_1^-,
$$
and such that $(\Psi_{z_0}\circ F)(X_1,\ldots, X_n)=[X_1,\ldots,
X_n]U$ for all $(X_1,\ldots, X_n)\in {\bf D}_{g}^l(\cH)$. The
converse follows easily  from Theorem \ref{Cartan3}.
\end{proof}
We remark that in the particular case when $l=1$ we will get a more
precise result in the next section (see Corollary \ref{biho-equi2}).

In what follows  we use  again  the interaction between  the theory
of functions in several complex variables   and our noncommutative
theory  to obtain  some results on the  classification  of  the
noncommutative domains ${\bold D}_f^m(\cH)$, $m\geq 1$.

\begin{theorem} \label{Thullen} Let $f$ and $g$ be positive regular free
holomorphic functions with n indeterminates and let $m,l\geq 1$. Assume that
 there is $p'\in \{1,2,\ldots\}$ such that   the domains $Int({\bf
D}_{f}^m(\CC^{p'}))$  and $Int({\bf D}_{g}^l(\CC^{p'}))$ are linearly equivalent and
  all the automorphisms of  $Int({\bf
D}_{f}^m(\CC^{p'}))$    fix the origin.

Then the noncommutative domains ${\bf D}_{f}^m(\cH)$ and ${\bf D}_{g}^l(\cH)$  are  free
biholomorphic equivalent if and only if  there is  an invertible bounded
linear operator $U\in B(\CC^n)$ such that
$$
[W_1^{(f)},\ldots, W_n^{(f)}]U\in {\bf D}_g^l(F^2(H_n)) \quad \text{
and } \quad [W_1^{(g)},\ldots, W_n^{(g)}]U^{-1}\in {\bf
D}_f^m(F^2(H_n)).
$$
\end{theorem}
\begin{proof} Since  the domains $Int({\bf
D}_{f}^m(\CC^{p'}))$  and $Int({\bf D}_{g}^l(\CC^{p'}))$ are
linearly equivalent, there is  a biholomorphic function $\varphi
:Int({\bf D}_{f}^m(\CC^{p'}))\to Int({\bf D}_{g}^l(\CC^{p'}))$ such
that $\varphi(0)=0$. Suppose that $F:{\bf D}_{f}^m(\cH)\to {\bf
D}_{g}^l(\cH)$ is a free biholomorphic function. According to
Theorem \ref{rep-finite3},  $n=q$ and
  the  representation  of $F$ on $\CC^{p'}$,  i.e., the map
  $F_{p'}$ defined by
  $$
  \CC^{n{p'}^2}\supset {\bf
D}_{f}^m(\CC^{p'})\ni (\Lambda_1,\ldots, \Lambda_n)\mapsto
F(\Lambda_1,\ldots, \Lambda_n)\in {\bf D}_{g}^l(\CC^{p'}) \subset
\CC^{n{p'}^2}
$$
is a  biholomorphic function from $Int({\bf D}_{f}^m(\CC^{p'})) $
onto  $Int({\bf D}_{g}^l(\CC^{p'}))$. Consequently, $\varphi^{-1}
\circ F_{p'}$ is an automorphism of $Int({\bf D}_{f}^m(\CC^{p'}))$.
Due to the hypothesis, we have $(\varphi^{-1} \circ F_{p'})(0)=0$.
Therefore, $F_{p'}(0)=0$, which obviously implies $F(0)=0$. Applying
now Theorem \ref{Cartan3}, the result follows. The converse is due
to the same theorem. The proof is complete.
\end{proof}

Using Thullen characterization of domains in $\CC^2$  with non-compact automorphism
group (\cite{Th}) and Theorem \ref{Cartan3}, we can obtain the following
classification result for noncommutative domains generated by positive regular
 free holomorphic functions in 2 indeterminates. We recall that
 Thullen proved that if a bounded Reinhardt domain in $\CC^2$ has a
 biholomorphic map that does not  fix the origin,  then the domain
 is linearly equivalent  to one of the following: polydisc, unit
 ball, or the so-called Thullen domain. Combining this result with
 Theorem \ref{Thullen}, we obtain the following consequence.

\begin{corollary} \label{Thullen2} Let $f$ and $g$ be positive regular free
holomorphic functions with 2 indeterminates and let $m,l\geq 1$. Assume that
  the Reinhardt domains  $Int({\bf
D}_{f}^m(\CC))$ and   $Int({\bf D}_{g}^l(\CC))$  are linearly
equivalent but they are not linearly  equivalent to  either the
polydisc, the unit ball, or any Thullen domain in $\CC^2$.

Then the noncommutative domains ${\bf D}_{f}^m(\cH)$ and ${\bf D}_{g}^l(\cH)$  are  free
biholomorphic equivalent if and only if  there is  an invertible bounded linear operator $U\in B(\CC^2)$ such that
$$
[W_1^{(f)}, W_2^{(f)}]U\in {\bf D}_g^l(F^2(H_2)) \quad \text{
and } \quad [W_1^{(g)},W_2^{(g)}]U^{-1}\in {\bf
D}_f^m(F^2(H_2)).
$$
\end{corollary}

 \bigskip

\section{Free biholomorphic classification of   noncommutative domains}

 The main result of this  section
   shows that the  free biholomorphic classification of the noncommutative
 domains ${\bold
D}_f^m(\cH)\subset B(\cH)^n$ is the same as the classification
   of the corresponding noncommutative domain
  algebras $\cA_n({\bf D}_f^m)\simeq
 A({\bf D}_{f, \text{\rm rad}}^m)$.
 Combining this  result  with the Cartan type results   from Section 4,
we provide several  consequences concerning the free biholomorphic
classification.

 Let $f$ and $g$ be positive regular free holomorphic
functions with $n$ and $q$ indeterminates, respectively,  and let
$m,l\geq 1$. According to Theorem \ref{A-infty},  there is a completely
isometric isomorphism
$$
A({\bf D}_{f,\text{\rm rad}}^m)\ni \chi\mapsto \widetilde \chi:=\lim_{r\to
1} \chi(rW_1,\ldots, rW_n)\in \cA_n({\bf D}_f^m),
$$
where the limit is in the operator norm topology,  whose inverse is
the noncommutative Berezin  transform ${\bf B}$, i.e., $\chi(X)={\bf
B}_X[\widetilde \chi]$ for any  $X\in {\bf D}_{f,\text{\rm
rad}}^m(\cH)$.  If $\Phi: A({\bf D}_{f,\text{\rm rad}}^m)\to A({\bf
D}_{g,\text{\rm rad}}^l)$ is a unital   algebra  homomorphism, it
induces a unique unital homomorphism $\widehat \Phi:\cA_n({\bf
D}_f^m)\to \cA_q({\bf D}_g^l)$ such that the  diagram
\begin{equation*}
\begin{matrix}
\cA_n({\bf D}_f^m)& \mapright{\widehat\Phi}&
\cA_q({\bf D}_g^l)\\
 \mapdown{{\bf B}}& & \mapdown{{\bf B}}\\
A({\bf D}_{f,\text{\rm rad}}^m)&  \mapright{\Phi}& A({\bf
D}_{g,\text{\rm rad}}^l)
\end{matrix}
\end{equation*}
is commutative, i.e., $\Phi {\bf B}={\bf B}\widehat \Phi$, where
${\bf B}$ is the appropriate noncommutative Berezin transform on the
noncommutative domain algebra $\cA_n({\bf D}_f^m)$ or  $\cA_q({\bf
D}_g^l)$ (see  relation \eqref{BT}). The homomorphisms $\Phi$ and
$\widehat \Phi$ uniquely determine each other by the formulas:
\begin{equation*}
\begin{split}
[\Phi (\chi)](X)&={\bf B}_X[\widehat \Phi(\widetilde \chi)], \qquad
\chi\in A({\bf D}_{f,\text{\rm rad}}^m),\ X\in {\bf D}_{g,\text{\rm
rad}}^l(\cH), \quad \text{ and }\\
\widehat \Phi(\widetilde \chi)&=\widetilde{\Phi(\chi)}, \qquad
\widetilde \chi\in \cA_n({\bf D}_f^m).
\end{split}
\end{equation*}
We recall that   $Bih({\bf D}_f^m, {\bf D}_g^l)$  stands for
the set of all biholomorphic functions  from  ${\bf D}_f^m(\cH)$ onto ${\bf D}_g^l(\cH)$, i.e., all homeomorphisms
 $\varphi:{\bf D}_f^m(\cH)\to {\bf D}_g^l(\cH)$ with the property
 that $\varphi|_{{\bf D}_{f,\text{\rm
rad}}^m(\cH)}$ and $\varphi^{-1}|_{{\bf D}_{g,\text{\rm
rad}}^l(\cH)}$ are free holomorphic functions.
Let $\Phi$ be
a unital completely
isometric  isomorphism
 from the noncommutative domain algebra  $A({\bf D}_{f,\text{\rm
rad}}^m)$ onto  $A({\bf D}_{g,\text{\rm rad}}^l)$. Consider the
closed operator systems  in $B(F^2(H_n))$,
$$
\cS_f:=\overline{\text{\rm  span}} \{ W_\alpha^{(f)}
{W_\beta^{(f)}}^*;\ \alpha,\beta\in \FF_n^+\}\quad \text{ and }\quad
\cS_g:=\overline{\text{\rm  span}} \{ W_\alpha^{(g)}
{W_\beta^{(g)}}^*;\ \alpha,\beta\in \FF_q^+\},
$$
where $(W_1^{(f)},\ldots, W_n^{(f)})$ and $(W_1^{(g)},\ldots, W_q^{(g)})$ are the universal models of the noncommutative domains ${\bf D}_f^m$ and ${\bf D}_g^l$, respectively.
We say that $\Phi$ has   completely contractive hereditary extension
if  the induced isomorphism   $\widehat \Phi:\cA_n({\bf D}_f^m)\to
\cA_q({\bf D}_g^l)$ has the property that
 the
linear maps $\widehat \Phi_*:\cS_f\to \cS_g$ and $(\widehat \Phi^{-1})_*:\cS_g\to \cS_f$ defined by
$$
 \widehat \Phi_*\left(W_\alpha^{(f)} {W_\beta^{(f)}}^*\right):=
\widehat\Phi\left(W_\alpha^{(f)}\right)\widehat\Phi\left(W_\beta^{(f)}\right)^*,\qquad
\alpha, \beta\in \FF_n^+,
$$
and
$$
 (\widehat \Phi^{-1})_*\left(W_\alpha^{(g)} {W_\beta^{(g)}}^*\right):=
\widehat\Phi^{-1}\left(W_\alpha^{(g)}\right)
\widehat\Phi^{-1}\left(W_\beta^{(g)}\right)^*,\qquad \alpha,
\beta\in \FF_q^+,
$$
respectively, are completely contractive. Using an approximation
argument, one can easily show that $\widehat \Phi_* \circ (\widehat
\Phi^{-1})_*=id_{\cS_g}$ and $(\widehat \Phi^{-1})_*\circ \widehat
\Phi_*=id_{\cS_f}$. Consequently, the map $\widehat \Phi_*:\cS_f\to
\cS_g$ is a   unital completely isometric linear isomorphism between
operator systems, which extends the completely isometric isomorphism
$\widehat \Phi:\cA_n({\bf D}_f^m)\to \cA_q({\bf D}_g^l)$.

Now, we are ready to prove the main result of this section.  We show that the
 free biholomorphic classification of the domains ${\bold
D}_f^m(\cH)$ is the same as the classification, up to unital
completely isometric isomorphisms  having  completely contractive
hereditary extension, of the corresponding noncommutative domain
algebras $A({\bf D}_{f, \text{\rm rad}}^m)\simeq  \cA_n({\bf
D}_f^m)$.

 \begin{theorem}
 \label{auto-disk}  Let $f$ and $g$ be positive regular free
holomorphic functions with $n$ and $q$ indeterminates, respectively,  and let $m,l\geq 1$.
 Then the following statements are equivalent:
\begin{enumerate}
\item[(i)]
    $\Psi:A({\bf D}_{f,\text{\rm
rad}}^m)\to A({\bf D}_{g,\text{\rm rad}}^l)$ is a unital completely
isometric  isomorphism  with completely contractive hereditary
extension;
\item[(ii)]
 there is   $\varphi \in Bih( {\bf D}_g^l,{\bf D}_f^m)$  such that
 $$
 \Psi(\chi)=\chi\circ \varphi,\qquad \chi\in A({\bf D}_{f,\text{\rm
rad}}^m).
 $$
 \end{enumerate}
In this case, $\widehat \Psi(\widetilde\chi)=\cB_{\widetilde
\varphi}[\widetilde \chi]$, $\widetilde \chi\in \cA_n({\bf D}_f^m)$,
where $\cB_{\widetilde \varphi}$ is the noncommutative Berezin
transform at $\widetilde \varphi$.
 In the particular case when $m=l=1$,   any unital completely isometric  isomorphism  has   completely contractive hereditary extension.
\end{theorem}

\begin{proof} Let $\sum_{\alpha\in \FF_n^+, |\alpha|\geq 1} a_\alpha
X_\alpha$ be the representation of $f$.
   First we prove that $(i)\implies (ii)$.
Assume that $\Psi$  is a
 unital completely isometric     isomorphism.
 Then, the induced map   $\widehat
 \Psi:\cA_n({\bf D}_f^m)\to \cA_q({\bf D}_g^l)$ is a completely isometric
 isomorphism.
Denote
 \begin{equation}
 \label{var-tilde}
 \widetilde\varphi_i:=\widehat \Psi(W_i^{(f)})\in \cA_q({\bf D}_g^l),
\qquad i=1,\ldots, n,
 \end{equation}
 where $W^{(f)}:=(W_1^{(f)},\ldots, W_n^{(f)})\in {\bf
 D}_f^m(F^2(H_n))$ is the universal model of the noncommutative
 domain ${\bf D}_f^m$.
 Since $\Phi_{f,W^{(f)}}(I)\leq I$, we have
$\Phi_{f,rW^{(f)}}(I)=\sum_{k=1}^\infty \sum_{|\alpha|=k} a_\alpha
r^{|\alpha|} W_\alpha^{(f)} {W_\alpha^{(f)}}^*\leq
 I$ and $\Phi_{f,rW^{(f)}}(I)\in  \cS_f$  for any $r\in [0,1)$. Using relation \eqref{var-tilde}
 and  the fact that
  $\widehat
 \Psi $ is a  completely isometric homomorphism, we deduce that
 $0\leq \Phi_{f,r\widetilde \varphi}(I)\leq I$ for $r\in[0,1)$, which proves that $r\widetilde
 \varphi:=(r\widetilde
 \varphi_1,\ldots, r\widetilde
 \varphi_n)\in {\bf D}_f^1(F^2(H_n))$. Since ${\bf D}_f^1(F^2(H_n))$ is closed in
the operator norm topology, we deduce that $\widetilde\varphi\in {\bf
D}_f^1(F^2(H_n))$.

  We recall  (see Corollary 1.5 from \cite{Po-Berezin}) that
if $\psi$ is a positive linear map on $B(\cH)$ such that
$\psi(I)\leq I$ and $(id-\psi)^m(I)\geq 0$ for some $m\in \NN$, then
$$
0\leq (id-\psi)^m(I) \leq  (id-\psi)^{m-1}(I)\leq \cdots\leq
(id-\psi)(I)\leq I.
$$
Since ${\bf
 D}_{f}^m$ is a starlike domain, $(rW_1,\ldots, rW_n)\in {\bf D}_{f
 }^m(F^2(H_n))$, i.e.,  $(id-\Phi_{f,rW^{(f)}})^m(I)\geq 0$.
  Using the  above-mentioned result in our setting, we have
$$
 0\leq (id-\Phi_{f,r W^{(f)}})^s(I)=\sum_{p=0}^s (-1)^p
\left(\begin{matrix} s\\p\end{matrix}\right) \Phi_{f,r
W^{(f)}}^p(I)=I-\sum_{p=1}^s (-1)^{p+1} \left(\begin{matrix}
s\\p\end{matrix}\right) \Phi_{f,r W^{(f)}}^p(I)\leq I
$$
 for any  $ s =1,\ldots, m$, which is equivalent to
 \begin{equation}
\label{ineq}
 0\leq \sum_{p=1}^s (-1)^{p+1} \left(\begin{matrix}
s\\p\end{matrix}\right) \Phi_{f,r W^{(f)}}^p(I)\leq I, \qquad  s
=1,\ldots, m.
\end{equation}
Since
$$
\Phi_{f,rW^{(f)}}^p(I)=\sum_{k=1}^\infty \sum_{|\alpha|=k} a_\alpha
r^{|\alpha|} W_\alpha^{(f)}
\Phi_{f,rW^{(f)}}^{p-1}(I){W_\alpha^{(f)}}^*
$$
 and
$\|\Phi_{f,rW^{(f)}}^p(I)\|\leq 1$ for any $p\in \NN$, it is clear
that $\Phi_{f,rW^{(f)}}^p(I)\in \cS_f:=\overline{\text{\rm  span}}
\{ W_\alpha^{(f)} {W_\beta^{(f)}}^*;\ \alpha,\beta\in \FF_n^+\}$.
Assume now that $\Psi:A({\bf D}_{f,\text{\rm rad}}^m)\to A({\bf
D}_{g,\text{\rm rad}}^l)$ satisfies condition (i). Then, due to the
considerations above,  relation \eqref{ineq} holds and  implies
$$
\left\|\sum_{p=1}^s (-1)^{p+1} \left(\begin{matrix}
s\\p\end{matrix}\right) \Phi_{f,r W^{(f)}}^p(I)\right\|\leq 1.
$$
Using relation \eqref{var-tilde}
 and  the fact that  $\widehat \Psi_*:\cS_f\to \cS_g$ is a unital completely
contractive linear map, we deduce that
\begin{equation*}
 \left\|\sum_{p=1}^s (-1)^{p+1} \left(\begin{matrix}
s\\p\end{matrix}\right) \Phi_{f,r \widetilde
\varphi}^p(I)\right\|\leq 1, \qquad  s =1,\ldots, m.
\end{equation*}
Hence  $\sum_{p=1}^s (-1)^{p+1} \left(\begin{matrix}
s\\p\end{matrix}\right) \Phi_{f,r \widetilde \varphi}^p(I)\leq I$
and, consequently,
$$
(id-\Phi_{f,r\widetilde \varphi})^s(I)\geq 0, \qquad s =1,\ldots, m,
$$
which shows that $r\widetilde
 \varphi:=(r\widetilde
 \varphi_1,\ldots, r\widetilde
 \varphi_n)\in {\bf D}_f^m(F^2(H_n))$ for any $r\in [0,1)$. Since  ${\bf
 D}_f^m(F^2(H_n))$ is closed in the operator norm topology, we
 deduce that
 $\widetilde \varphi:=(\widetilde \varphi_1,\ldots,
\widetilde \varphi_n)$ is in the noncommutative domain ${\bf
D}_f^m(F^2(H_n))$.

Setting
$$\varphi_i(X):={\bf B}_X[\widetilde \varphi_i], \qquad X\in {\bf D}_{g,\text{\rm
rad}}^l(\cH), $$
 and due to Theorem \ref{A-infty},  we deduce that  the map
$\varphi(X):=(\varphi_1(X),\ldots,\varphi_n(X))$   is a free
holomorphic function on ${\bf D}_{g,\text{\rm rad}}^l(\cH)$ with
continuous extension  to ${\bf D}_g^l(\cH)$. This extension is also
denoted by $\varphi$. On the other hand, since,  for each  $s
=1,\ldots, m$, $(id-\Phi_{f,r\widetilde \varphi})^s(I)$ is  a
positive operator in $\cS_g$ and the noncommutative Berezin
transform $\cB_X$ on the operator system $\cS_g$ is positive, we
deduce that
$$
(id-\Phi_{f,r  \varphi(X)})^s(I)=\cB_X\left[(id-\Phi_{f,r\widetilde
\varphi})^s(I)\right]\geq 0
$$
for all $s =1,\ldots, m$. Consequently,  $r\varphi(X)\in {\bf
D}_f^m(\cH)$ for $r\in [0,1)$. Since ${\bf D}_f^m(\cH)$ is closed in
the operator norm topology, we deduce that $\varphi(X)\in {\bf
D}_f^m(\cH)$. Therefore, $\varphi$ takes
 values in ${\bf D}_{f}^m(\cH)$.
Note also  that due to the properties of the noncommutative Berezin
transform $\cB_X$ on the operator system $\cS_g$ (see relation
\eqref{BT}), we have
\begin{equation}
\label{fibx} \varphi_i(X)=\cB_X[\widetilde \varphi_i],\qquad X\in
{\bf D}_g^l(\cH).
\end{equation}
Similarly, we set
\begin{equation}
 \label{xi}\widetilde\xi_i:=\widehat \Psi^{-1}(W_i^{(g)})\in \cA_n({\bf D}_f^m),\qquad
 i=1,\ldots, q,
 \end{equation}
   and   deduce that $\widetilde \xi:=(\widetilde \xi_1,\ldots, \widetilde
\xi_q)$ is in
  ${\bf D}_g^l(F^2(H_n))$. Now, if we define
  \begin{equation*}
  \xi_i(X):={\bf B}_X[\widetilde \xi_i], \qquad X\in {\bf D}_{f,\text{\rm
rad}}^m(\cH), \ i=1,\ldots, q,
\end{equation*}
 then, using again   Theorem \ref{A-infty},  we obtain that the map
$\xi(X):=(\xi_1(X),\ldots,\xi_q(X))$  is a free holomorphic function
on ${\bf D}_{f,\text{\rm rad}}^m(\cH)$ with values in ${\bf
D}_{g}^l(\cH)$ and  continuous extension  to ${\bf D}_f^m(\cH)$,
which we also  denote by $\xi$.  Note  that
\begin{equation}
\label{xibx2} \xi_i(X)=\cB_X[\widetilde \xi_i],\qquad X\in {\bf
D}_f^m(\cH),
\end{equation}
where $\cB_X$ is the noncommutative Berezin transform on the
operator system $\cS_f$.
 Now,  due to Theorem \ref{A-infty}, each $\widetilde \xi_i\in
\cA_n({\bf D}_f^m)$, $i=1,\ldots, q$, has a unique Fourier
representation $\sum_{\alpha\in \FF_n^+} c_\alpha W_\alpha^{(f)}$
such that
 $$
 \widetilde \xi_i=\lim_{r\to 1} \sum_{k=0}^\infty
\sum_{|\alpha|=k} c_\alpha r^{|\alpha|} W_\alpha^{(f)},
$$
 where the
limit is in the operator norm topology. Hence, using the
continuity of $\widehat \Psi$ in the operator norm, and relations
\eqref{xi} and \eqref{var-tilde}, we obtain
\begin{equation*}
\begin{split}
W_i^{(g)}&=\widehat \Psi(\widetilde \xi_i)=\widehat
\Psi\left(\lim_{r\to 1} \sum_{k=0}^\infty \sum_{|\alpha|=k} c_\alpha
r^{|\alpha|} W_\alpha^{(f)}
\right)\\
&=\lim_{r\to 1} \sum_{k=0}^\infty \sum_{|\alpha|=k} c_\alpha
r^{|\alpha|} \widehat \Psi(W_\alpha^{(f)})= \lim_{r\to 1}
\sum_{k=0}^\infty \sum_{|\alpha|=k} c_\alpha r^{|\alpha|} \widetilde
\varphi_\alpha.
\end{split}
\end{equation*}
Consequently, using the  continuity in the operator norm  of  the
noncommutative Berezin transform  at $X\in {\bf D}_g^l(\cH)$ on the
operator system $\cS_g$, and  relations \eqref{fibx} and
\eqref{xibx2},  we have
\begin{equation*}
\begin{split}
X_i&={\cB}_X[W_i^{(g)}]={\cB}_X\left[\lim_{r\to 1}\sum_{k=0}^\infty
\sum_{|\alpha|=k} c_\alpha
r^{|\alpha|} \widetilde \varphi_\alpha\right]\\
&=\lim_{r\to 1} \sum_{k=0}^\infty \sum_{|\alpha|=k} c_\alpha
r^{|\alpha|}  {\cB}_X[\widetilde\varphi_\alpha]
=\lim_{r\to 1} \sum_{k=0}^\infty \sum_{|\alpha|=k} c_\alpha
r^{|\alpha|}   \varphi_\alpha(X)\\
&=\lim_{r\to 1} \cB_{\varphi(X)}\left[ \sum_{k=0}^\infty
\sum_{|\alpha|=k} c_\alpha r^{|\alpha|} W_\alpha^{(f)}\right]\\
&=\cB_{\varphi(X)}[\widetilde\xi_i]=\xi_i(\varphi(X))
\end{split}
\end{equation*}
for each $i=1,\ldots, q$, and any $X\in {\bf D}_g^l(\cH)$. Hence
$(\xi\circ \varphi)(X)=X$ for all  $X\in {\bf D}_g^l(\cH)$.
Similarly, one can prove that $(\varphi\circ \xi)(X)=X$ for  $X\in
{\bf D}_f^m(\cH)$. Therefore,
 $\varphi:{\bf D}_g^l(\cH)\to {\bf D}_f^m(\cH)$ is a homeomorphism  such that
  $\varphi$ and $\varphi^{-1}:=\xi$ are free holomorphic functions on   the noncommutative domains
  ${\bf D}_{g,\text{\rm
rad}}^l(\cH)$ and ${\bf D}_{f,\text{\rm
rad}}^m(\cH)$, respectively. This shows that $\varphi \in Bih( {\bf D}_g^l,{\bf D}_f^m)$.

Now let   $\widetilde \chi\in \cA_n({\bf D}_f^m)$  have  the Fourier
representation $\sum_{\alpha\in \FF_n^+} d_\alpha W_\alpha^{(f)}$.
Then
 $$
 \widetilde \chi=\lim_{r\to 1} \sum_{k=0}^\infty
\sum_{|\alpha|=k} d_\alpha r^{|\alpha|} W_\alpha^{(f)},
$$
 where the
limit is in the operator norm topology. According to Theorem
\ref{A-infty}, the map \begin{equation} \label{cibx}
\chi(X):=\cB_X[\widetilde \chi],\qquad X\in {\bf D}_f^m(\cH),
\end{equation}
 is in the domain algebra $A({\bf D}_{f,\text{\rm rad}}^m)$.
Due to the continuity of the homomorphism  $\widehat \Psi$ in the
operator norm and  using relation \eqref{var-tilde}, we have
$$
\widehat \Psi(\widetilde \chi)=\lim_{r\to 1} \sum_{k=0}^\infty
\sum_{|\alpha|=k} d_\alpha r^{|\alpha|} \widetilde \varphi_\alpha.
$$
Hence, using relations \eqref{fibx}, \eqref{cibx}, and applying  the
noncommutative Berezin transform  at $X\in {\bf D}_g^l(\cH)$ on the
operator system $\cS_g$ and using its continuity in norm, we get
\begin{equation*}
\begin{split}
[\Psi(\chi)](X)&={\cB}_X[\widehat \Psi(\widetilde \chi)]
=\lim_{r\to 1} \sum_{k=0}^\infty \sum_{|\alpha|=k} d_\alpha
r^{|\alpha|} {\cB}_X[\widetilde \varphi_\alpha]\\
 &=\lim_{r\to 1}
\sum_{k=0}^\infty \sum_{|\alpha|=k} d_\alpha r^{|\alpha|}
\varphi_\alpha(X)\\
&=\lim_{r\to 1}\cB_{\varphi(X)}\left[\sum_{k=0}^\infty
\sum_{|\alpha|=k} d_\alpha
r^{|\alpha|} W_\alpha^{(f)}\right]\\
&=\cB_{\varphi(X)}[\widetilde \chi] =(\chi\circ\varphi)(X)
\end{split}
\end{equation*}
for any $X\in {\bf D}_g^l(\cH)$, which completes the proof of   the
implication $(i)\implies (ii)$.

To prove that $(ii)\implies (i)$, assume that  $\varphi \in Bih( {\bf D}_g^l,{\bf D}_f^m)$ and define
 $
 \Psi(\chi):=\chi\circ \varphi$  for $ \chi\in A( {\bf D}_{f,\text{\rm rad}}^m)$.
 First note that, due to Theorem \ref{more-prop},
 $\chi\circ \varphi\in A( {\bf D}_{g,\text{\rm rad}}^l)$
  for all $\chi\in A( {\bf D}_{f, \text{\rm rad}}^m)$, and
   $\Psi$ is a
well-defined completely contractive homomorphism.
 A similar result can be deduced for  the map
$\Lambda(\chi):=\chi\circ\varphi^{-1}$, $\chi\in A( {\bf D}_{g, \text{\rm
rad}}^l)$. Now, note that, for any $\chi\in A( {\bf D}_{f, \text{\rm
rad}}^m)$ and $X\in {\bf D}_f^m(\cH)$, we have $(
\Lambda\circ\Psi)(\chi)=(\chi\circ \varphi)\circ \varphi^{-1}$ and
\begin{equation*}
\begin{split}
[(\chi\circ \varphi)\circ \varphi^{-1}](X)&=(\chi\circ
\varphi)(\varphi^{-1}(X))=\lim_{r\to 1}
\chi(r\varphi_1(\varphi^{-1}(X)),\ldots, r\varphi_n(\varphi^{-1}(X)))\\
&=\lim_{r\to1} \chi(rX_1,\ldots, rX_n)=\chi(X_1,\ldots, X_n),
\end{split}
\end{equation*}
where the convergence is in the operator norm topology.  Hence,
$\Lambda\circ \Psi=id$. Similarly, we can show that $\Psi\circ
\Lambda=id$. Since $\Psi$ and $\Lambda$ are completely contractive
homomorphisms, we deduce that   $\Psi$ is a completely isometric
isomorphism. Moreover, since $\varphi:=(
 \varphi_1,\ldots,
 \varphi_n)\in
 Bih( {\bf D}_g^l,{\bf D}_f^m)$,  we can use  Theorem \ref{A-infty} and   deduce  that $\widetilde
 \varphi:=(\widetilde
 \varphi_1,\ldots, \widetilde
 \varphi_n)\in {\bf D}_f^m(F^2(H_n))$ and $\widetilde
 \varphi_i\in \cA_q({\bf D}_g^l)$. Consequently, using   the noncommutative
 von Neumann inequality \eqref{vn1}, we deduce that
 the linear  map
$$
 \cS_f \ni W_\alpha^{(f)} {W_\beta^{(f)}}^*\mapsto
\widetilde\varphi_\alpha\widetilde\varphi_\beta^* \in
 \cS_g,\qquad \alpha, \beta\in \FF_n^+,
$$
is completely contractive.  A similar result can be deduced for
$\varphi^{-1}$. Therefore, $\Psi$ has   completely contractive
hereditary extension, and item (i) holds. Moreover, due to Theorem
\ref{more-prop}, we have $\widehat
\Psi(\widetilde\chi)=\cB_{\widetilde \varphi}[\widetilde \chi]$,
$\widetilde \chi\in \cA_n({\bf D}_f^m)$, where $\cB_{\widetilde
\varphi}$ is the noncommutative Berezin transform at $\widetilde
\varphi$.

Now, we consider  the particular  case when $m=l=1$. Let  $\widehat
 \Psi:\cA_n({\bf D}_f^1)\to \cA_q({\bf D}_g^1)$ be a completely isometric
 isomorphism. Then, we
showed at the beginning of this proof that  $\widetilde
 \varphi:=(\widetilde
 \varphi_1,\ldots, \widetilde
 \varphi_n)\in {\bf D}_f^1(F^2(H_n))$, where $\widetilde \varphi_i$
 are given by relation \eqref{var-tilde}. Similarly, we deduce that
  $\widetilde \xi:=(\widetilde \xi_1,\ldots, \widetilde
\xi_q)$ is in
  ${\bf D}_g^1(F^2(H_n))$, where $\widetilde \xi_i$ are given by relation \eqref{xi}.
As above, using   the noncommutative
 von Neumann inequality \eqref{vn1}, we deduce that
  $\Psi$  has   unital completely contractive hereditary extension,
  which completes the proof.
\end{proof}

\begin{corollary}\label{iso-dom} Let $f$ and $g$ be positive regular free
holomorphic functions with $n$ and $q$ indeterminates, respectively.
Then the noncommutative domains ${\bf D}_f^1(\cH)$ and  ${\bf
D}_g^1(\cH)$ are free biholomorphic equivalent if and only if   the
domain algebras $\cA_n({\bf D}_f^1)$  and $\cA_q({\bf D}_g^1)$ are
completely isometrically isomorphic.
\end{corollary}
Let us consider an example.
In \cite{ArL}, Arias and Latr\'emoli\`ere showed that if
$\omega:=X_1+X_2+ X_1X_2$ and $\xi:=X_1+X_2+ \frac{1}{2}(X_1X_2+
X_2X_1)$, then
  the noncommutative domain algebras $\cA_2({\bf D}_\omega^1)$ and $\cA_2({\bf D}_\xi^1)$
   are not completely isometrically   isomorphic.
Therefore, according to Corollary \ref{iso-dom}
 the noncommutative domains ${\bf D}_{\omega}^1(\cH)$ and ${\bf D}_{\xi}^1(\cH)$  are  not  free
biholomorphic equivalent.

Now,  using  Theorem \ref{rep-finite3},  Corollary \ref{iso-dom},
  and  Arias-Latr\'emoli\`ere
 characterization of the noncommutative disc algebra $\cA_n$ (see Theorem 4.7 from \cite{ArL}),
 we deduce   the following  classification result.

\begin{corollary}\label{biho-equi2} Let $g$ be  a positive regular free
holomorphic function with $q$  indeterminates. Then the
noncommutative domain ${\bf D}_g^1(\cH)$ is biholomorphic equivalent
to the unit ball $[B(\cH)^n]_1$ if and only if \ $q=n$ and $g=c_1
X_1+\cdots + c_n X_n$ for some $c_i>0$.
\end{corollary}
\begin{proof}
Let $p:=X_1+\cdots + X_n$ and recall  that  ${\bf
D}_p^1(\cH)=[B(\cH)^n]_1$. Due to Theorem \ref{rep-finite3} and
Corollary \ref{iso-dom}, ${\bf D}_g^1(\cH)$ is free biholomorphic
equivalent to the unit ball $[B(\cH)^n]_1$  if and only if $q=n$ and
the domain algebras $\cA_n $  and $\cA_n({\bf D}_g^l)$ are
completely isometrically isomorphic. Applying Arias-Latr\'emoli\`ere
result, the result follows. The converse is obvious by rescaling.
\end{proof}

 \begin{corollary}
 \label{auto-disk2}  Let $f$ and $g$ be positive regular free
holomorphic functions with $n$ and $q$ indeterminates, respectively.
Let  $\Psi:A({\bf D}_{f,\text{\rm rad}}^m)\to A({\bf D}_{g,\text{\rm
rad}}^l)$ be a unital  algebra homomorphism. Then $\Psi$ is a unital
completely isometric  isomorphism  having   completely contractive
hereditary extension
  if and only if
   $\Psi$  is  a  continuous homeomorphism     such that the $n$-tuples of operators
 $(\hat \Psi(W_1^{(f)}),\ldots, \hat \Psi(W_n^{(f)}))$ and
  $(\hat\Psi^{-1}(W_1^{(g)}),\ldots, \hat \Psi^{-1}(W_q^{(g)}))$  are in
  ${\bf D}_f^m(F^2(H_n))$ and ${\bf D}_g^l(F^2(H_q))$, respectively.
\end{corollary}
\begin{proof} The direct implication follows from the proof of Theorem \ref{auto-disk}.
 Conversely, assume
that $\Psi$ is  a  continuous homeomorphism  such
   that $(\hat \Psi(W_1^{(f)}),\ldots,
\hat \Psi(W_n^{(f)}))$  is  in ${\bf D}_f^m(F^2(H_n))$. Due to the
noncommutative von Neumann inequality  (see  \eqref{vn1}), we have
$$
\|[\hat\Psi(p_{ij}(W_1^{(f)},\ldots, W_n^{(f)}))]_{k\times k}\|=
\|[p_{ij}(\hat\Psi(W_1^{(f)}),\ldots,\hat\Psi( W_n^{(f)})]_{k\times k}\| \leq
\|[ p_{ij}(W_1^{(f)},\ldots, W_n^{(f)})]_{k\times k}\|
$$
for any operator matrix  $[p_{ij}(W_1^{(f)},\ldots,
W_n^{(f)})]_{k\times k}\in\cA_n({\bf D}_f^m)\otimes M_{k\times k}$
and $k\geq 1$. Since $\hat \Psi$ is continuous on $\cA_n({\bf
D}_f^m)$, which
is the norm closed algebra generated by $W_1^{(f)},\ldots,
W_n^{(f)}$ and the identity, we deduce that $\hat\Psi:\cA_n({\bf
D}_f^m )\to \cA_q({\bf D}_g^l )$ is a unital completely contractive
homomorphism. Similarly, if $(\hat\Psi^{-1}(W_1^{(g)}),\ldots, \hat
\Psi^{-1}(W_q^{(g)}))$ is in ${\bf D}_g^l(F^2(H_q))$,
   we can prove that $\hat \Psi^{-1}$ is a unital
 completely contractive homomorphism. Therefore, $\hat\Psi$ is a complete
 isometry. Moreover, using   the noncommutative
 von Neumann inequality \eqref{vn1}, we deduce that $\Psi$ has   completely
contractive hereditary extension, and  complete the proof.
\end{proof}

We denote by  $Aut({\bf D}_f^m):=Bih({\bf D}_f^m, {\bf D}_f^m)$
the  automorphism group of $ {\bf D}_f^m(\cH)$.  The group of
all unital completely isometric   automorphisms
 of the noncommutative domain algebra  $A({\bf D}_{f,\text{\rm
rad}}^m)$ is denoted
 by $Aut_{ci}(A({\bf D}_{f,\text{\rm rad}}^m))$.   We also use the notation $Aut_{ci}^*(A({\bf D}_{f,\text{\rm rad}}^m))$
 for the set of all automorphisms $\Psi\in Aut_{ci}(A({\bf D}_{f,\text{\rm rad}}^m))$
  having   completely contractive hereditary extension.

 \begin{corollary}
 \label{auto-disk3}  Let $f$  be  a positive regular free
holomorphic function with $n$ indeterminates, and let $m\geq 1$.
 Then the following statements are equivalent:
\begin{enumerate}
\item[(i)]
    $\Psi\in  Aut_{ci}^*(A({\bf D}_{f,\text{\rm rad}}^m))$;
\item[(ii)]
 there is   $\varphi \in Aut({\bf D}_f^m)$  such that
 $$
 \Psi(\chi)=\chi\circ \varphi,\quad \chi\in A({\bf D}_{f,\text{\rm
rad}}^m) .
 $$
 \end{enumerate}
 Consequently,  $Aut_{ci}^*(A({\bf D}_{f,\text{\rm rad}}^m))$ is a subgroup
  of $Aut_{ci}(A({\bf D}_{f,\text{\rm rad}}^m))$ and
 $$
 Aut_{ci}^*(A({\bf D}_{f,\text{\rm rad}}^m))\simeq Aut({\bf D}_f^m).$$
 In the particular case when $m=1$,  $$Aut_{ci}^*(A({\bf D}_{f,\text{\rm rad}}^1))=Aut_{ci}(A({\bf D}_{f,\text{\rm rad}}^1))\simeq Aut({\bf D}_f^1).$$
\end{corollary}
We mention that in the particular case when $m=1$ and $f=X_1+\cdots
+X_n$, Corollary \ref{auto-disk3} was obtained in
\cite{Po-automorphism} using different methods. In addition, we used
the theory of characteristic functions for row contractions (see
\cite{Po-charact})
   to determine the group
$Aut[B(\cH)^n]_1)$ of all free holomorphic automorphisms of
$[B(\cH)^n]_1$. It was shown that $Aut([B(\cH)^n]_1)\simeq
Aut(\BB_n)$, the Moebius group of the open unit ball
$\BB_n:=\{\lambda\in \CC^n:\ \|\lambda\|_2<1\}$.

          If  $\Psi:A({\bf D}_{f,\text{\rm
rad}}^m)\to A({\bf D}_{g,\text{\rm rad}}^l)$ is  a unital completely
isometric  isomorphism  having completely contractive hereditary
extension, then due to Theorem \ref{auto-disk},
$\Psi(\chi)=\chi\circ \varphi$, $\chi\in A({\bf D}_{f,\text{\rm
rad}}^m)$, for some $\varphi \in Bih( {\bf D}_g^l,{\bf D}_f^m)$. In
this case, we call $\varphi$ the symbol of $\Psi$. Using   Theorem
\ref{Cartan3}
 and Theorem \ref{auto-disk}, we  can deduce  the
following result.

\begin{theorem}\label{A-complete}  Let $f$ and $g$ be positive regular free
holomorphic functions with $n$ and $q$ indeterminates, respectively.
 A map $\Psi:A({\bf D}_{f,\text{\rm
rad}}^m)\to A({\bf D}_{g,\text{\rm rad}}^l)$ is  a unital completely
isometric  isomorphism  having completely contractive hereditary
extension and such that  its symbol $\varphi$ fixes the origin  if
and only if  $n=q$ and  $\Psi$ is given by
 $$
 \Psi(\chi)= \chi\circ \varphi,\qquad \chi\in A({\bf D}_{f,\text{\rm
rad}}^m),
 $$
 for some $\varphi \in Bih( {\bf D}_g^l,{\bf D}_f^m)$ of the form
\begin{equation*}
 \varphi(X_1,\ldots X_n)=[X_1,\ldots, X_n]U , \qquad (X_1,\ldots, X_n)\in
 {\bf D}_{g,\text{\rm
rad}}^l(\cH),
\end{equation*}
where    $U$  is an invertible operator on $\CC^n$  such that
$$
[W_1^{(g)},\ldots, W_n^{(g)}]U\in {\bf D}_f^m(F^2(H_n)) \quad \text{ and } \quad
[W_1^{(f)},\ldots, W_n^{(f)}]U^{-1}\in {\bf D}_g^l(F^2(H_n)).
$$
In  this case,  we have
$$
[\widehat \Psi(W_1^{(f)}),\ldots,\widehat \Psi(W_n^{(f)})]= \widetilde \varphi=[W_1^{(g)},\ldots, W_n^{(g)}]U .
$$
\end{theorem}

  We mention that, in the particular case when $m=l=1$, Arias and Latr\' emoli\` ere proved
  (see  Theorem 3.18 from \cite{ArL}) that if  there is an completely
  isometric isomorphism between two noncommutative domain algebras
   $\cA_n({\bf D}_f^1)$ and  $\cA_n({\bf D}_g^1)$ whose dual map fixes the origin,
   then  the algebras are related by  a linear  relation of  their generators.
Our Theorem \ref{A-complete} implies  and strengthens their result
and also provides a converse.

 \bigskip

\section{Isomorphisms of noncommutative Hardy algebras}

In this section we  characterize the   unitarily implemented
  isomorphisms of noncommutative   Hardy
algebras associated with noncommutative domains. Similar results are
deduced for noncommutative domain algebras.

Let $f$
and $g$ be positive regular free holomorphic functions with $n$ and
$q$ indeterminates, respectively,  and let $m,l\geq 1$.
   According to Theorem \ref{f-infty},  there is a unital
completely isometric isomorphism
$$
H^\infty({\bf D}_{f,\text{\rm rad}}^m)\ni \chi\mapsto \widetilde
\chi:=\text{\rm SOT-}\lim_{r\to 1} \chi(rW_1,\ldots, rW_n)\in
F_n^\infty({\bf D}_f^m),
$$
where the limit is in the strong operator   topology,  whose inverse
is the noncommutative Berezin  transform ${\bf B}$, i.e.,
$\chi(X)={\bf B}_X[\widetilde \chi]$ for  $X\in {\bf D}_{f,\text{\rm
rad}}^m(\cH)$. If $\Psi: H^\infty({\bf D}_{f,\text{\rm rad}}^m)\to
H^\infty({\bf D}_{g,\text{\rm rad}}^l)$ is a unital homomorphism, it
induces a unique unital homomorphism $\widehat \Psi:F_n^\infty({\bf
D}_f^m)\to F_q^\infty({\bf D}_g^l)$ such that the  diagram
\begin{equation*}
\begin{matrix}
F_n^\infty({\bf D}_f^m)& \mapright{\widehat\Psi}&
F_q^\infty({\bf D}_g^l)\\
 \mapdown{{\bf B}}& & \mapdown{{\bf B}}\\
H^\infty({\bf D}_{f,\text{\rm rad}}^m)&  \mapright{\Psi}&
H^\infty({\bf D}_{g,\text{\rm rad}}^l)
\end{matrix}
\end{equation*}
is commutative, i.e., $\Psi {\bf B}={\bf B}\widehat \Psi$, where
${\bf B}$ is the appropriate noncommutative Berezin transform on the
Hardy space $F^\infty_n({\bf D}_f^m)$ or  $F^\infty_q({\bf D}_g^l)$.
The homomorphisms $\Psi$ and $\widehat \Psi$ uniquely determine each
other by the formulas:
\begin{equation*}
\begin{split}
[\Psi (\chi)](X)&={\bf B}_X[\widehat \Psi(\widetilde \chi)], \qquad  \chi\in
H^\infty({\bf D}_{f,\text{\rm rad}}^m),\ X\in {\bf D}_{g,\text{\rm
rad}}^l(\cH), \quad \text{ and }\\
\widehat \Psi(\widetilde \chi)&=\widetilde{\Psi(\chi)}, \qquad \tilde \chi\in
F_n^\infty({\bf D}_f^m).
\end{split}
\end{equation*}
 If $\widehat \Psi$ is  is a SOT-continuous
homomorphism, we say that $\Psi$  has the same property.

A bijective map $F:{\bf D}_{f,\text{\rm rad}}^m(\cH)\to {\bf D}_{g,
\text{\rm rad}}^l(\cH)$ is called free biholomorphic function  if $F
$ and $F^{-1} $ are free holomorphic functions on ${\bf
D}_{f,\text{\rm rad}}^m(\cH)$ and ${\bf D}_{g,\text{\rm
rad}}^l(\cH)$, respectively. We denote by $Bih({\bf D}_{g,\text{\rm
rad}}^l, {\bf D}_{f,\text{\rm
rad}}^m)$ the set of all
 free biholomorphic functions  from  ${\bf D}_{g, \text{\rm
rad}}^l(\cH)$ onto ${\bf D}_{f, \text{\rm rad}}^m(\cH)$. Note the
difference between  $Bih({\bf D}_{g,\text{\rm rad}}^l, {\bf
D}_{f,\text{\rm rad}}^m)$  and $Bih({\bf D}_{g}^l, {\bf D}_{f}^m)$.

 \begin{theorem}
 \label{auto-hardy3}   Let $f$ and $g$ be positive regular free
holomorphic functions with $n$ and $q$ indeterminates, respectively.
If  $\psi\in Bih({\bf D}_{g,\text{\rm rad}}^l, {\bf D}_{f,\text{\rm
rad}}^m)$, then the
  map
  $$
 \Psi(\chi):=\chi\circ \psi,\qquad \chi\in H^\infty({\bf D}_{f,\text{\rm
rad}}^m),
 $$
 is a
   completely isometric isomorphism  from the  Hardy algebra
 $H^\infty({\bf D}_{f,\text{\rm
rad}}^m)$ onto $H^\infty({\bf D}_{g,\text{\rm
rad}}^l)$. If, in addition, the model boundary function $\tilde \psi$ is a pure $n$-tuple, then
$\Psi$ is sequentially  SOT-continuous.

In particular,  if $\psi\in  Bih({\bf D}_{g }^l, {\bf D}_{f, }^m)$ and $\psi(0)=0$,
then $\Psi$ is a
   completely isometric isomorphism and has the form mentioned in Theorem \ref{A-complete}.
\end{theorem}
\begin{proof} Using Theorem \ref{compo}, we deduce that $\chi\circ
\psi\in H^\infty({\bf D}_{g,\text{\rm rad}}^l)$ and $\|\chi\circ
\psi\|_\infty\leq \|\chi\|_\infty$. Moreover, passing to matrices,
we can  similarly  show that $\Psi$ is a completely contractive
homomorphism. In the same manner, one can prove that  the map
$\Gamma(\chi):=\chi\circ \psi^{-1}$ is a completely contractive
homomorphism from $H^\infty({\bf D}_{g,\text{\rm rad}}^l)$ to
$H^\infty({\bf D}_{f,\text{\rm rad}}^m)$.
 Let $\chi\in
H^\infty({\bf D}_{f,\text{\rm rad}}^m)$ have the representation
$\sum_{k=0}^\infty \sum_{|\alpha|=k} c_\alpha Z_\alpha $, and fix
$X:=(X_1,\ldots, X_n)\in {\bf D}_{f,\text{\rm rad}}^m(\cH)$. Since
$\psi^{-1}(X)\in {\bf D}_{g,\text{\rm rad}}^m(\cH)$ and $\psi({\bf
D}_{g,\text{\rm rad}}^l(\cH))\subset {\bf D}_{f,\text{\rm
rad}}^m(\cH)$ we   deduce that
$$
(\chi\circ \psi)(\psi^{-1}(X))=\sum_{k=0}^\infty \sum_{|\alpha|=k}
c_\alpha \psi_\alpha(\psi^{-1}(X)),
$$
where  the convergence is in the operator norm topology. Since
$\psi\in Bih({\bf D}_{g,\text{\rm
rad}}^l, {\bf D}_{f,\text{\rm
rad}}^m)$, we have
$\psi_i(\psi^{-1}(X))=X_i$, $i=1,\ldots, n$, and, therefore,
$\psi_\alpha(\psi^{-1}(X))=X_\alpha$ for $\alpha\in \FF_n^+$.
Consequently, $(\Gamma\circ \Psi)(\chi)=\chi$ for any $\chi\in H^\infty({\bf
D}_{f,\text{\rm rad}}^m)$. Similarly, one can prove that
$\Phi\circ \Gamma=id$. Summing up, we conclude that the map
$\Psi$ is  a
   completely isometric isomorphism  of  noncommutative  Hardy
   algebras.

If, in addition,  we assume that $\tilde \psi$ is a pure $n$-tuple,
then, according to  Theorem \ref{more-prop2}, we have
$$
\widehat \Psi(\widetilde \chi)=\widetilde{\Psi(\chi)}=\widetilde
{\chi\circ\psi}= {\cB}_{\widetilde \psi}[\widetilde \chi]={K_{{\widetilde
\psi}, f}^{(m)}}^* [\widetilde \chi\otimes
   I_{F^2(H_n)}]K_{{\widetilde \psi},f}^{(m)},
$$
where ${\cB}_{\widetilde \psi}$ is the noncommutative Berezin
transform at the pure $n$-tuple  $\widetilde \psi$ (see relation
\eqref{Be-transf}). Consequently, $\Psi$ is SOT-continuous on
bounded sets.

Now, we prove the last part of the theorem. According to Theorem
\ref{Cartan3}, if $\psi\in  Bih({\bf D}_{g }^l, {\bf D}_{f, }^m)$
and $\psi(0)=0$, then $n=q$ and $\psi$ has the form
\begin{equation*}
 \psi(X_1,\ldots, X_n)=[X_1,\ldots, X_n]U , \qquad (X_1,\ldots, X_n)\in{\bf D}_{g}^l(\cH),
\end{equation*}
where    $U$  is an invertible operator on $\CC^n$  such that
$$
[W_1^{(g)},\ldots, W_n^{(g)}]U\in {\bf D}_f^m(F^2(H_n)) \quad \text{ and } \quad
[W_1^{(f)},\ldots, W_n^{(f)}]U^{-1}\in {\bf D}_g^l(F^2(H_n)).
$$
  Hence,  it is obvious that we have
$\psi(rW_1^{(g)},\ldots, rW_n^{(g)})=r\psi(W_1^{(g)},\ldots,
W_n^{(g)})\in {\bf D}_{f, \text{\rm rad}}^m(F^2(H_n))$ and
$\psi^{-1}(rW_1^{(f)},\ldots,
rW_n^{(f)})=r\psi^{-1}(W_1^{(f)},\ldots, W_n^{(f)})\in {\bf D}_{g,
\text{\rm rad}}^l(F^2(H_n))$ for any $r\in[0,1)$. Due to Lemma
\ref{range}, we deduce that $\psi({\bf D}_{g,\text{\rm
rad}}^l)\subseteq {\bf D}_{f,\text{\rm rad}}^m$ and $\psi^{-1}({\bf
D}_{f,\text{\rm rad}}^m)\subseteq {\bf D}_{g,\text{\rm rad}}^l$.
Therefore, $\psi\in Bih({\bf D}_{g,\text{\rm rad}}^l, {\bf
D}_{f,\text{\rm rad}}^m)$. Applying the first part of this theorem,
we complete the proof.
\end{proof}

We introduce now the set $Bih({\bf D}_{g,\text{\rm pure}}^l,{\bf
D}_{f,\text{\rm pure}}^m)$  of all
 bijections
 $\varphi:{\bf D}_{g,\text{\rm
pure}}^l(\cH)\to {\bf D}_{f,\text{\rm pure}}^m(\cH)$ such that
$\varphi|_{{\bf D}_{g,\text{\rm rad}}^l(\cH)}$ and
$\varphi^{-1}|_{{\bf D}_{f,\text{\rm rad}}^m(\cH)}$ are  pure free
holomorphic functions, i.e., their model boundary functions are
pure,  and $\varphi$ and $\varphi^{-1}$ are their radial extensions
in the strong operator topology, respectively, i.e.,
$$
\varphi(X)=\text{\rm SOT-}\lim_{r\to 1} \varphi(rX), \qquad
  X\in {\bf D}_{g,\text{\rm pure}}^l(\cH),
$$
and
$$\varphi^{-1}(X)=\text{\rm SOT-}\lim_{r\to 1}
\varphi^{-1}(rX),\qquad X\in {\bf D}_{f,\text{\rm pure}}^m(\cH).
$$

\begin{theorem}\label{implement}  Let $f$ and $g$ be positive regular free
holomorphic functions with $n$ and $q$ indeterminates, respectively.
A map  $\Psi: H^\infty({\bf D}_{f,\text{\rm rad}}^m)\to
H^\infty({\bf D}_{g,\text{\rm rad}}^l)$ is  a unitarily implemented
isomorphism if and only if  it has the form
 $$
 \Psi(\chi):=\chi\circ \varphi,\qquad \chi\in H^\infty({\bf D}_{f,\text{\rm
rad}}^m),
 $$
for some
 $\varphi \in  Bih({\bf D}_{g,\text{\rm pure}}^l,{\bf D}_{f,\text{\rm pure}}^m)$
  with the property that $\widetilde \varphi$ is unitarily
  equivalent to  the universal model $(W_1^{(f)},\ldots, W_n^{(f)})$ associated with ${\bf D}_f^m$.
 In this case,
 $$\widehat
\Psi(\widetilde\chi)=\cB_{\widetilde \varphi}[\widetilde
\chi]:={K_{f,\widetilde \varphi}^{(m)}}^* (\widetilde \chi\otimes
I_{\cD_{f,m,\widetilde \varphi}})K_{f,\widetilde
\varphi}^{(m)},\qquad \widetilde \chi\in F_n^\infty({\bf D}_f^m),
 $$
 where   the noncommutative Berezin
kernel $K_{f,\widetilde \varphi}^{(m)}$   is a unitary operator and
$\text{\rm dim}\,\cD_{f,m,\widetilde \varphi}=1$.
\end{theorem}

\begin{proof}
Assume that  the induced map   $\widehat
 \Psi:F_n^\infty({\bf D}_f^m)\to F_q^\infty({\bf D}_g^l)$ is a
  unitarily implemented isomorphism.
  Then, there is a unitary operator $U\in B(F^2(H_q), F^2(H_n))$ such that
$\widehat \Psi (\widetilde\chi)=U^* \widetilde\chi U\in
F_q^\infty({\bf D}_g^l)$ for all $\widetilde\chi\in F_n^\infty({\bf
D}_f^m)$. Denote
 \begin{equation}
 \label{var}
 \widetilde\varphi_i:=\widehat \Psi(W_i^{(f)})\in F_q^\infty({\bf D}_g^l),
\qquad i=1,\ldots, n.
 \end{equation}
Since $\widetilde \varphi:=(\widetilde \varphi_1,\ldots, \widetilde
\varphi_n)$ is unitarily equivalent to  $(W_1^{(f)},\ldots,
W_n^{(f)})$, which is a pure $n$-tuple in the noncommutative domain
${\bf D}_f^m(F^2(H_n))$, so is the $n$-tuple
  $\widetilde \varphi$. Setting
 \begin{equation}
  \label{fib}
\varphi_i(X):={\cB}_X[\widetilde \varphi_i], \qquad X\in {\bf
D}_{q,\text{\rm pure}}^l(\cH),
 \end{equation}
 and due to Theorem \ref{f-infty} and Lemma \ref{pure},  the map
$\varphi(X):=(\varphi_1(X),\ldots,\varphi_n(X))$   is a  bounded
free holomorphic function on ${\bf D}_{g,\text{\rm rad}}^l(\cH)$ with values
 in ${\bf D}_{f,\text{\rm pure}}^m(\cH)$.
Moreover, due to the properties  of the noncommutative Berezin
transform (see \eqref{Be-transf}), we have $\varphi(X)=\text{\rm
SOT-}\lim_{r\to 1} \varphi(rX)$  for $  X\in {\bf D}_{g,\text{\rm
pure}}^l(\cH)$.

  Similarly,  we set
 \begin{equation}
 \label{xi2}\widetilde\xi_i:=\widehat \Psi^{-1}(W_i^{(g)})\in F_n^\infty({\bf D}_f^m),\qquad
 i=1,\ldots, q,
 \end{equation}
   and  using the fact that  $\widehat
 \Psi^{-1}:F_q^\infty({\bf D}_g^l)\to F_n^\infty({\bf D}_f^m)$ is given
 by $\widehat \Psi^{-1}(\widetilde \zeta)=U \widetilde \zeta U^*$, $\widetilde \zeta\in F_q^\infty({\bf D}_g^l)$, we deduce that
$\widetilde \xi:=(\widetilde \xi_1,\ldots, \widetilde \xi_q)$ is  a
pure $q$-tuple in
  ${\bf D}_g^l(F^2(H_n))$. Now, if we define
  \begin{equation}
  \label{xib}
  \xi_i(X):={\cB}_X[\widetilde \xi_i], \qquad X\in {\bf D}_{f,\text{\rm
pure}}^m(\cH),\ i=1,\ldots,q,
\end{equation}
 then, using again   Theorem \ref{f-infty} and Lemma \ref{pure},  we obtain that the map
$\xi(X):=(\xi_1(X),\ldots,\xi_n(X))$  is a  bounded free holomorphic
function on ${\bf D}_{f,\text{\rm rad}}^m(\cH)$ with values in ${\bf
D}_{g,\text{\rm pure}}^l(\cH)$ and has the property that
$\xi(X)=\text{\rm SOT-}\lim_{r\to 1} \xi(rX)$  for $ X\in {\bf
D}_{f,\text{\rm pure}}^m(\cH)$. Now, due to Theorem \ref{f-infty},
each $\widetilde \xi_i\in F_n^\infty({\bf D}_f^m)$, $i=1,\ldots, q$,
has a unique Fourier representation $\sum_{\alpha\in \FF_n^+}
c_\alpha W_\alpha^{(f)}$ such that
 \begin{equation}
 \label{xibe}
 \widetilde \xi_i=\text{\rm SOT-}\lim_{r\to 1} \sum_{k=0}^\infty
\sum_{|\alpha|=k} c_\alpha r^{|\alpha|} W_\alpha^{(f)},
\end{equation}
 where the
limit is in the strong operator topology. Hence, using the
continuity  of $\widehat \Psi$ in the strong operator topology, and
relations \eqref{var} and \eqref{xi2}, we obtain
\begin{equation*}
\begin{split}
W_i^{(g)}&=\widehat \Psi(\widetilde \xi_i)=\widehat \Psi\left(\text{\rm
SOT-}\lim_{r\to 1} \sum_{k=0}^\infty \sum_{|\alpha|=k} c_\alpha
r^{|\alpha|} W_\alpha^{(f)}
\right)\\
&=\text{\rm SOT-}\lim_{r\to 1} \sum_{k=0}^\infty \sum_{|\alpha|=k}
c_\alpha r^{|\alpha|} \widehat \Psi(W_\alpha^{(f)})= \text{\rm
SOT-}\lim_{r\to 1} \sum_{k=0}^\infty \sum_{|\alpha|=k} c_\alpha
r^{|\alpha|} \widetilde \varphi_\alpha,
\end{split}
\end{equation*}
for any $i=1,\ldots,q$. Consequently, applying   the noncommutative
Berezin transform  at $X\in {\bf D}_{g, \text{\rm pure}}^l(\cH)$ on
the Hardy algebra $F_q^\infty({\bf D}_g^l)$, using its continuity in
the strong operator topology on bounded sets and relations
\eqref{fib}, \eqref{xibe}, \eqref{xib},  we have
\begin{equation*}
\begin{split}
X_i&={\cB}_X[W_i^{(g)}]={\cB}_X\left[\text{\rm SOT-}\lim_{r\to
1}\sum_{k=0}^\infty \sum_{|\alpha|=k} c_\alpha
r^{|\alpha|} \widetilde \varphi_\alpha\right]\\
&=\text{\rm SOT-}\lim_{r\to 1} \sum_{k=0}^\infty \sum_{|\alpha|=k}
c_\alpha
r^{|\alpha|}  {\cB}_X[\widetilde\varphi_\alpha]
=\text{\rm SOT-}\lim_{r\to 1} \sum_{k=0}^\infty \sum_{|\alpha|=k}
c_\alpha
r^{|\alpha|}   \varphi_\alpha(X)\\
&=\text{\rm SOT-}\lim_{r\to1}
\cB_{\varphi(X)}\left[\sum_{k=0}^\infty \sum_{|\alpha|=k} c_\alpha
r^{|\alpha|} W_\alpha^{(f)}\right]\\
&=\cB_{\varphi(X)}[\widetilde \xi]=\xi_i(\varphi(X))
\end{split}
\end{equation*}
for each $i=1,\ldots, q$, and any $X\in {\bf D}_{g,\text{\rm
pure}}^l(\cH)$. Hence $(\xi\circ \varphi)(X)=X$ for $X\in {\bf
D}_{g, \text{\rm pure}}^l(\cH)$. Similarly, one can prove that
$(\varphi\circ \xi)(X)=X$ for  $X\in {\bf D}_{f,\text{\rm
pure}}^m(\cH)$. Consequently,
 $\varphi:{\bf D}_{g,\text{\rm pure}}^l(\cH)\to {\bf D}_{f,\text{\rm pure}}^m(\cH)$ is
 a bijection  such that
  $\varphi$ and $\varphi^{-1}:=\xi$ are free holomorphic functions on   the noncommutative domain
  ${\bf D}_{g,\text{\rm
rad}}^l(\cH)$ and ${\bf D}_{f,\text{\rm rad}}^m(\cH)$, respectively,
and  have the required properties. Therefore,  $\varphi \in  Bih({\bf
D}_{g,\text{\rm pure}}^l,{\bf D}_{f,\text{\rm pure}}^m)$.

Let   $\widetilde \chi\in F_n^\infty({\bf D}_f^m)$  have  the
  representation $\sum_{\alpha\in \FF_n^+} d_\alpha
W_\alpha^{(f)}$. Then
 $
 \widetilde \chi=\text{\rm SOT-}\lim\limits_{r\to 1} \sum\limits_{k=0}^\infty
\sum_{|\alpha|=k} d_\alpha r^{|\alpha|} W_\alpha^{(f)}. $
   Due to relation
\eqref{var} and the continuity of
$\widehat \Psi$ in the strong operator topology, we have
$$
\widehat \Psi(\widetilde \chi)=\text{\rm SOT-}\lim_{r\to 1}
\sum_{k=0}^\infty \sum_{|\alpha|=k} d_\alpha r^{|\alpha|} \widetilde
\varphi_\alpha.
$$
Consequently, applying   the noncommutative Berezin  transform at
$X\in {\bf D}_{g,\text{\rm pure}}^l(\cH)$,  using its SOT-continuity
on bounded sets, and relation \eqref{fib}, we get
\begin{equation*}
\begin{split}
[\Psi(\chi)](X)&={\cB}_X[\widehat \Psi(\widetilde \chi)] =\text{\rm
SOT-}\lim_{r\to 1} \sum_{k=0}^\infty \sum_{|\alpha|=k} d_\alpha
r^{|\alpha|} {\cB}_X[\widetilde \varphi_\alpha]\\
 &=\text{\rm SOT-}\lim_{r\to 1}
\sum_{k=0}^\infty \sum_{|\alpha|=k} d_\alpha r^{|\alpha|}
\varphi_\alpha(X)\\
&=\text{\rm SOT-}\lim_{r\to1}\cB_{\varphi(X)}\left[\sum_{k=0}^\infty
\sum_{|\alpha|=k} d_\alpha
r^{|\alpha|} W_\alpha^{(f)}\right]\\
&=\cB_{\varphi(X)}[\widetilde \chi]=(\chi\circ\varphi)(X)
\end{split}
\end{equation*}
for any $X\in {\bf D}_{g,\text{\rm pure}}^l(\cH)$, which proves the
direct implication of the theorem.

Conversely, assume that  $\varphi\in
  Bih({\bf D}_{g,\text{\rm pure}}^l,{\bf D}_{f,\text{\rm pure}}^m)$
  and $\widetilde \varphi$ is unitarily
  equivalent to  the universal model $W^{(f)}:=(W_1^{(f)},\ldots,
  W_n^{(f)})$. Define
 $
 \Psi(\chi):=\chi\circ \varphi ,\quad \chi\in H^\infty( {\bf D}_{f,\text{\rm rad}}^m)$.
 First note that, due to Theorem \ref{more-prop2},
  $\chi\circ \varphi\in H^\infty( {\bf D}_{g,\text{\rm rad}}^l)$
  for all $\chi\in H^\infty( {\bf D}_{f,\text{\rm rad}}^m)$, and
   $\Psi $ is a
well-defined completely contractive homomorphism.
 A similar result can be obtained for  the map
$\Lambda(\zeta):=\zeta\circ\varphi^{-1}$, $\zeta\in H^\infty( {\bf
D}_{g, \text{\rm rad}}^l)$. Now, for any $\chi\in H^\infty( {\bf
D}_{f, \text{\rm rad}}^m)$ and $X\in {\bf D}_{f,\text{\rm
pure}}^m(\cH)$, we have $( \Lambda\circ\Psi)(\chi)=(\chi\circ
\varphi)\circ \varphi^{-1}$ and
\begin{equation*}
\begin{split}
[(\chi\circ \varphi)\circ \varphi^{-1}](X)&=(\chi\circ
\varphi)(\varphi^{-1}(X))=\text{\rm SOT-}\lim_{r\to 1}
\chi(r\varphi_1(\varphi^{-1}(X)),\ldots, r\varphi_n(\varphi^{-1}(X)))\\
&=\text{\rm SOT-}\lim_{r\to1} \chi(rX_1,\ldots, rX_n)=\chi(X_1,\ldots,
X_n).
\end{split}
\end{equation*}
Hence, $\Lambda\circ \Psi=id$. Similarly, we can show that
$\Psi\circ \Lambda=id$  and deduce that   $\Psi$ is a completely
isometric isomorphism. Due to Theorem \ref{more-prop2},  we have
\begin{equation}
\label{ultima}
 \widehat \Psi(\widetilde \chi)=\widetilde{\chi\circ
\varphi}=\cB_{\widetilde\varphi}[\widetilde \chi]={K_{f,\widetilde
\varphi}^{(m)}}^* (\widetilde \chi\otimes I_{\cD_{f,m,\widetilde
\varphi}})K_{f,\widetilde \varphi}^{(m)},\qquad \widetilde \chi\in
F_n^\infty({\bf D}_f^m).
 \end{equation}
 Now, we  show that the Berezin kernel
$K_{f,\widetilde \varphi}^{(m)}$ is a unitary operator.
 Recall that
the   noncommutative Berezin kernel \ $K_{f,W^{(f)}}^{(m)}:F^2(H_n)
\to F^2(H_n)\otimes \cD_{f,m,W^{(f)}} $  is defined
  by
 \begin{equation*}
 K_{f,W^{(f)}}^{(m)}\ell=\sum_{\alpha\in \FF_n^+} \sqrt{b^{(m)}_\alpha}
e_\alpha\otimes \Delta_{f,m,W^{(f)}} {W^{(f)}_\alpha}^* \ell,\qquad
\ell\in F^2(H_n),
\end{equation*}
where
$\cD_{f,m,W^{(f)}}:=\overline{\Delta_{f,m,W^{(f)}}(F^2(H_n))}$.
Since
 $\Delta^2_{f,m,W^{(f)}}=\left(id-\Phi_{f,W^{(f)}}\right)^{m}(I)=P_\CC$ (see Section 1), we deduce
  that $\text{\rm dim}\,\cD_{f,m,W^{(f)}}=1$, so that we can identify the space
   $\cD_{f,m,W^{(f)}}$ with $\CC$. Under this identification and using   relation
   \eqref{WbWb}, simple computations reveal that
    $K_{f,W^{(f)}}^{(m)}=I_{F^2(H_n)}$.
Since $\widetilde \varphi$ is unitarily
  equivalent to  the universal model $W^{(f)}:=(W_1^{(f)},\ldots,
  W_n^{(f)})$,  one can easily see that the Berezin kernel
$K_{f,\widetilde \varphi}^{(m)}$ is a unitary operator. Therefore,
due to relation \eqref{ultima},  $\widehat \Psi$ is a unitarily
implemented isomorphism.
  The proof is
complete.
\end{proof}
We remark that  when $m=1$,  we can  use the functional model for
pure $n$-tuples of operators in ${\bf D}_f^1(\cH)$ in terms of
characteristic functions (see Theorem 3.26 from \cite{Po-domains}),
to deduce that $\widetilde \varphi$ is unitarily
  equivalent to  the universal model $W^{(f)}:=(W_1^{(f)},\ldots,
  W_n^{(f)})$ if and only if  $\widetilde\varphi$ is a pure $n$-tuple, i.e., $\text{\rm
SOT-}\lim_{k\to \infty}\Phi_{f,\widetilde \varphi}^k(I)=0$, with
  $\text{\rank}\,[I-\Phi_{f,\widetilde \varphi}(I)]=1$ and the
  characteristic function $\Theta_{\widetilde\varphi}=0$.

The group of all  unitarily implemented   automorphisms
 of the   Hardy  algebra  $H^\infty({\bf D}_{f,\text{\rm
rad}}^m)$ is denoted
 by $Aut_{u}(H^\infty({\bf D}_{f,\text{\rm rad}}^m))$. Let
  $Aut_w({\bf D}_{f,\text{\rm pure}}^m)$  be the set of all
$\varphi \in Bih({\bf D}_{f,\text{\rm pure}}^m,{\bf D}_{f,\text{\rm
pure}}^m)$ such that $\widetilde \varphi$ is unitarily
  equivalent to  the universal model $W^{(f)}:=(W_1^{(f)},\ldots,
  W_n^{(f)})$. Using Theorem \ref{f-infty} and some  results
  concerning the composition of pure free holomorphic functions (see
  Lemma \ref{pure} and Theorem \ref{more-prop2}), one can easily show
  that $Aut_w({\bf D}_{f,\text{\rm pure}}^m))$ is a group with
  respect to the composition.
As a consequence of Theorem \ref{implement}, we
  deduce that
  $$
  Aut_{u}(H^\infty({\bf D}_{f,\text{\rm rad}}^m))\simeq Aut_w({\bf D}_{f,\text{\rm
  pure}}^m).
  $$
A closer look at Theorem \ref{implement} reveals that one can easily
obtain a version of it for noncommutative domain algebras. In this
case, $\varphi \in Bih_w( {\bf D}_g^l,{\bf D}_f^m)$, the set of all
$\psi\in  Bih( {\bf D}_g^l,{\bf D}_f^m)$ such that  $\widetilde
\psi$ is unitarily
  equivalent to  the universal model $W^{(f)}:=(W_1^{(f)},\ldots,
  W_n^{(f)})$. The proof is basically the same but uses some parts
  from the proof of Theorem \ref{auto-disk}. Setting  $Aut_w({\bf
  D}_f^m)):=Bih_w( {\bf D}_f^m,{\bf D}_f^m)$, we can deduce that
$$
  Aut_{u}(A({\bf D}_{f,\text{\rm rad}}^m))\simeq Aut_w({\bf D}_{f}^m).
  $$
Note that due to our remark above,  when $m=1$,    the group
$Aut_w({\bf D}_{f}^1)$ consists of all free biholomorphic functions
$\varphi\in Aut({\bf D}_{f}^1)$ with the property that the model
boundary function $\widetilde\varphi$ is a pure $n$-tuple  with
  $\text{\rank}\,[I-\Phi_{f,\widetilde \varphi}(I)]=1$ and the
  characteristic function $\Theta_{\widetilde\varphi}=0$.

We mention that in the particular case   when $m=1$ and
$f=X_1+\cdots +X_n$ we have a more precise result. Davidson and
Pitts (\cite{DP2}) used Voiculescu's (\cite{Vo1}) group of
automorphisms of the Cuntz-Toeplitz algebra $C^*(S_1,\ldots, S_n)$
(\cite{Cu}),
 to show (among other things) that the subgroup $Aut_u(F_n^\infty)$ of  unitarily implemented
automorphisms of the noncommutative analytic Toeplitz algebra
$F_n^\infty$  is isomorphic with $Aut(\BB_n)$. In
\cite{Po-automorphism},
 we obtained a new proof of their result, using
 noncommutative Poisson transforms \cite{Po-poisson}, and
showed that   the unitarily implemented automorphisms of the
noncommutative  disc algebra $\cA_n$   and the  algebra
$F_n^\infty$, respectively, are determined by the free holomorphic
automorphisms of $[B(\cH)^n]_1$, via the noncommutative Poisson
transform.  According to \cite{DP2} and \cite{Po-automorphism}, we
have
$$
 Aut([B(\cH)^n]_1)\simeq Aut(\BB_n)\simeq  Aut_u(\cA_n)\simeq Aut_u(F_n^\infty).
 $$

We remark  here that the conformal automorphisms of $\BB_n$ also
occur in the   work  of Muhly and Solel  (\cite{MuSo3})
 concerning the automorphisms of Hardy algebras associated with
 $W^*$-correspondence over von Neumann algebras (\cite{MuSo1},
 \cite{MuSo2}).

       %

      \end{document}